\documentclass{amsart}

\usepackage[margin=1.1in]{geometry}
\usepackage{comment}
\usepackage{mathrsfs}
\usepackage{amsmath}
\usepackage{amsthm}
\usepackage{amsfonts}
\usepackage{amssymb}
\usepackage{enumitem}
\usepackage{graphicx}
\usepackage{overpic}
\usepackage{cancel}
\graphicspath{ {images/} }
\usepackage{url}
\usepackage{hyperref}
\hypersetup{
    colorlinks=true,
    linkcolor=blue,
}
\usepackage[dvipsnames]{xcolor}

\usepackage{tikz}

\usetikzlibrary{decorations.markings}
\tikzset{
  set arrow inside/.code={\pgfqkeys{/tikz/arrow inside}{#1}},
  set arrow inside={end/.initial=>, opt/.initial=},
  /pgf/decoration/Mark/.style={
    mark/.expanded=at position #1 with
    {
      \noexpand\arrow[\pgfkeysvalueof{/tikz/arrow inside/opt}]{\pgfkeysvalueof{/tikz/arrow inside/end}}
    }
  },
  arrow inside/.style 2 args={
    set arrow inside={#1},
    postaction={
      decorate,decoration={
        markings,Mark/.list={#2}
      }
    }
  },
}
\newsavebox{\largestimage}
\usetikzlibrary{shapes}
\usepackage{subcaption}

\newtheorem{theorem}{Theorem}[section]
\newtheorem{proposition}[theorem]{Proposition}
\newtheorem{lemma}[theorem]{Lemma}
\newtheorem{corollary}[theorem]{Corollary}

\theoremstyle{definition}
\newtheorem{definition}[theorem]{Definition}
\newtheorem{example}[theorem]{Example}

\newtheorem{question}[theorem]{Question}
\newtheorem{assumption}[theorem]{Assumption}

\theoremstyle{remark}
\newtheorem{remark}[theorem]{Remark}

\numberwithin{equation}{section}

\newcommand{\norm}[1]{\left\lVert#1\right\rVert}
\newcommand{\set}[2]{\left\{\,#1~:~#2\,\right\}}

\newcommand{\R}{\mathbb{R}}

\newcommand{\D}{\mathbb{D}}

\newcommand\rotreflPi{\rotatebox[origin=c]{180}{\reflectbox{$\Pi$}}}
\DeclareMathOperator{\skel}{Skel}

\usepackage{palatino}

\begin{document}

\title{Torus bundle Liouville domains are stably Weinstein}
\author{Joseph Breen and Austin Christian}

\begin{abstract}
We develop explicit local operations that may be applied to Liouville domains, with the goal of simplifying the dynamics of the Liouville vector field.  These local operations, which are Liouville homotopies, are inspired by the techniques used by Honda and Huang in \cite{honda2019convex} to show that convex hypersurfaces are $C^0$-generic in contact manifolds.  As an application, we use our operations to show that certain Liouville-but-not-Weinstein domains constructed by Huang in \cite{huang2019dynamical} are stably Weinstein.
\end{abstract}

\maketitle

\tableofcontents

\section{Introduction}\label{section:introduction}
\subsection{Liouville-but-not-Weinstein domains}
In this paper we will concern ourselves with compact, exact symplectic manifolds-with-boundary $(W,\omega)$.  A \emph{Liouville structure} on such a manifold is a choice of symplectic potential $\lambda$ with the property that $X_\lambda$ points transversely out of $\partial W$, where $X_\lambda$ is the vector field on $W$ defined by $\iota_{X_\lambda}\omega=\lambda$.  We call $X_\lambda$ the \emph{Liouville vector field} associated to $\lambda$, and refer to the pair $(W,\lambda)$ as a \emph{Liouville domain}.

Throughout, we will consider Liouville structures up to \emph{Liouville homotopy}.  For Liouville domains, a Liouville homotopy is simply a smooth family $\lambda_t$ of Liouville forms.  For Liouville \emph{manifolds} --- a noncompact analogue of Liouville domains --- the notion of Liouville homotopy is a bit more subtle, as is the notion of Liouville structure itself.

As with any geometric structure on a manifold, it is natural to seek a handlebody calculus which is compatible with Liouville structures.  Weinstein showed in \cite{weinstein1991contact} that such a calculus can be obtained for a restricted class of Liouville domains.  We will call a Liouville structure a \emph{Weinstein structure} if there exists a Morse function $\phi\colon W\to\mathbb{R}$ which is locally constant on $\partial W$ and for which $X_{\lambda}$ is gradient-like.  That is, there is a constant $c>0$ and some Riemannian metric on $W$ for which
\[
d\phi(X_\lambda) \geq c(\|X_{\lambda}\|^2+\|d\phi\|^2).
\]
Weinstein domains may be built up using \emph{Weinstein handles}, and thus their study has a distinctly topological flavor.

With these two definitions in hand, the obvious question is then whether or not a given Liouville structure is in fact a Weinstein structure, at least up to Liouville homotopy.  In general, the answer is no, and there is a well-known topological obstruction which a Liouville domain must avoid if it is to admit a Weinstein structure.

\begin{proposition}\label{prop:weinstein-topological-condition}
Let $(W,\lambda)$ be a Weinstein domain of dimension $2n$, and let $\phi$ be a Morse function for which $X_\lambda$ is gradient-like.  Then $\phi$ admits no critical points of index greater than $n$.
\end{proposition}

This means that a Weinstein domain has the homotopy type of a half-dimensional CW-complex.  In particular, as is the case for compact complex manifolds with pseudo-convex boundary, Weinstein domains of dimension at least $4$ must have connected boundaries.  The first examples of Liouville-but-not-Weinstein domains were given by McDuff in \cite{mcduff1991symplectic}, and further examples of Liouville-but-not-Weinstein domains have been constructed by Geiges \cite{geiges1994symplectic}, \cite{geiges1995examples}, Mitsumatsu \cite{mitsumatsu1995anosov}, Massot-Niederkr\"{u}ger-Wendl \cite{massot2013weak}, and Huang \cite{huang2019dynamical}, with such examples now existing in all even dimensions.

To the best of the authors' knowledge, all existing constructions of Liouville-but-not-Weinstein domains rule out the existence of a Weinstein structure by appealing to Proposition~\ref{prop:weinstein-topological-condition}.  According to the \emph{Weinstein existence theorem} (c.f. \cite[Theorem 13.1]{cieliebak2012stein}), a Liouville domain which has the homotopy type of a half-dimensional CW-complex must admit \emph{some} Weinstein structure, and thus we are left with the following question.

\begin{question}\label{question:homotopy-type}
Let $(W,\lambda)$ be a Liouville domain which has the homotopy type of a half-dimensional CW-complex.  Must $\lambda$ admit a Liouville homotopy to a Weinstein structure?
\end{question}

Wrapped up in this question are very delicate dynamical considerations.  In the absence of a Lyapunov function, there is very limited control over  $X_\lambda$ and thus the problem of homotoping $X_\lambda$ into a vector field with severely restricted dynamical properties --- all while maintaining the vector field's status as Liouville --- seems to be enormously difficult.  The overarching goal of this paper is to better understand the extent to which wild Liouville vector fields may be tamed, provided their underlying domains have sufficiently simple topology.

\subsection{Stabilizations and skeleta}
One way to achieve the topological criterion of Proposition~\ref{prop:weinstein-topological-condition} is to increase the dimension of our domain.  The \emph{stabilization} of a Liouville domain $(W,\lambda)$ is given by the product
\[
(W,\lambda) \times (r_0\,\mathbb{D}^2,\lambda_{\mathrm{stab}}) = (W\times r_0\,\mathbb{D}^2,\lambda + \lambda_{\mathrm{stab}}),
\]
for some $r_0>0$, where $\lambda_{\mathrm{stab}}$ is the standard Liouville form on $r_0\,\mathbb{D}^2$, defined by
\[
\lambda_{\mathrm{stab}} = \tfrac{1}{2}(p\,dq-q\,dp).
\]
The Liouville vector field $X_{\lambda_{\mathrm{stab}}} = \frac{1}{2}p\, \partial_p + \frac{1}{2}q\, \partial_q$ expands radially from the origin in $r_0\,\mathbb{D}^2$, and the Liouville vector field of the stabilization is given by $X_\lambda+X_{\lambda_{\mathrm{stab}}}$.  It follows that stabilization does not fundamentally change the dynamics of a Liouville domain.

To make this last statement more precise, let us recall the \emph{skeleton} of a Liouville domain, defined by
\[
\mathrm{Skel}(W,\lambda) := \bigcap_{t<0} \psi^t(W),
\]
where $\psi^t\colon W\to W$ is the time-$t$ flow of the Liouville vector field $X_\lambda$.  The interesting dynamics of $(W,\lambda)$ are encoded on $\mathrm{Skel}(W,\lambda)$, and it is clear that the skeleton of the stabilization is given by $\mathrm{Skel}(W,\lambda)\times\{(0,0)\}$.

We now arrive at a weaker version of Question~\ref{question:homotopy-type}.

\begin{question}\label{question:stably-weinstein}
Is every Liouville domain \emph{stably Weinstein}?
\end{question}

While allowing stabilizations weakens the hypothesis of Question~\ref{question:homotopy-type}, a clear answer to Question~\ref{question:stably-weinstein} would still shed a great deal of light on the difference between Liouville and Weinstein dynamics.  A Liouville domain which is not stably Weinstein would necessarily feature dynamics far more complicated than any witnessed on a Weinstein domain.

We point out that in the case of Liouville manifolds, an affirmative answer to Question~\ref{question:stably-weinstein} follows from work of Eliashberg-Gromov \cite{eliashberg1991convex}.  In fact, Eliashberg-Ogawa-Yoshiyasu showed in \cite{eliashberg2021stabilized} that if a Liouville manifold of dimension $2n$ has Morse type at most $n+1$, then a single stabilization suffices.  However, the relevant Liouville homotopies are not known to be compactly supported, and thus Question~\ref{question:stably-weinstein} remains. 

\subsection{Hands-on homotopies}
Our strategy for taming wild Liouville dynamics is to deal with the structures directly.  Inspired by the techniques used by Honda and Huang in \cite{honda2019convex} to study hypersurfaces in contact manifolds, we construct explicit Liouville homotopies which can be installed as local operations.  These local operations are meant to simplify the dynamics of our Liouville domain, with the goal of producing a Weinstein domain following their installation.

\begin{figure}
\centering
\begin{tikzpicture}[scale=0.75]
\begin{scope}[rotate=270, transform shape]
\draw (-3,-3) -- (3,-3) -- (3,3) -- (-3,3) -- (-3,-3);

\draw[domain=-3:3,variable=\x] plot ({0},{\x}) [arrow inside={end=latex}{0.5}];
\draw[domain=-3:3,variable=\x] plot ({1},{\x}) [arrow inside={end=latex}{0.5}];
\draw[domain=-3:3,variable=\x] plot ({2},{\x}) [arrow inside={end=latex}{0.5}];
\draw[domain=-3:3,variable=\x] plot ({-1},{\x}) [arrow inside={end=latex}{0.5}];
\draw[domain=-3:3,variable=\x] plot ({-2},{\x}) [arrow inside={end=latex}{0.5}];
\draw[domain=-3:3,variable=\x] plot ({2.75},{\x}) [arrow inside={end=latex}{0.5}];
\draw[domain=-3:3,variable=\x] plot ({-2.75},{\x}) [arrow inside={end=latex}{0.5}];
\end{scope}

\begin{scope}[xshift=11cm, rotate=270, transform shape]
\draw (-3,-3) -- (3,-3) -- (3,3) -- (-3,3) -- (-3,-3);

\draw[domain=-3:-1.5,variable=\x] plot ({0},{\x}) [arrow inside={end=latex}{0.5}];
\draw[domain=-1.5:1.5,variable=\x] plot ({0},{-\x}) [arrow inside={end=latex}{0.5}];
\draw[domain=1.5:3,variable=\x] plot ({0},{\x}) [arrow inside={end=latex}{0.5}];

\draw[domain=0:1.5,variable=\x] plot ({\x},{2*\x*\x-1.5}) [arrow inside={end=latex}{0.5}];
\draw[domain=0:1.5,variable=\x] plot ({-\x},{2*\x*\x-1.5}) [arrow inside={end=latex}{0.5}];
\draw[domain=0:0.75,variable=\x] plot ({\x},{(8/3)*\x*\x+1.5}) [arrow inside={end=latex}{0.5}];
\draw[domain=0:0.75,variable=\x] plot ({-\x},{(8/3)*\x*\x+1.5}) [arrow inside={end=latex}{0.5}];

\draw[domain=1.5:2.25,variable=\x] plot ({\x},{tan(deg((\x-1.875)*1.24905/0.375))}) [arrow inside={end=latex}{0.5}];
\draw[domain=1.5:2.25,variable=\x] plot ({-\x},{tan(deg((\x-1.875)*1.24905/0.375))}) [arrow inside={end=latex}{0.5}];
\draw[domain=-3:3,variable=\x] plot ({2.75},{\x}) [arrow inside={end=latex}{0.5}];
\draw[domain=-3:3,variable=\x] plot ({-2.75},{\x}) [arrow inside={end=latex}{0.5}];

\fill (0,-1.5) circle[radius=2pt];
\fill (0,1.5) circle[radius=2pt];
\end{scope}
\end{tikzpicture}
\caption{A box fold is installed in 2D by identifying a region as on the left and replacing it (via a Liouville homotopy) with a region as on the right.}
\label{fig:2d-box-fold}
\end{figure}

As a first example, let us consider the \emph{two-dimensional box fold}, to be carefully defined in Sections~\ref{section:box_folds_pl} and~\ref{section:box_folds_smooth}.  To install this operation on a two-dimensional Liouville domain $(W,\lambda)$, we first identify a region $U\subset W$ of the form
\[
(U,\lambda|_U) \cong ([0,s_0]\times[0,t_0],e^s\,dt).
\]
On $U$, the Liouville vector field $X_\lambda$ may be identified with $\partial_s$, meaning that flowlines of $X_\lambda$ simply pass through $U$ unperturbed.  With the goal of interrupting chaotic behavior, the box fold uses a Liouville homotopy to replace $X_\lambda$ on $U$ with a vector field which has two critical points --- one of index 0 and another of index 1 --- and is topologically equivalent to the vector field depicted in Figure~\ref{fig:2d-box-fold}.

The box fold will be far from sufficient as a general purpose operation on Liouville domains, but we can use it to demonstrate the spirit of our techniques.  Consider a Liouville structure on $S^1\times[-1,1]$ whose Liouville vector field behaves as seen in Figure~\ref{fig:liouville-cylinder}.  Namely, $S^1\times\{0\}$ is a closed orbit of the Liouville vector field, repelling other orbits.  The presence of a closed orbit precludes the existence of a Lyapunov function for our Liouville vector field, and thus this structure is not Weinstein.  After installing a box fold on a region which intersects the skeleton $S^1\times\{0\}$, we obtain a Liouville domain of the type depicted in Figure~\ref{fig:weinstein-cylinder}.  By interrupting the backward flow of our Liouville vector field\footnote{In this toy case, the Liouville vector field is admittedly not so much chaotic as simply \emph{not gradient-like}.} we are able to produce a Weinstein domain.

\begin{figure}
  \centering
  \savebox{\largestimage}{\begin{tikzpicture}[scale=0.75]
\draw (-3,-3) -- (3,-3) -- (3,3) -- (-3,3) -- (-3,-3);

\draw[thick,domain=-3:3,smooth,variable=\x,blue] plot ({\x},  {0}) [arrow inside={end=latex}{0.5}];

\draw[domain=-3:3,variable=\x,black] plot ({\x},  {0.5*exp(\x/10)}) [arrow inside={end=latex}{0.5}];
\draw[domain=-3:3,variable=\x,black] plot ({\x},  {0.5*exp((\x+6)/10)}) [arrow inside={end=latex}{0.5}];
\draw[domain=-3:3,variable=\x,black] plot ({\x},  {0.5*exp((\x+12)/10)}) [arrow inside={end=latex}{0.5}];
\draw[domain=-3:-0.08,variable=\x,black] plot ({\x},  {0.5*exp((\x+18)/10)}) [arrow inside={end=latex}{0.5}];

\draw[domain=-3:3,variable=\x,black] plot ({\x},  {-0.5*exp(\x/10)}) [arrow inside={end=latex}{0.5}];
\draw[domain=-3:3,variable=\x,black] plot ({\x},  {-0.5*exp((\x+6)/10)}) [arrow inside={end=latex}{0.5}];
\draw[domain=-3:3,variable=\x,black] plot ({\x},  {-0.5*exp((\x+12)/10)}) [arrow inside={end=latex}{0.5}];
\draw[domain=-3:-0.08,variable=\x,black] plot ({\x},  {-0.5*exp((\x+18)/10)}) [arrow inside={end=latex}{0.5}];
\end{tikzpicture}}%
  \begin{subfigure}[b]{0.48\textwidth}
    \centering
    \usebox{\largestimage}
    \caption{This Liouville structure is not Weinstein, because the vector field is not gradient-like.}
    \label{fig:liouville-cylinder}
  \end{subfigure}
  \quad
  \begin{subfigure}[b]{0.48\textwidth}
    \centering
    \raisebox{\dimexpr.5\ht\largestimage-.5\height}{%
      \begin{tikzpicture}[scale=0.75]
\draw (-3,-3) -- (3,-3) -- (3,3) -- (-3,3) -- (-3,-3);


\draw[thick,blue] plot [smooth,tension=1]
        coordinates {(-3,0) (-1.5,0)}
        [arrow inside={end=latex}{0.5}];
\draw[thick,blue] plot [smooth,tension=1]
        coordinates {(1.5,0) (-1.5,0)}
        [arrow inside={end=latex}{0.5}];
\draw[thick,blue] plot [smooth,tension=1]
        coordinates {(1.5,0) (3,0)}
        [arrow inside={end=latex}{0.5}];

\draw plot [smooth,tension=1]
        coordinates {(-1.5,0) (-1.5,3)}
        [arrow inside={end=latex}{0.5}];
\draw plot [smooth,tension=1]
        coordinates {(-1.5,0) (-1.5,-3)}
        [arrow inside={end=latex}{0.5}];
\draw plot [smooth,tension=1]
        coordinates {(1.5,0) (1.5,3)}
        [arrow inside={end=latex}{0.5}];
\draw plot [smooth,tension=1]
        coordinates {(1.5,0) (1.5,-3)}
        [arrow inside={end=latex}{0.5}];
        
\draw[domain=-0.885:1.5,variable=\x,black] plot ({\x},  {tan(deg((\x-1.5)*pi/6))}) [arrow inside={end=latex reversed}{0.5}];
\draw[domain=-0.885:1.5,variable=\x,black] plot ({\x},  {-tan(deg((\x-1.5)*pi/6))}) [arrow inside={end=latex reversed}{0.5}];
\draw[domain=1.5:3,variable=\x,black] plot ({\x},  {tan(deg((\x-1.5)*pi/6))}) [arrow inside={end=latex}{0.75}];\draw[domain=1.5:3,variable=\x,black] plot ({\x},  {-tan(deg((\x-1.5)*pi/6))}) [arrow inside={end=latex}{0.75}];
\draw[domain=-3:-2.12,variable=\x,black] plot ({\x},  {tan(deg((\x+4.5)*pi/6))}) [arrow inside={end=latex}{0.5}];
\draw[domain=-3:-2.12,variable=\x,black] plot ({\x},  {-tan(deg((\x+4.5)*pi/6))}) [arrow inside={end=latex}{0.5}];

\fill[blue] (-1.5,0) circle[radius=2pt];
\fill[blue] (1.5,0) circle[radius=2pt];
\end{tikzpicture}}
    \caption{The box fold introduces critical points which make the vector field gradient-like.}
    \label{fig:weinstein-cylinder}
  \end{subfigure}
\caption{A Liouville structure on $S^1\times[-1,1]$, before and after the installation of a box fold.  In both figures, the vertical edges are identified, and the skeleton is $S^1\times\{0\}$.}
\label{fig:cylinder-toy}
\end{figure}

While the box fold has the helpful feature of interrupting flowlines, it comes at the cost of a drastic holonomy for uninterrupted flowlines.  That is, after the box fold is installed, there will be flowlines which pass through the region $U$ with very different $t$-values along $\{s=0\}\times[0,t_0]$ and $\{s=s_0\}\times[0,t_0]$.  This local perturbation can then have unpredictable effects on the global dynamics of our Liouville domain, making our increased local understanding substantially less useful.

Consequently, we see that there are two desired features in any local operation: that it traps Liouville flowlines in backward time, and that its effect on untrapped flowlines is controllable.  In the analysis we will see that these desires are often at odds with one another, and that our control over the holonomy of these operations is hampered by the requirement that they be Liouville homotopies.  However, we present in this paper an important local operation which can be installed in a stabilized setting.

\subsection{The blocking apparatus}
We now introduce the key local operation developed in this paper, which we call the \emph{blocking apparatus}.  As with the box fold, the blocking apparatus is installed on a region $U\subset (W,\lambda)$ in a Liouville domain of dimension $2n$.  We require that this region have the form
\[
(U,\lambda|_U) \cong ([0,s_0]\times[0,t_0]\times W_0\times r_0\,\mathbb{D}^2,e^s(dt + \lambda_0 + \lambda_{\mathrm{stab}})),
\]
where $(W_0,\lambda_0)$ is some Weinstein domain of dimension $2n-4$ and $r_0>0$ is some positive constant.  In words, $(U,\lambda|_U)$ is required to be a symplectization of a contactization of a stabilization of a Weinstein domain.  On $U$ we have $X_\lambda=\partial_s$, inducing a trivial holonomy map $\partial_+U\to\partial_-U$, where
\[
\partial_-U := \{s=0\} \times [0,t_0]\times W_0\times r_0\,\mathbb{D}^2
\quad\text{and}\quad
\partial_+U := \{s=s_0\} \times [0,t_0]\times W_0\times r_0\,\mathbb{D}^2.
\]
Following the installation of the blocking apparatus, there will be a partially-defined holonomy map $\partial_+U\dashrightarrow\partial_-U$ which we must consider.

In order to precisely state the effect of the blocking apparatus, we define a few bits of notation.

\begin{definition}
For any Liouville domain $(W,\lambda)$, we let $\psi^s$ denote the time-$s$ flow of the Liouville vector field $X_\lambda$ and define $\|\cdot\|_{(W,\lambda)}\colon W\to [0,1]$ as follows:
\[
\|x\|_{(W,\lambda)} := \left\lbrace\begin{matrix}
0, & x\in \mathrm{Skel}(W,\lambda)\\
e^{-s_x}, & x\not\in\mathrm{Skel}(W,\lambda)
\end{matrix}\right.,
\]
where $s_x>0$ is the unique number such that $\psi^{s_x}(x) \in \partial W$.
\end{definition}

\begin{remark}
We will often simply write $\|\cdot\|_W$ in place of $\|\cdot\|_{(W,\lambda)}$.
\end{remark}

\begin{definition}
On the Weinstein domain $(r_0\,\mathbb{D}^2,\lambda_{\mathrm{stab}})$, for any $r_0>0$, we define $\|(p,q)\|_{\mathrm{stab}}=\sqrt{p^2+q^2}$.
\end{definition}

\begin{remark}
Note that $\|\cdot\|_{\mathrm{stab}}\neq \|\cdot\|_{(r_0\,\mathbb{D}^2,\lambda_{\mathrm{stab}})}$.
\end{remark}

\begin{definition}
For any Liouville domain $(W,\lambda)$ and any parameter $0<\epsilon<1$, we define
\[
I_\epsilon(W,\lambda) := \|\cdot\|_W^{-1}([0,1-\epsilon]),
\]
obtained from $W$ by removing an open neighborhood of the boundary $\partial W$.
\end{definition}

With this notation established, we can now state the key effects of installing a blocking apparatus.

\begin{theorem}\label{prop:mainprop}
    
Consider the Weinstein cobordism $(U=[0,s_0]\times[0,t_0]\times W_0\times r_0\,\mathbb{D}^2,e^s\,(dt+\lambda_0+\lambda_{\mathrm{stab}}))$.  Fix $0<\epsilon\ll \min(1,s_0,\frac{1}{2}(1-e^{-s_0})t_0)$ sufficiently small. If $r_0>0$ is sufficiently large, a blocking apparatus can be installed in $U$, i.e., a Liouville homotopy $\lambda_U \rightsquigarrow \lambda_{\mathrm{BA}}$ can be applied to $U$, such that the following properties hold. 
\begin{enumerate}
    \item (Weinstein compatibility)\label{mainprop:weinstein}

    \noindent The Liouville vector field $X_{\lambda_{\mathrm{BA}}}$ is Morse with 
    $8N_0$ critical points, where $N_0$ is the number of critical points of $(W_0, \lambda_0)$. Of these $8N_0$ critical points, $N_0^{\mathrm{crit}}$ of them are critical points of middle index, where $N_0^{\mathrm{crit}}$ is the number of middle index critical points of $(W_0, \lambda_0)$. 

    \item (Trapping properties)\label{mainprop:trapping}

\noindent There is a neighborhood $U_{\mathrm{trap}}$ of 
\[
[\epsilon,t_0-\epsilon]\times I_\epsilon(W_0,\lambda_0)\times \{(0,0)\} \subset [0,t_0] \times W_0 \times r_0\,\mathbb{D}^2
\]
such that any flowline passing through $\{s=s_0\}\times U_{\mathrm{trap}}\subseteq \partial_+U$ converges to a critical point of $X_{\lambda_{\mathrm{BA}}}$ in backward time.

    \item (Holonomy properties)\label{mainprop:holonomy}

\noindent Let $x\in \partial_+U$ be in the domain of the partially-defined holonomy map $h\colon\partial_+U\dashrightarrow\partial_-U$ induced by $X_{\lambda_{\mathrm{BA}}}$. Let $t(x),W_0(x)$ denote the $t$-coordinate and $W_0$-coordinate of $x$, respectively. Then the following properties hold. 
\begin{enumerate}
    \item For some constant $0<K<1$, we have $\|W_0(h(x))\|_{W_0}\leq Ke^{s_0}\,\|W_0(x)\|_{W_0}$.\label{mainprop:holonomy:weinstein}

    \item For the same constant $0<K<1$, whenever $\|W_0(x)\|_{W_0}<e^{-s_0}$ we have\label{prop:mainprop:stabilization-part}
    \[
    \|\pi_{r_0\,\mathbb{D}^2}(h(x))\|_{\mathrm{stab}}\leq Ke^{\frac{s_0}{2}}\,\|\pi_{r_0\,\mathbb{D}^2}(x)\|_{\mathrm{stab}}.
    \]
    
    \item If $t(x) \in [0,\epsilon] \cup [t_0 - \epsilon,t_0]$ and $W_0(x)\in I_{1 - e^{-s_0}}(W_0, \lambda_0)$, then $t(h(x)) \in [0,\epsilon] \cup [t_0 - \epsilon,t_0]$.\label{prop:mainprop:t-holo-part}
    
\end{enumerate}
\end{enumerate}
\end{theorem}

Some discussion is in order to better understand the statement of the theorem. First, we highlight three important features:

\begin{itemize}
    \item Morseness of the vector field $X_{
    \lambda_{\mathrm{BA}}}$ ensures that inside a blocking apparatus, each flowline either exits or limits to a critical point, in both forward and backward time; that is, there are no wandering or periodic orbits trapped in the interior.
	\item In backward time, a codimension $0$ subset of the flowlines intersecting $\partial_+U$ will  converge to a critical point in $U$.
	\item For flowlines which are not trapped, we can estimate certain aspects of the holonomy. In particular, the holonomy can be made arbitrarily small in the $t$-component near the stabilization origin. Though these estimates do not describe all aspects of the holonomy, they are sufficient for our applications.
\end{itemize}

Second, toward the final point above, we give a schematic depiction in Figure ~\ref{fig:schematic} of how we install a blocking apparatus. We will install blocking apparatuses so that the set we wish to trap in backward time is sufficiently far from $\partial W_0$ and near the origin in the stabilization direction, so that \ref{mainprop:holonomy:weinstein} and \ref{prop:mainprop:stabilization-part} provide useful control of the holonomy.  (While the estimate provided by \ref{mainprop:holonomy:weinstein} is modest, our strategy also takes advantage of a global contraction in the $W_0$-component.)  We cannot typically prevent the set we wish to trap from meeting $\partial([0,t_0])$, and thus the strategy is to install blocking apparatuses on supports which have some overlap in the $t$-direction (while remaining disjoint in the $s$-direction). The primary complication is that a flowline could enter an apparatus near $\partial([0,t_0])$ and exit with $t$-holonomy such that it misses the next apparatus in line. That this does not occur is precisely \ref{prop:mainprop:t-holo-part}. See Figure~\ref{fig:schematic} for a schematic of this strategy.

\begin{figure}[ht]
    \centering
    \vskip-1.5cm
	\begin{overpic}[scale=0.33]{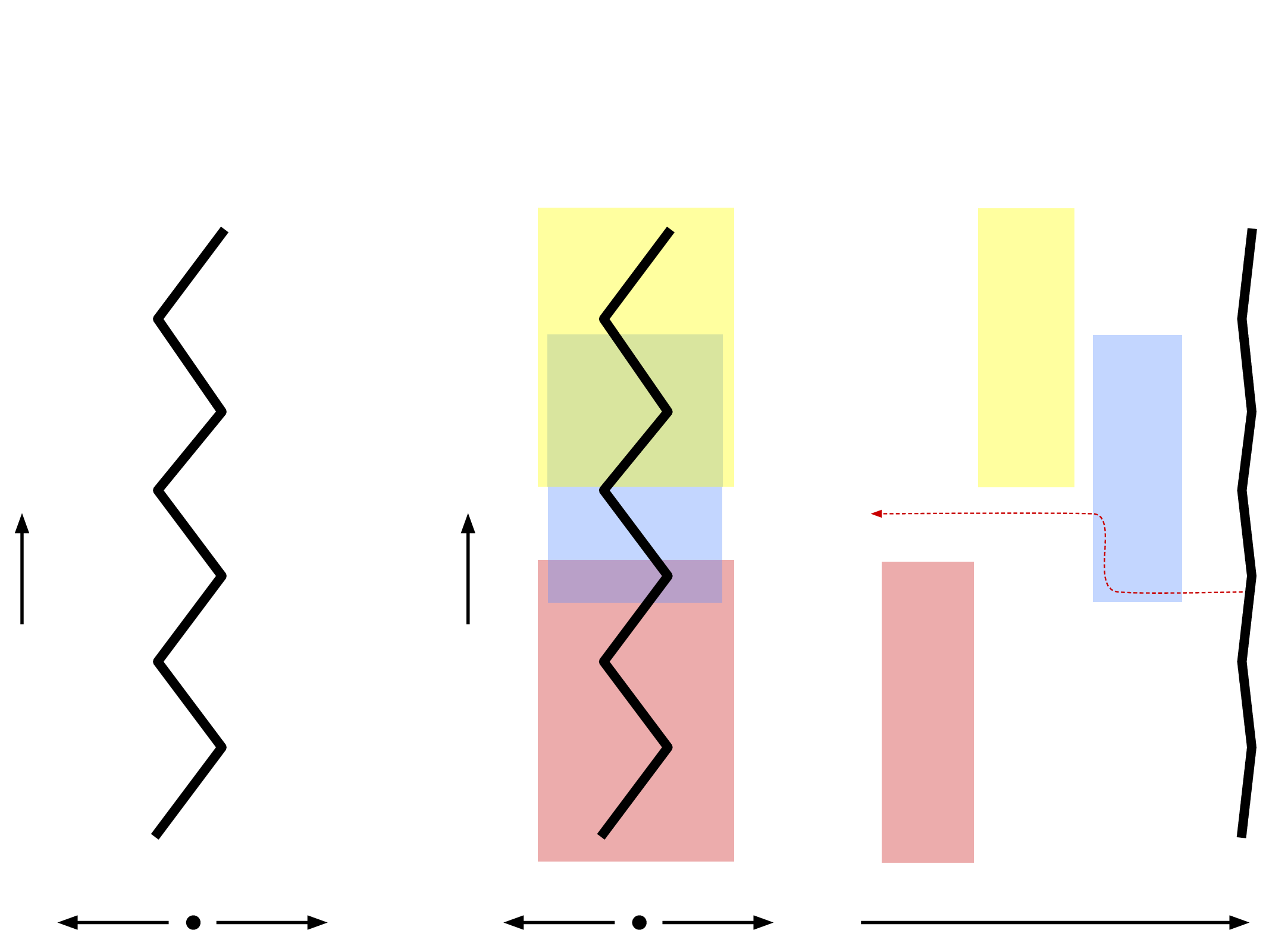}
        \put(1.5,35.5){$t$}
        \put(0,1.5){$W_0$}
        \put(36.5,35.5){$t$}
        \put(35,1.5){$W_0$}
        \put(99,1.5){$s$}
        \put(19.5,55){$K$}
	\end{overpic}
	\caption{A schematic for blocking apparatus installation. On the left, the set $K$ represents a cross section of flowlines we wish to trap; the $s$- and $\D^2$-directions are projected out of the picture. On the right, two projections of blocking apparatus profiles are given by the colored boxes. The dashed red path indicates a hypothetical problematic flowline that does not occur thanks to \ref{prop:mainprop:t-holo-part}.}
	\label{fig:schematic}
\end{figure}

Finally, we comment on the construction of the blocking apparatus and its relation to the previously studied explicit folds of Honda-Huang. These authors establish the paradigm of modeling the flowbox $U$ as a hypersurface in its contactization and modifying the dynamics of $U$ via explicit perturbations of $U$ in this contactization. The proof of Theorem ~\ref{prop:mainprop} proceeds in the same manner and we follow Honda-Huang in using the box fold as our most basic method of perturbation. However, the goal of Honda-Huang (achieving global convexity of a hypersurface containing $U$, i.e. Morse dynamics with positive and negative critical points) contrasts with our own (achieving Morse dynamics with only positive critical points) and the box fold alone is sufficient for neither purpose. Honda-Huang overcome insufficiency of the box fold with non-graphical perturbations; we must restrict to graphical perturbations, so our technique is the \emph{chimney fold}, a graphical perturbation based over a complicated but deliberately chosen subset of $U$. Thus, the novelty of Theorem ~\ref{prop:mainprop} is in the design and analysis of a chimney fold.

\subsection{Applications}
As discussed, our local operations are constructed with the goal of simplifying Liouville dynamics.  As an example of how such a simplification might play out, we will in Section~\ref{section:mitsumatsu} exhibit Liouville domains in every even dimension which are not Weinstein, but are stably Weinstein.

The domains we study are constructed by Huang in \cite[Example 0.7]{huang2019dynamical} and generalize well-known examples of Mitsumatsu \cite{mitsumatsu1995anosov} and \cite{geiges1995examples}.  We will recall the details of this construction in Section~\ref{section:mitsumatsu}, and we follow \cite{cieliebak2022floer} in referring to these as \emph{torus bundle Liouville domains}.

\begin{theorem}\label{thm:torus-bundle-domains}
The torus bundle Liouville domains of \cite[Example 0.7]{huang2019dynamical} are stably Weinstein.
\end{theorem}

Theorem~\ref{thm:torus-bundle-domains} gives an affirmative answer to \cite[Question 0.8]{huang2019dynamical}.  As pointed out by Huang, this is an interesting phenomenon, because $W_A\times r_0\,\mathbb{D}^2$ is diffeomorphic to the cotangent bundle of $\mathrm{Skel}(W_A,\lambda)$, but cannot be symplectomorphic to this cotangent bundle, since the stabilization component of $W_A\times r_0\,\mathbb{D}^2$ allows us to displace the skeleton.

The authors' expectation is that the technology introduced in this paper can be used to simplify Liouville dynamics in more general settings.  For instance, consider a Liouville domain $(W^{2n},\lambda)$ that admits a \emph{global transversal} $M^{2n-1}\subset W$.  That is, $M$ is a codimension 1 submanifold-with-boundary through which all flowlines of $W$ pass.  If the contact manifold $(M,\lambda)$ admits a partial open book decomposition compatible with the form\footnote{We point out that compatibility with $\lambda$ is strictly more difficult to achieve than is compatibility with $\ker\lambda$.} $\lambda$, then a collection of \textit{partial} blocking apparatuses can be installed on the stabilization of $(W,\lambda)$ to interrupt the flow through $M\times\{(0,0)\}$.  This ensures that the flowlines of our perturbed stabilization flow backwards to critical points, and with a bit more care we can argue that this new domain is also free of broken loops.  This strategy is inspired by the example of Honda-Huang in \cite[Section 7.3]{honda2019convex}; continuing the analogy with Honda-Huang, one would hope to create a ``backwards barricade" for an arbitrary $(W,\lambda)$ --- a collection of local transversals analogous to $(M,\lambda)$ above --- and then install several interrupters.  The difficulty we currently face in carrying out this strategy lies in obtaining partial open book decompositions compatible with the contact forms on our local transversals.  Certainly the strategy should not work if $W$ does not have the homotopy type of a CW-complex of dimension $n+1$, and the question of how this topological consideration might interact with the existence of a backwards barricade is the subject of ongoing investigation.

\subsection{Organization}
In Section~\ref{section:background} we recall the machinery by which Liouville domains may be treated as hypersurfaces in contact manifolds, and also give a simple criterion for determining whether or not a Liouville domain is Weinstein.  This criterion motivates our understanding of what it means for Liouville dynamics to be ``simplified," and thus the sort of local operations we should develop.

Sections~\ref{section:box_folds_pl} and~\ref{section:box_folds_smooth} are spent developing and exploring our first local operations: \emph{box folds}. The former section is dedicated to a piecewise-linear version of the box fold, primarily to build intuition and understanding; the smooth version which we make use of is defined in the latter section. The smooth analysis is technically independent from the piecewise-linear analysis, but is facilitated by the understanding developed in the piecewise-linear setting. As noted above, these operations will prove insufficient for our purposes, but will be a fundamental building block in our desired local operation. 

In Section~\ref{section:chimney_folds} we introduce \emph{chimney folds}, which are a variation on box folds and the primary component of the blocking apparatus. Again, we study a piecewise-linear version of the construction first before defining the technically independent smooth version. We define the blocking apparatus and prove Theorem~\ref{prop:mainprop} in Section~\ref{section:proof_main_prop}.

We discuss in Section~\ref{section:technical_properties} some technical properties which are useful when installing blocking apparatuses and in Section~\ref{section:mitsumatsu} we prove Theorem~\ref{thm:torus-bundle-domains}.

\subsection*{Acknowledgments}
The authors are happy to thank Ko Honda for countless conversations explaining (and re-explaining) the contents of \cite{honda2019convex}, as well as for offering a great deal of guidance and feedback throughout the completion of this project.  We also thank Yang Huang for his interest in our work, and for helpful discussions.  Finally, we are grateful to anonymous referees for comments which first led to a major revision of this paper and then to subsequent improvements.  Both authors were partially supported by NSF grant DMS-2003483; the first author was also partially supported by NSF grant DMS-2038103; and the second author was also partially supported by NSF grant DMS-1745583.

\section{Background}\label{section:background}

\subsection{Liouville domains and characteristic foliations}

Throughout the paper, we treat Liouville domains as hypersurfaces in their contactizations. This subsection is a brief introduction to this philosophy. First, we recall some basic notions of contact geometry as they pertain to the study of Liouville domains.

\begin{definition}[Symplectization and contactization] Let $(M, \xi = \ker \alpha)$ be a contact manifold, and let $(W, \lambda)$ be an exact symplectic manifold. 
\begin{enumerate}[label=(\roman*)]
    \item The \textbf{symplectization} of $(M, \alpha)$ is the exact symplectic manifold $(\R_s \times M, e^s\, \alpha)$. 
    \item The \textbf{contactization} of $(W,\lambda)$ is the contact manifold $(\R_z \times W, dz + \lambda)$. 
\end{enumerate}

\end{definition}

\begin{remark}
Both symplectization and contactization have compact analogues, where the first component $\mathbb{R}$ is replaced by a compact interval $[0,s_0]$ or $[0,z_0]$; we will sometimes also refer to these compact constructions as symplectization or contactization.
\end{remark}

\begin{definition}[Characteristic foliation] Let $(M^{2n+1},\xi = \ker \alpha)$ be a contact manifold, and let $\Sigma^{2n} \subset M$ be a hypersurface. The \textbf{characteristic foliation} of $\Sigma$ is the singular $1$-dimensional foliation of $\Sigma$ given by 
\[
(T\Sigma \cap \xi\mid_{\Sigma})^{\bot}.
\]
Here, $\bot$ is the symplectic orthogonal complement with respect to the conformal symplectic structure on $\xi$ induced by $d\alpha$. Note that $(T\Sigma \cap \xi\mid_{\Sigma})^{\bot} \subset T\Sigma$, as $T\Sigma \cap \xi \subset \xi$ is necessarily coisotropic.  
\end{definition}

In practice, we will use the following lemma to compute characteristic foliations. 

\begin{lemma}\label{cflemma}\cite[Lemma 2.5.20]{geiges_2008}
Let $\Sigma^{2n}$ be an oriented hypersurface in a contact manifold $(M^{2n+1}, \alpha)$. Let $\beta = \alpha\mid_{\Sigma}$ and let $\Omega$ be a volume form on $\Sigma$. The characteristic foliation of $\Sigma$ is directed by the unique vector field $X$ satisfying 
\[
\iota_X \Omega = \beta\, (d\beta)^{n-1}.    
\]
\end{lemma}

The importance of the characteristic foliation in the study of Liouville domains is indicated by the following fact, which is the foundation of the philosophy of this paper.

\begin{lemma}\label{lemma:liouville-vf-directs-char-fol}
Let $(W^{2n}, \lambda)$ be an exact symplectic manifold with Liouville vector field $X_{\lambda}$. Consider the contactization $(\R_z \times W, dz + \lambda)$. The oriented characteristic foliation of the hypersurface $\Sigma := \{z = 0\}$ is directed by $X_{\lambda}$. 
\end{lemma}

\begin{remark}
For the rest of the paper, we will represent the characteristic foliation by any vector field directing the foliation. 
\end{remark}

\begin{proof}
A volume form on $\Sigma$ is given by $\Omega = \frac{1}{n}\, (d\lambda)^n$, and $(dz + \lambda)\mid_{\Sigma} = \lambda$. Note that 
\begin{align*}
    \iota_{X_{\lambda}} \Omega &= \frac{1}{n}\, \iota_{X_{\lambda}}(d\lambda)^n \\
    &= \frac{1}{n}\, n\, (\iota_{X_{\lambda}} d\lambda)\, (d\lambda)^{n-1} \\
    &= \lambda\, (d\lambda)^{n-1}.
\end{align*}
By Lemma \ref{cflemma}, it follows that $X_{\lambda}$ directs the characteristic foliation of $\Sigma$. 
\end{proof}

The observation of Lemma~\ref{lemma:liouville-vf-directs-char-fol} drives all of the constructions in this paper.  We know from \cite[Proposition 11.8]{cieliebak2012stein} that, up to a diffeotopy, Liouville homotopies of $(W^{2n},\lambda)$ are exact.  To be precise, we know that if $\lambda_t$ is a Liouville homotopy with $\lambda_0=\lambda$, then there is a diffeotopy $\phi_t\colon W\to W$ with $\phi_0=\mathrm{Id}$ and a smooth family of functions $f_t\colon W\to W$ such that
\[
\phi_t^*\lambda_t - \lambda_0 = df_t.
\]
But exact Liouville homotopies of $(W,\lambda)$ are precisely those which may be realized by \emph{graphical perturbations} of $W$ in $(\mathbb{R}_z\times W,dz+\lambda)$.  Indeed, we may define maps $F_t\colon W \to \Sigma_t$ by
\[
F_t(x) := (f_t(\phi_t^{-1}(x)),\phi_t^{-1}(x)),
\]
where $\Sigma_t:=\{z=f_t(x)\}$, and observe that $F_t^*(dz+\lambda)|_{\Sigma_t}=\lambda_t$, for all $t$.

The local Liouville homotopies constructed in this paper are all exact, and thus we view them as arising from compactly supported graphical perturbations of $(W,\lambda)$ in its contactization.

\subsection{A Weinstein criterion}\label{subsection:weinstein-criterion}
We present here a simple dynamical test to determine whether or not a Liouville domain is in fact Weinstein.

\begin{proposition}\label{prop:weinstein-criterion}
A Liouville domain $(W,\lambda)$, with Liouville vector field $X_{\lambda}$, is Weinstein if and only if the following conditions are satisfied:
\begin{enumerate}
	\item for any zero $p\in W$ of $X_{\lambda}$, there is a neighborhood of $p$ on which $X_{\lambda}$ admits a Morse Lyapunov function;\label{criterion:morse-cps}
	\item for any $p\in W$, the unique flowline of $X_{\lambda}$ passing through $p$ converges to a zero of $X_{\lambda}$ in backward time;\label{criterion:backward-time-cps}
	\item the vector field $X_{\lambda}$ does not admit any \emph{broken loops}, where a broken loop is a map $c\colon\mathbb{R}/\mathbb{Z}\to W$ with $0=a_0 < a_1 < \cdots < a_N =1$ such that $c(a_i)$ is a zero of $X_{\lambda}$ and $c(a_i,a_{i+1})$ is an oriented flowline of $X_{\lambda}$ from $c(a_i)$ to $c(a_{i+1})$.\label{criterion:no-broken-loops}
\end{enumerate}
\end{proposition}

\begin{remark}
Proposition \ref{prop:weinstein-criterion} and its proof are adapted from \cite[Proposition 2.1.3]{honda2019convex}, which gives criteria for a vector field on a closed manifold to be Morse.
\end{remark}

\begin{proof}[Proof of Proposition \ref{prop:weinstein-criterion}]
If $(W,\lambda)$ is Weinstein, then certainly conditions \ref{criterion:morse-cps}-\ref{criterion:no-broken-loops} are satisfied, since in this case $X_{\lambda}$ admits a globally-defined Morse Lyapunov function.  Let us assume that these conditions hold and show that $(W,\lambda)$ is Weinstein.  For this, it suffices to construct a handle decomposition of $W$ which is compatible with $X_{\lambda}$.

Let $C(X_{\lambda})=\{p_1,\ldots,p_k\}$ be the zeros of $X_{\lambda}$, a finite set by the compactness of $W$ and Criterion~\ref{criterion:morse-cps}.  Criterion~ \ref{criterion:no-broken-loops} allows us to put a partial order on $C(X_{\lambda})$, with $p_i\prec p_j$ whenever $X_{\lambda}$ admits a (possibly broken) flowline from $p_i$ to $p_j$.  Criterion~\ref{criterion:backward-time-cps} then implies that any minimal element of $C(X_{\lambda})$ is a critical point of index 0.  We let $C_0$ denote the set of minimal elements of $C(X_{\lambda})$, and for $j\geq 0$ we let $C_{j+1}$ be the union of $C_j$ and the minimal elements of $C(X_{\lambda})\setminus C_j$.  Our handle decomposition of $W$ can then be constructed inductively: we begin with a standard neighborhood of $C_0$, then attach the handles corresponding to $C_1$, then those of $C_2$, and so on.

To see that $(W,\lambda)$ is Weinstein, notice that our handle decomposition provides a Morse Lyapunov function for $X_\lambda$ on a neighborhood $U$ of $\mathrm{Skel}(W,\lambda)$.  For any point $p\in W\setminus U$, Criterion~\ref{criterion:backward-time-cps} ensures that the unique flowline of $X_{\lambda}$ passing through $p$ will enter $U$ in backward time, and we may use these flowlines to extend the domain of our Morse Lyapunov function to be all of $W$.
\end{proof}

Because we are considering Liouville domains up to Liouville homotopy, we will always assume that Criterion~\ref{criterion:morse-cps} is satisfied.  Indeed, standard transversality arguments show that a generic vector field has only nondegenerate critical points.  Thus, with $X_{\lambda}$ as above, we may choose an arbitrarily small vector field $\tilde{X}$ on $W$ so that $X_{\lambda}+\tilde{X}$ has only nondegenerate critical points.  We choose $\tilde{X}$ small enough to ensure that $X_t := X_{\lambda}+t\tilde{X}$ is outwardly transverse to $\partial W$, for $t\in[0,1]$.  Then
\[
\lambda_t := \iota_{X_t}\omega
\]
defines a Liouville homotopy from $(W,\lambda)$ to $(W,\lambda_1)$, the latter of which has Liouville vector field $X_{\lambda}+\tilde{X}$.

Verifying that a given Liouville domain is Weinstein will thus consist of verifying Criteria~\ref{criterion:backward-time-cps} and \ref{criterion:no-broken-loops}.  Our local operations are developed with these criteria in mind --- the operations aim to trap large sets of flowlines in backward time, and to do so with a holonomy which prevents broken loops from developing.

\section{Piecewise linear box folds}\label{section:box_folds_pl}

The goal of this section and the next is to define the \textit{box fold}, originally due to \cite{honda2019convex}. Box folds are graphs in the contactization of smooth approximations of the indicator function $\mathbf{1}_B$ of certain subsets $B$ of a Liouville domain.

The effect of a graphical perturbation to a Liouville domain $(W,\lambda)$ is straightforward. Namely if $F:W \to \R$ is a smooth function, there is an induced Liouville homotopy from $(W,\lambda)$ to $(W, dF + \lambda)$. The Liouville vector field of the latter is $X_F + X_{\lambda}$, where $X_F$ is the Hamiltonian vector field of $F$ with respect to $d\lambda$. To prove Theorem~\ref{prop:mainprop}, we need a careful understanding of the relevant dynamics of such Hamiltonian vector fields, but rather than analyze these Hamiltonian vector fields directly, our strategy takes inspiration from \cite{honda2019convex} and proceeds as follows: 
\begin{enumerate}
    \item First we consider non-graphical, piecewise linear hypersurfaces in the contactization that approximate the desired smooth graph. The dynamics of the characteristic foliation on such a hypersurface, which is easy to compute, can be viewed as a ``limit'' of the dynamics of the Hamiltonian vector fields. \label{non-graphical}

    \item Next, we ``tip'' the vertical sides of the hypersurfaces to make them graphical.
    
    \item Finally, we smooth corners and edges of the piecewise linear hypersurfaces via a convolution.
   
\end{enumerate}

This section addresses~\ref{non-graphical} and is organized as follows. In \ref{subsection:PL_low} we study the piecewise linear box fold in a low-dimensional model. In \ref{subsection:PL_high} we extend this piecewise linear model to arbitrary dimensions. In \ref{subsection:box_holes} we discuss a variant called a \textit{box hole}. Finally, in \ref{subsection:pre_chimney} we introduce a piecewise linear fold called a \textit{pre-chimney} fold. This is preparation for Section \ref{section:chimney_folds}, where we define the notion of a (full) chimney fold.

\subsection{Piecewise linear box folds in dimension $2$}\label{subsection:PL_low}

We begin with the piecewise linear box fold in a standard low-dimensional model. In particular, consider the contactization of $(\R^2, e^s\, dt)$:
\[
\left(\R^3_{z,s,t}, \, \alpha = dz + e^s\, dt\right).
\]
The surface $W = \{z=0\}$ represents a region of a $2$-dimensional Liouville domain that we wish to perturb. Observe that the Liouville vector field of $(W, e^s\, dt)$ is $\partial_s$, and that this directs the (unoriented) characteristic foliation of $W$. To streamline the analysis of the backward-time dynamics, for the rest of the paper \textit{we choose to orient all characteristic foliations with respect to the backward Liouville flow}. Hence, the oriented characteristic foliation of $W$ is directed by $-\partial_s$. 

\begin{definition}
Fix $z_0, s_0, t_0 > 0$. A \textbf{(piecewise linear, low-dimensional) box fold with parameters $z_0,s_0,t_0$}, denoted $\Pi^{PL}$, is the surface 
\[
\Pi^{PL} := \overline{\partial\left([0,z_0] \times [0,s_0] \times [0,t_0]\right) \setminus \{z=0\}}.
\]
\end{definition}

We will refer to $[0,s_0]$ as the \textit{$s$-support} or \textit{symplectization length} of the fold, and $[0,t_0]$ as the \textit{$t$-support} or \textit{Reeb length} of the fold. Note that this instance of ``Reeb'' is not referring to $\partial_z$, which is the Reeb vector field of the contactization of $W$. We will also use the following shorthand notation to refer to the various sides of $\Pi^{PL}$:
\[
\underline{z=z_0} := \Pi^{PL} \cap \{z=z_0\}
\]
and likewise with the other sides $\underline{t=0}$, $\underline{t=t_0}$, $\underline{s=0}$, and $\underline{s=s_0}$.

\begin{figure}[htb]
    \centering
    \def\svgwidth{0.75\linewidth}
    \graphicspath{{Figures/}}
\begingroup%
  \makeatletter%
  \providecommand\color[2][]{%
    \errmessage{(Inkscape) Color is used for the text in Inkscape, but the package 'color.sty' is not loaded}%
    \renewcommand\color[2][]{}%
  }%
  \providecommand\transparent[1]{%
    \errmessage{(Inkscape) Transparency is used (non-zero) for the text in Inkscape, but the package 'transparent.sty' is not loaded}%
    \renewcommand\transparent[1]{}%
  }%
  \providecommand\rotatebox[2]{#2}%
  \newcommand*\fsize{\dimexpr\f@size pt\relax}%
  \newcommand*\lineheight[1]{\fontsize{\fsize}{#1\fsize}\selectfont}%
  \ifx\svgwidth\undefined%
    \setlength{\unitlength}{1006.28189623bp}%
    \ifx\svgscale\undefined%
      \relax%
    \else%
      \setlength{\unitlength}{\unitlength * \real{\svgscale}}%
    \fi%
  \else%
    \setlength{\unitlength}{\svgwidth}%
  \fi%
  \global\let\svgwidth\undefined%
  \global\let\svgscale\undefined%
  \makeatother%
  \begin{picture}(1,0.63398689)%
    \lineheight{1}%
    \setlength\tabcolsep{0pt}%
    \put(0,0){\includegraphics[width=\unitlength]{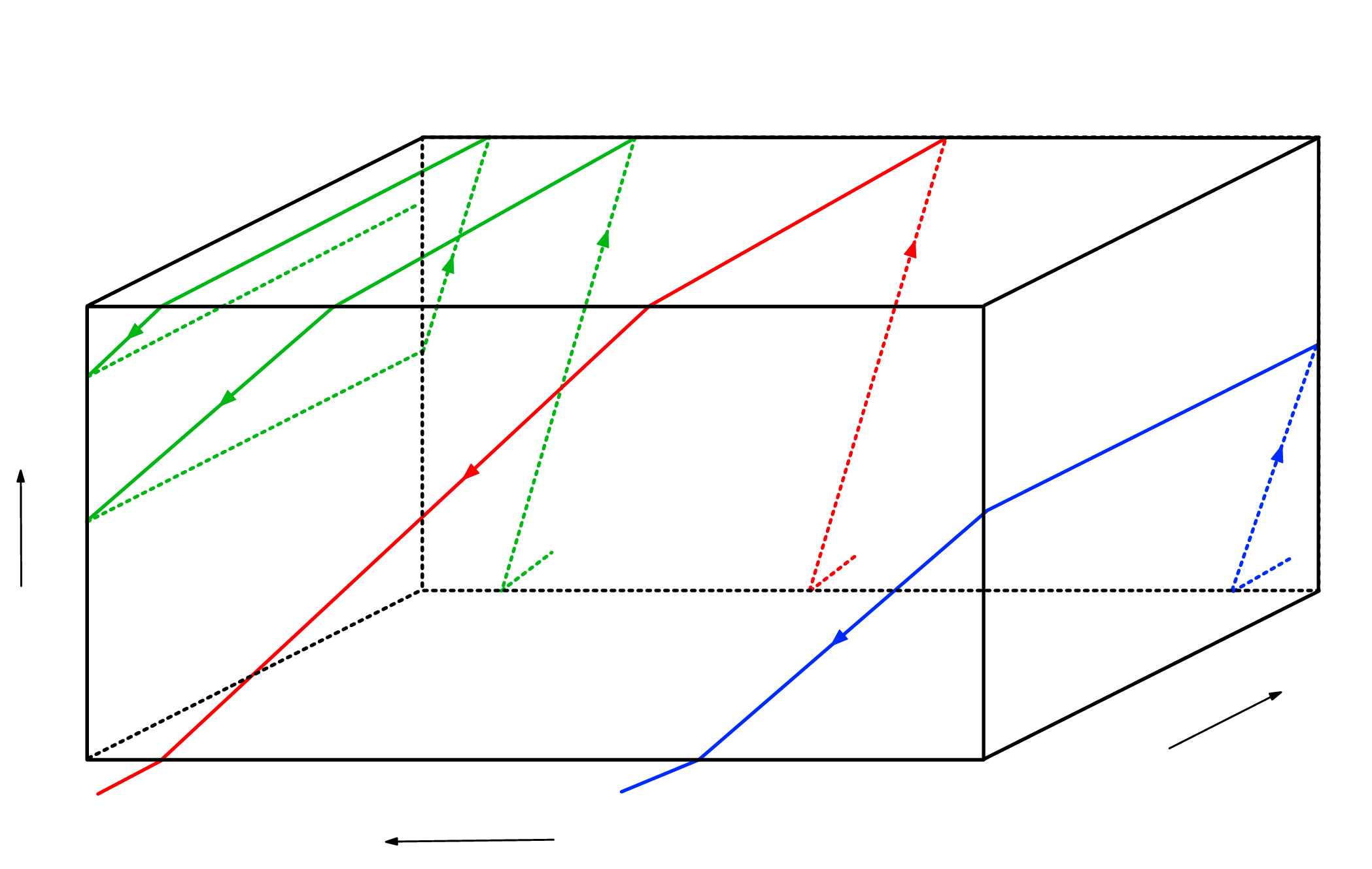}}%
    \put(0.01112926,0.30378216){\makebox(0,0)[lt]{\lineheight{1.25}\smash{\begin{tabular}[t]{l}$z$\end{tabular}}}}%
    \put(0.25870645,0.0177363){\makebox(0,0)[lt]{\lineheight{1.25}\smash{\begin{tabular}[t]{l}$t$\end{tabular}}}}%
    \put(0.94277264,0.13081061){\makebox(0,0)[lt]{\lineheight{1.25}\smash{\begin{tabular}[t]{l}$s$\end{tabular}}}}%
  \end{picture}%
\endgroup%

    \caption{A box fold with $z_0 < t_0$ and the three qualitative types of flowlines entering the fold according to Lemma \ref{lemma:PL_box_fold_holonomy_low}. The flowline in green is trapped in backward time, because it spirals around $\underline{t=t_0}\cap \underline{z=z_0}$ via the faces $\underline{s=s_0} \to \underline{z=z_0}\to \underline{s=0}\to \underline{t=t_0} \to \underline{s=s_0}$. The flowline in red passes through the fold, and its holonomy is given by a shift in the $t$-direction. The flowline in blue also passes through the fold, but its holonomy is given by a scaling in the $t$-direction.}
    \label{fig:box1}
\end{figure}

The foliation of each side of $\Pi^{PL}$ is given by the following table, where we have used Lemma \ref{cflemma}. Note that the orientation of each side is chosen to be consistent with the backward Liouville flow. 

\vspace{2mm}
\begin{center}
\begin{tabular}{ |c|c|c|c| } 
 \hline
 Side & Area form & Restriction of $dz + e^s\, dt$ & Characteristic foliation \rm \\ 
 \hline 
 $\underline{z=z_0}$ & $-e^s\, ds\, dt$ & $e^s\, dt$ & $  -\partial_s$ \\
 $\underline{s=0}$ & $-dz\, dt$ & $dz + dt$ & $\partial_t -\partial_z $ \\
 $\underline{s=s_0}$ & $- dt\, dz$ & $dz + e^{s_0}\, dt$ & $-\partial_t + e^{s_0} \partial_z$ \\
 $\underline{t=0}$ & $-ds\, dz $ & $dz$ & $-\partial_s$ \\
 $\underline{t=t_0}$ & $-dz\, ds $ & $dz$ & $\partial_s$ \\
 \hline
\end{tabular}
\end{center}
\vspace{2mm}

The most important feature of a box fold is that, in backward time, it traps some of the flowlines entering the fold through $\underline{z=0} \cap \underline{s=s_0}$. The flowlines that are not trapped reach $\underline{z=0} \cap \underline{s=0}$ and subsequently exit the fold after experiencing some holonomy in the $t$-direction; see Figure \ref{fig:box1}. For a flowline to exit the fold, it must reach $\underline{z=0}\cap \underline{s=0}$. 

The following lemma summarizes all of the induced dynamical behavior of a piecewise linear box fold in low dimensions. 

\begin{lemma}\label{lemma:PL_box_fold_holonomy_low}
Let $h^{PL}:\{z=0\} \times \{s=s_0\} \times [0,t_0] \dashrightarrow \{z=0\} \times \{s=0\}\times [0,t_0]$ be the partially-defined holonomy map given by the oriented characteristic foliation of $\Pi^{PL}$. The domain of $h^{PL}$ is 
\[
\left(0,\,e^{-s_0}\min(z_0,t_0)\right)_t \, \cup \, \left(e^{-s_0}\min(z_0,t_0),\, t_0 - (1-e^{-s_0})\min(z_0,t_0)\right)_t
\]
and 
\[
h^{PL}(t) = \begin{cases} 
e^{s_0}t & t\in \left(0,\,e^{-s_0}\min(z_0,t_0)\right) \\
t + (1-e^{-s_0})z_0 & t\in \left(e^{-s_0}\min(z_0,t_0),\, t_0 - (1-e^{-s_0})\min(z_0,t_0)\right)
\end{cases}.
\]
In particular, if $z_0 \geq t_0$ then the domain of $h^{PL}$ is $(0,e^{-s_0}t_0)$, and $h^{PL}(t) = e^{s_0}t$.
\end{lemma}

\begin{remark}
Before we prove Lemma \ref{lemma:PL_box_fold_holonomy_low}, we include some discussion about its statement and role in the context of the paper.

\begin{enumerate}
    
    \item Most of the box folds in this paper will satisfy $z_0 \geq t_0$ (in fact, most will satisfy $z_0 \geq e^{s_0}t_0$), so for simplicity, the reader can focus on the ``in particular'' statement. 
    
    \item We will often abuse notation and refer to $h^{PL}$ as a partially-defined function on $\{s=s_0\} \times [0,t_0]$ or simply $[0,t_0]$. The complement of the domain of $h^{PL}$ in $[0,t_0]$ is the set of points whose Liouville flowlines do not pass through the fold, called the \textit{trapping region}. For example, if $z_0 \geq t_0$ then this region is $[e^{-s_0}t_0, t_0]$.  

    \item As $\Pi^{PL}$ is not a smooth hypersurface, we are abusing the notion of ``flowline.'' Consequently, when discussing piecewise linear folds we often do not carefully address $t$-values like $0, t_0$, and $e^{-s_0}t_0$. For example, we may refer to a trapping region of $\Pi^{PL}$ as $(e^{-s_0}t_0, t_0)$ or $[e^{-s_0}t_0, t_0]$, even though there are technically no flowlines passing through $\{s=s_0\} \times \{t=t_0\}$ or $\{s=s_0\} \times \{t=e^{-s_0}t_0\}$ due to edges and corners of the fold. In the present section this is unimportant. Intervals and trapping regions will be stated precisely when we consider smooth folds in Section \ref{section:box_folds_smooth}. 
    
\end{enumerate}
\end{remark}

\begin{proof}
Assume first that $z_0 \geq t_0$.

Suppose that a flowline enters the fold along $\underline{z=0}\cap \underline{s=s_0}$ with initial $t$-coordinate $\bar{t}\in (0, e^{-s_0}t_0)$. The characteristic foliation of $\underline{s=0}$ is directed by $-\partial_t + e^{s_0}\, \partial_z$. Because $\bar{t} < e^{-s_0}t_0$ and  $z_0 > t_0$, the flowline reaches $\underline{t=0}$ before $\underline{z=z_0}$. It reaches $\underline{t=0}$ with $z$-coordinate $e^{s_0}\bar{t} < t_0 \leq z_0$. Along $\underline{t=0}$ the flowline travels via $-\partial_s$ to $\underline{s=0}$. Here the characteristic foliation is $\partial_t - \partial_z$. Since the $z$-coordinate $e^{s_0}\bar{t}$ is less than $z_0$, the flowline reaches $\underline{z=0}$ before $\underline{t=t_0}$ and exits the fold with $t$-coordinate $e^{s_0}\bar{t}$. This proves that $h^{PL}(t) = e^{s_0}t$ for $t\in (0,e^{-s_0}t_0)$.

Next, suppose that a flowline enters along $\underline{z=0}\cap \underline{s=s_0}$ with initial $t$-coordinate $\bar{t}\in (e^{-s_0}t_0, t_0)$. The characteristic foliation of $\underline{s=0}$ is directed by $-\partial_t + e^{s_0}\, \partial_z$. The flowline reaches either $\underline{t=0}$ or $\underline{z=z_0}$. 

\vspace{2mm}
\noindent \textit{Case 1: The flowline reaches $\underline{t=0}$.}
\vspace{2mm}

The flowline reaches $\underline{t=0}$ with $z$-coordinate $e^{s_0}\bar{t} > t_0$. It travels along $\underline{t=0}$ via $-\partial_s$ to $\underline{s=0}$. Here the foliation is directed by $\partial_t - \partial_z$. Because $e^{s_0}\bar{t} > t_0$, the flowline reaches $\underline{t=t_0}$ with $z$-coordinate $e^{s_0}\bar{t} - t_0 > 0$. Along $\underline{t=t_0}$ it follows $\partial_s$, returning to $\underline{s=s_0}$. Note that the $z$-coordinate has increased from the last time the flowline was on $\underline{s=s_0}$. Here the flowline follows $-\partial_t + e^{s_0}\, \partial_z$ and enters either Case 1 (again) or Case 2. Each time the flowline cycles through Case 1, the $z$-coordinate increases. Thus, eventually, the initial $z$-coordinate along $\underline{s=s_0}$ will be large enough for the flowline to enter Case 2. 

\vspace{2mm}
\noindent \textit{Case 2: The flowline reaches $\underline{z=z_0}$.}
\vspace{2mm}

Let $\tilde{t}>0$ denote the $t$-coordinate at which the flowline reaches $\underline{z=z_0}$. The characteristic foliation here is directed by $-\partial_s$, and so the flowline then reaches $\underline{s=0}$. Since $z_0 \geq t_0$ and $\tilde{t} > 0$, the flowline follows $\partial_t - \partial_z$ and reaches $\underline{t=t_0}$ with $z$-coordinate $z_0 - (t_0 - \tilde{t})$. Here the foliation is $\partial_s$, and thus the flowline travels to $\underline{s=s_0}$. The flowline then follows $-\partial_t + e^{s_0}\, \partial_z$, returning to $\underline{z=z_0}$, i.e., returning to Case 2, with $t$-coordinate $t_0 - e^{-s_0}(t_0 - \tilde{t}) > \tilde{t}$. Note that the $t$-coordinate has increased upon return to $\underline{z=z_0}$. 

\vspace{2mm}

The above analysis shows that every flowline entering the fold with initial $t$-coordinate in $(e^{-s_0}t_0, t_0)$ cycles through Case 1 sufficiently many times to reach Case 2, which is then cycled through indefinitely; see Figure \ref{fig:box2}. The flowline spirals around $\underline{z=z_0}\cap \underline{t=t_0}$, never exiting the fold. This completes the proof of Lemma \ref{lemma:PL_box_fold_holonomy_low} when $z_0 \geq t_0$. 

The casework for a fold with $z_0 < t_0$ is identical but more tedious because of the presence of a third case. As we do not use this type of fold in the rest of paper, we refer to Figure \ref{fig:box2}.
\end{proof}

\begin{figure}[htb]
    \centering
    \def\svgwidth{0.85\linewidth}
    \graphicspath{{Figures/}}
\begingroup%
  \makeatletter%
  \providecommand\color[2][]{%
    \errmessage{(Inkscape) Color is used for the text in Inkscape, but the package 'color.sty' is not loaded}%
    \renewcommand\color[2][]{}%
  }%
  \providecommand\transparent[1]{%
    \errmessage{(Inkscape) Transparency is used (non-zero) for the text in Inkscape, but the package 'transparent.sty' is not loaded}%
    \renewcommand\transparent[1]{}%
  }%
  \providecommand\rotatebox[2]{#2}%
  \newcommand*\fsize{\dimexpr\f@size pt\relax}%
  \newcommand*\lineheight[1]{\fontsize{\fsize}{#1\fsize}\selectfont}%
  \ifx\svgwidth\undefined%
    \setlength{\unitlength}{968.60520039bp}%
    \ifx\svgscale\undefined%
      \relax%
    \else%
      \setlength{\unitlength}{\unitlength * \real{\svgscale}}%
    \fi%
  \else%
    \setlength{\unitlength}{\svgwidth}%
  \fi%
  \global\let\svgwidth\undefined%
  \global\let\svgscale\undefined%
  \makeatother%
  \begin{picture}(1,0.67247392)%
    \lineheight{1}%
    \setlength\tabcolsep{0pt}%
    \put(0,0){\includegraphics[width=\unitlength]{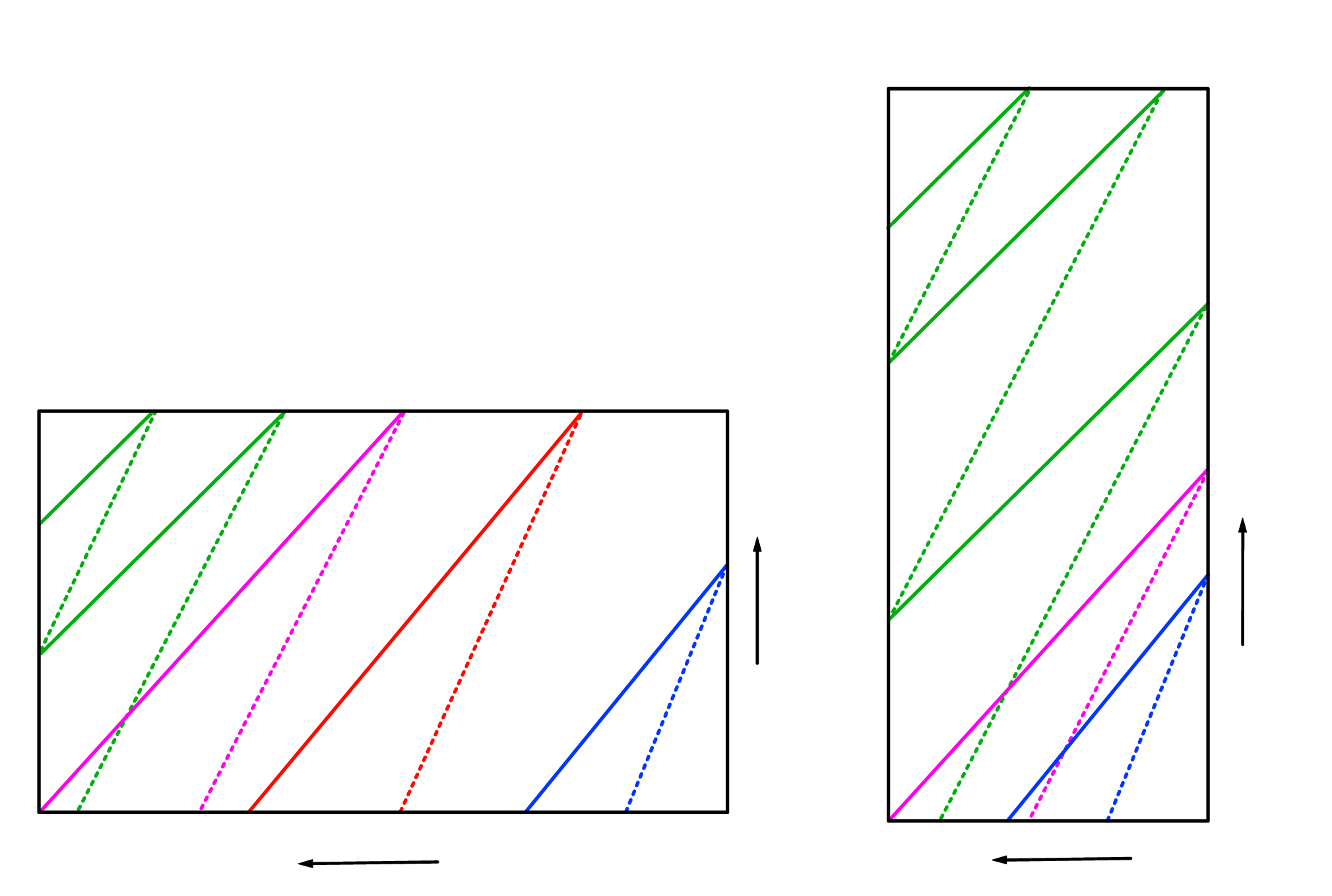}}%
    \put(0.20682011,0.02093307){\makebox(0,0)[lt]{\lineheight{1.25}\smash{\begin{tabular}[t]{l}$t$\end{tabular}}}}%
    \put(0.56171945,0.27465159){\makebox(0,0)[lt]{\lineheight{1.25}\smash{\begin{tabular}[t]{l}$z$\end{tabular}}}}%
    \put(0.7288211,0.02553218){\makebox(0,0)[lt]{\lineheight{1.25}\smash{\begin{tabular}[t]{l}$t$\end{tabular}}}}%
    \put(0.92680123,0.2868969){\makebox(0,0)[lt]{\lineheight{1.25}\smash{\begin{tabular}[t]{l}$z$\end{tabular}}}}%
  \end{picture}%
\endgroup%

    \caption{A head-on view of two different box folds, one with $z_0 < t_0$ (left) and $z_0 > t_0$ (right). This is a visual depiction of Lemma \ref{lemma:PL_box_fold_holonomy_low}. All of the dashed lines on $\underline{s=s_0}$ have $(t,z)$ slope $-e^{s_0}$ and all of the solid lines on $\underline{s=0}$ have $(t,z)$ slope $-1$. The pink flowlines represent the lower threshold of the trapping region. Observe on the right picture that increasing $z_0$ beyond $t_0$ (with $s_0$ fixed) does not increase the size of the trapping region.}
    \label{fig:box2}
\end{figure}

\begin{remark}[Symplectization support]
We have defined $\Pi^{PL}$ as a fold with $s$-support $[0,s_0]$, but it will be necessary for us to install box folds with symplectization support in some other interval $[s_1, s_2]$, where $s_2 - s_1 = s_0$. The only change to the characteristic foliation is that the $(t,z)$-slope on $\underline{s=s_1}$ is $-e^{s_1}$, and on $\underline{s=s_2}$ it is $-e^{s_2}$. The proof of Lemma \ref{lemma:PL_box_fold_holonomy_low}, see also Figure \ref{fig:box2}, implies that such a fold with $z_0 \geq e^{s_1}t_0$ has identical holonomy to a box fold installed with $z_0 \geq t_0$ and symplectization support in $[0,s_0]$. In other words, the holonomy of a box fold only depends on the symplectization \textit{length}, and not the actual interval, provided we increase $z_0$ accordingly. This poses no problem for us. In the future, we will say ``install a box fold with $s$-support $[s_1,s_2]$'' with the understanding that that the $z$-parameter is adjusted accordingly.   
\end{remark}

\subsection{Piecewise linear box folds in dimension $> 2$}\label{subsection:PL_high}

To install box folds on Liouville domains of arbitrary dimensions, we consider folds based over the symplectization of a contact handlebody. 

\begin{definition}
A \textbf{contact handlebody} is a contact manifold of the form 
\[
(H_0 := [0,t_0] \times W_0,\, dt + \lambda_0)
\]
where $(W_0, \lambda_0)$ is a Weinstein domain. 
\end{definition}

\begin{remark}
It is important that $(W_0, \lambda_0)$ is a domain, as opposed to a cobordism with nonempty negative boundary. When we discuss pre-chimney folds in \ref{subsection:pre_chimney} we will weaken this, but for now our analysis crucially uses the assumption that $(W_0, \lambda_0)$ has a skeleton and outward pointing Liouville vector field. 
\end{remark}

The low-dimensional box fold from \ref{subsection:PL_low} is based over the contact handlebody $([0,t_0], dt)$ where $W_0 = \{\text{pt}\}$. The standard example of a contact handlebody that the reader should keep in mind as it pertains to this paper is $([0,t_0] \times \D^2, \, dt + \frac{1}{2}r^2\, d\theta)$; here, the Liouville vector field of $(W_0 = \D^2,\, \lambda_0 = \frac{1}{2}r^2\, d\theta)$ is the radial vector field $\frac{1}{2}r\, \partial_r$.

Let $(H_0 = [0,t_0] \times W_0^{2n-2},\, dt + \lambda_0)$ be a contact handlebody and let 
\[
\left([0,s_0] \times H_0, \, e^s(dt + \lambda_0)\right)
\]
be its symplectization. This represents a region in our given $2n$-dimensional Liouville domain that we wish to perturb. The Liouville vector field of this model is $\partial_s$. In practice, given an arbitrary Liouville domain we will identify such a region by finding a contact handlebody transverse to the Liouville vector field and considering its time-$s_0$ flow. 

As before, we realize this region as the hypersurface $\{z=0\}$ inside its contactization:
\[
\left(\R_z \times [0,s_0] \times H_0, \, dz + e^s(dt + \lambda_0) \right).
\]

\begin{definition}
Fix $z_0, s_0 > 0$, and $H_0 = [0,t_0] \times W_0$ as above. A \textbf{(piecewise linear, high-dimensional) box fold with parameters $z_0, s_0, t_0$}, denoted $\Pi^{PL}$, is the hypersurface
\[
\Pi^{PL} := \overline{\partial\left([0,z_0] \times [0,s_0] \times H_0\right) \setminus \{z=0\}}.
\]
\end{definition}

We use the same language and notation as before. This time, there is an additional vertical side to consider: 
\[
\underline{\partial W_0} = [0,z_0] \times [0,s_0]\times [0,t_0] \times \partial W_0.
\]
The following lemma computes the backward oriented characteristic foliation of each side of $\Pi^{PL}$.

\begin{lemma}\label{lemma:PL_cf_high_dimensional}
Let $X_{\lambda_0}$ denote the Liouville vector field of $(W_0^{2n-2}, \lambda_0)$, let $\eta_0 := \lambda_0\mid_{\partial W_0}$ be the induced contact form on the boundary of $W_0$, and let $R_{\eta_0}$ denote the Reeb vector field on $\partial W_0$ of $\eta_0$. The backward oriented characteristic foliation of $\Pi^{PL}$ is given by Table \ref{tab:CF_PL_high}: 

\end{lemma}

\begin{table}[htb]
    \centering
    \begin{tabular}{ |c|c|c|c| } 
 \hline
 Side  & Characteristic foliation  \\ 
 \hline 
 $\underline{z=z_0}$  & $  -\partial_s$ \\
 $\underline{s=0}$  & $\partial_t -\partial_z $ \\
 $\underline{s=s_0}$  & $-\partial_t + e^{s_0} \partial_z$ \\
 $\underline{t=0}$ & $-\partial_s + X_{\lambda_0}$ \\
 $\underline{t=t_0}$ & $\partial_s - X_{\lambda_0}$ \\
 $\underline{\partial W_0}$ & $\partial_t - R_{\eta_0}$ \\
 \hline 
\end{tabular}
    \caption{The characteristic foliation of high-dimensional box fold.}
    \label{tab:CF_PL_high}
\end{table}

\begin{remark}
Observe that the foliation is identical to the low-dimensional fold on $\underline{z=z_0}$, $\underline{s=0}$, and $\underline{s=s_0}$. On the sides $\underline{t=0}$ and $\underline{t=t_0}$ there is an additional $\pm X_{\lambda_0}$ term, inducing motion in the $W_0$ direction. This motion, together with the new side $\underline{\partial W_0}$, is the key feature that distinguishes the behavior of the high-dimensional and low-dimensional folds. We also point out that the vector fields on the latter three sides project to the characteristic foliation of $\partial H_0$ in $H_0$. 
\end{remark}

\begin{proof}
We will consider the sides $\underline{t=0}$, $\underline{t=t_0}$, and $\underline{\partial W_0}$ that feature new behavior; the other three sides are similar. As before, we use Lemma \ref{cflemma}.

First, consider $\underline{t=0}$. A correctly oriented volume form on this side is 
\[
\Omega = e^{(n-1)s}\, dz\, ds\, (d\lambda_0)^{n-1}. 
\]
Let $\beta := (dz + e^s\, (dt + \lambda_0))\mid_{\underline{t=0}} = dz + e^s\, \lambda_0$. Then $d\beta = e^s\, ds\, \lambda_0 + e^s\, d\lambda_0$, and so 
\begin{align*}
\beta \, (d\beta)^{n-1} &= \left[dz + e^s\, \lambda_0\right]\, \left[(n-1)e^{(n-1)s}\, ds\, \lambda_0\, (d\lambda_0)^{n-2} + e^{(n-1)s}\,(d\lambda_0)^{n-1}\right] \\
&= e^{(n-1)s}\,\left[(n-1)\, dz\, ds\, \lambda_0\, (d\lambda_0)^{n-2} + dz\, (d\lambda_0)^{n-1}\right].
\end{align*}
Note that 
\begin{align*}
    \iota_{-\partial_s + X_{\lambda_0}}\Omega &= e^{(n-1)s}\,\left[dz\, (d\lambda_0)^{n-1} + (n-1)\, dz\, ds\, \lambda_0\, (d\lambda_0)^{n-2} \right] \\
    &= \beta \, (d\beta)^{n-1}.
\end{align*}
By Lemma \ref{cflemma}, it follows that $-\partial_s + X_{\lambda_0}$ directs the characteristic foliation of $\underline{t=0}$. The computation for $\underline{t=t_0}$ is identical, using the volume form $\Omega = -e^{(n-1)s}\, dz\, ds\, (d\lambda_0)^{n-1}.$

Finally, consider $\underline{\partial W_0}$. A correctly oriented volume form on this side is 
\[
\Omega = (n-1)e^{(n-1)s}\, dz\, ds\, dt\, \eta_0\, (d\eta_0)^{n-2}.
\]
Let $\beta := (dz + e^s\, (dt + \lambda_0))\mid_{\underline{\partial W_0}} = dz + e^s(dt + \eta_0)$. Then $d\beta = e^s\, ds\, (dt + \eta_0) + e^s\, d\eta_0$, and so 
\begin{align*}
    \beta \, (d\beta)^{n-1} &= \left[dz + e^s(dt + \eta_0)\right]\, \left[(n-1)e^{(n-1)s}\, ds\, (dt + \eta_0)\, (d\eta_0)^{n-2}\right] \\
    &= (n-1)e^{(n-1)s}\,\left[dz\, ds\, dt\, (d\eta_0)^{n-2} + dz\, ds\, \eta_0\, (d\eta_0)^{n-2}\right].
\end{align*}
Note that 
\begin{align*}
    \iota_{\partial_t - R_{\eta_0}} \Omega &= (n-1)e^{(n-1)s}\,\left[dz\, ds\, \eta_0\, (d\eta_0)^{n-2} + dz\, ds\, dt\, (d\eta_0)^{n-2}\right] \\
    &= \beta \, (d\beta)^{n-1}.
\end{align*}
By Lemma \ref{cflemma}, $\partial_t - R_{\eta_0}$ directs the characteristic foliation of $\underline{\partial W_0}$. 
\end{proof}

When visualizing the dynamics of a high-dimensional box fold, there are two useful ``contact projections'' of $\Pi^{PL}$: the projection to $[0,z_0]\times [0,s_0]\times [0,t_0]$, and the projection to $H_0 = [0,t_0]\times W_0$. See, for example, Figure \ref{fig:PL_high_dimensional_holonomy}.

\begin{lemma}\label{lemma:t_0_trap}
Let $\Pi^{PL}$ be a high-dimensional box fold (with no assumption on $z_0$). If a flowline reaches $\underline{t=t_0}$, it is trapped in backward time by the fold. 
\end{lemma}

\begin{proof}
Suppose that a flowline reaches $\underline{t=t_0}$ with $z$-coordinate $\bar{z}$, $s$-coordinate $\bar{s}$, and $W_0$-coordinate $\bar{w}$. The characteristic foliation here is directed by $\partial_s - X_{\lambda_0}$. Since $(W_0, \lambda_0)$ is a domain, the flowline travels to $\underline{s=s_0}$, concurrently moving toward $\skel(W_0,\lambda_0)$. In particular, if $\psi^s$ denote the time-$s$ flow of $X_{\lambda_0}$ on $W_0$, then the flowline reaches $\underline{s=s_0}$ with $W_0$-coordinate $\psi^{-(s_0-\bar{s})}(\bar{w})$. Here the foliation is directed by $-\partial_t + e^{s_0}\, \partial_z$. The flowline then reaches either $\underline{t=0}$ or $\underline{z=z_0}$. 

\vspace{2mm}
\noindent \textit{Case 1: the flowline reaches $\underline{t=0}$.}
\vspace{2mm}

The flowline reaches $\underline{t=0}$ with $z$-coordinate $\bar{z} + e^{s_0}t_0$. Here the foliation is directed by $-\partial_s + X_{\lambda_0}$. There are two subcases: the flowline reaches either $\underline{s=0}$ or $\underline{\partial W_0}$.

\vspace{2mm}
\noindent \textit{Case 1A: the flowline reaches $\underline{s=0}$.}
\vspace{2mm}

In this case, the flowline reaches $\underline{s=0}$ with $W_0$-coordinate $\psi^{s_0}(\psi^{-(s_0-\bar{s})}(\bar{w})) = \psi^{\bar{s}}(\bar{w})$. Note that Case 1A is characterized precisely by the fact that $\psi^{\bar{s}}(\bar{w}) \in W_0 \setminus \partial W_0$. Along $\underline{s=0}$ the flowline follows $\partial_t - \partial_z$. Since the current $z$-coordinate is $\bar{z} + e^{s_0}t_0 > t_0$, the flowline reaches $\underline{t=t_0}$ before $\underline{z=0}$, and it does so with $z$-coordinate $\bar{z} + (e^{s_0} - 1)t_0 > \bar{z}$. 

Observe that we return to the hypothesis of the lemma (reaching $\underline{t=t_0}$) with an increased $z$-coordinate and a $W_0$-coordinate which is closer to $\partial W_0$, namely, $\psi^{\bar{s}}(\bar{w})$. The flowline enters either Case 1 or Case 2. If the flowline continues to re-enter Case 1A, the $W_0$-coordinate will eventually be close enough to $\partial W_0$ to ensure that the flowline reaches $\underline{\partial W_0}$ before $\underline{s=0}$, entering Case 1B.

\vspace{2mm}
\noindent \textit{Case 1B: the flowline reaches $\underline{\partial W_0}$.}
\vspace{2mm}

Here the foliation is directed by $\partial_t - R_{\eta_0}$. The flowline travels around $\partial W_0$ via $-R_{\eta_0}$ and ultimately reaches $\underline{t=t_0}$, returning to the hypothesis of the lemma. Note that the $z$-coordinate has increased from $\bar{z}$ and the $s$-coordinate has increased from $\bar{s}$. The flowline either re-enters Case 1 or enters Case 2. If it continues to re-enter Case 1B, its $s$-coordinate will eventually increase to the point where it enters Case 1A.

\vspace{2mm}

The upshot of the Case 1A/1B analysis above is the following. If a flowline cycles through Case 1, it cannot cycle through Case 1A forever or Case 1B forever. The flowline will cycle through Case 1A sufficiently many times to reach Case 1B, where it cycles through Case 1B sufficiently many times to reach Case 1A, and so on. Both Case 1A and Case 1B contribute to a net increase in the $z$-coordinate. Thus, eventually the flowline will cycle through Case 1 sufficiently many times for the $z$-coordinate to be large enough to enter Case 2.

\vspace{2mm}
\noindent \textit{Case 2: the flowline reaches $\underline{z=z_0}$.}
\vspace{2mm}

Let $\tilde{z}$ denote the $z$-coordinate from which the flowline left $\underline{t=t_0}$ before reaching $\underline{z=z_0}$. Let $\tilde{t}, \tilde{w}$ denote the $t$-coordinate and $W_0$-coordinate at which the flowline reaches $\underline{z=z_0}$. Note that $\tilde{t} = t_0 - e^{-s_0}(z_0-\tilde{z})$. Here the foliation is directed by $-\partial_s$, and so the flowline reaches $\underline{s=0}$. It then follows $\partial_t - \partial_z$ to $\underline{t=t_0}$, attaining a $z$-coordinate of
\[
z_0 - (t_0 - \tilde{t}) = z_0 - e^{-s_0}(z_0-\tilde{z}) > \tilde{z}.
\]
The $s$-coordinate is now $0$ and the $W_0$-coordinate is still $\tilde{w}$. 

The foliation along $\underline{t=t_0}$ is directed by $\partial_s - X_{\lambda_0}$. Thus, since $(W_0, \lambda_0)$ is a domain, the flowline reaches $\underline{s=s_0}$ with $W_0$-coordinate $\psi^{-s_0}(\tilde{w})$. The flowline travels along $-\partial_t + e^{s_0}\, \partial_z$ to $\underline{z=z_0}$ with $W_0$-coordinate $\psi^{-s_0}(\tilde{w})$, re-entering Case 2.

\vspace{2mm}

The above analysis shows that when a flowline reaches Case 2, it continues to re-enter Case 2 indefinitely. The flowline does not exit the fold and spirals around $\underline{t=t_0} \cap \underline{z=z_0}$ while limiting towards $\skel(W_0, \lambda_0)$. See Figure \ref{fig:PL_high_dimensional_holonomy}. 

\end{proof}

\begin{figure}[htb]
    \centering
    \def\svgwidth{\linewidth}
    \graphicspath{{Figures/}}
    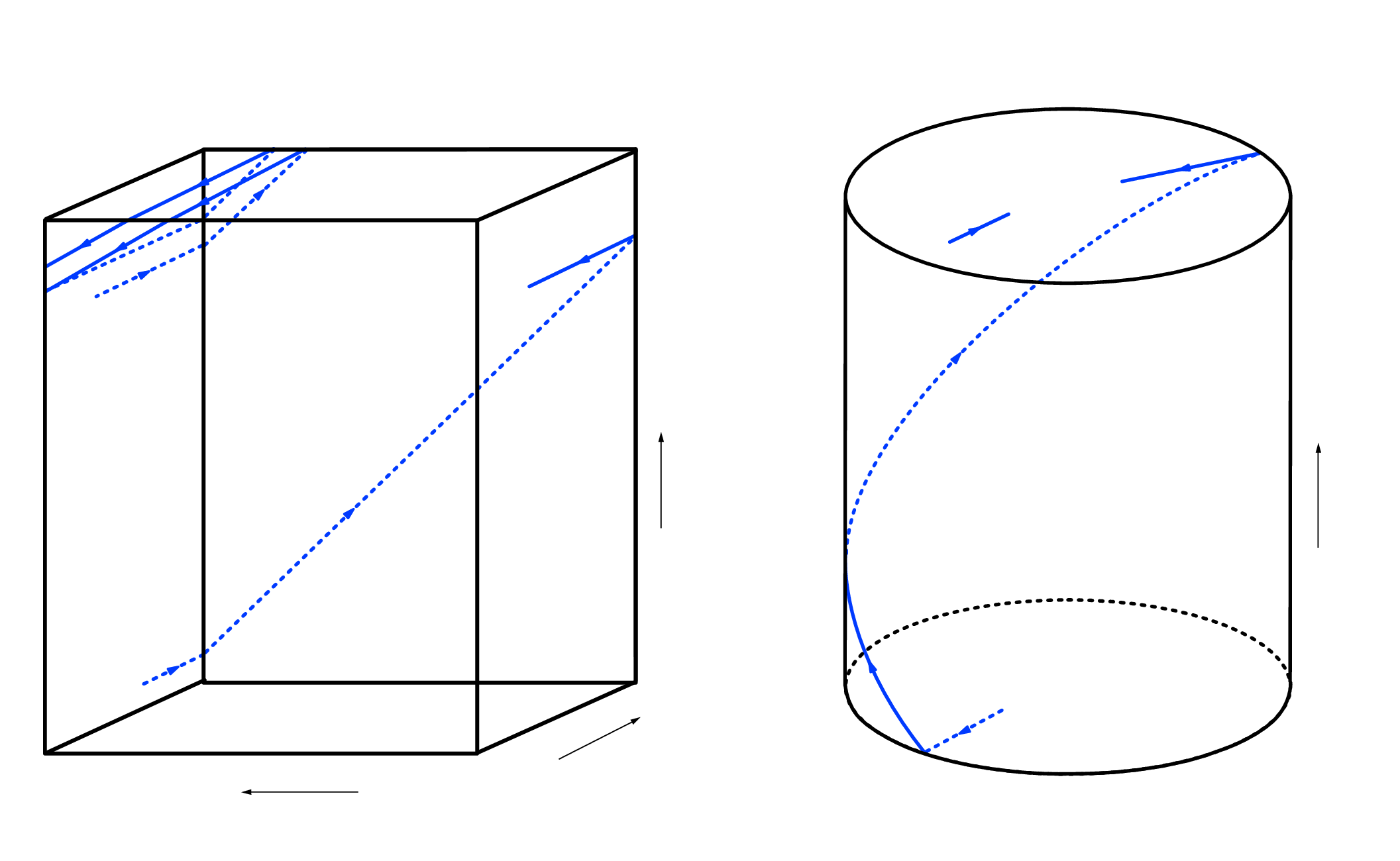
    \caption{A visualization of Lemma \ref{lemma:t_0_trap}. On the left is the $(z,s,t)$ projection, and on the right is the contact handlebody $H_0= [0,t_0]\times W_0$. The figure depicts a single sample flowline beginning on $\underline{t=t_0}$ at $x_1$. It travels along $\underline{t=t_0}$ and reaches $\underline{s=s_0}$ at $x_2$. Here the flowline is in Case 1, as it reaches $\underline{t=0}$ at $x_3$ before reaching $\underline{z=z_0}$. Then the flowline is in Case 1B, because it travels along $\underline{t=0}$ and reaches $\underline{\partial W_0}$ at $x_4$ before reaching $\underline{s=0}$. Along $\underline{\partial W_0}$ the flowline swirls around $\partial W_0$ via $-R_{\eta_0}$ and reaches $\underline{t=t_0}$ at $x_5$. From here, the flowline enters Case 2. The flowline cycles through Case 2 indefinitely, ultimately swirling around $\underline{t=t_0} \cap \underline{z=z_0}$ on the left and limiting towards $\skel(W_0, \lambda_0)$ on the right.}
    \label{fig:PL_high_dimensional_holonomy}
\end{figure}

The proof of Lemma \ref{lemma:t_0_trap} suggests the following (overly simplified) principle: when a flowline enters a high-dimensional box fold, its trajectory takes turns following the characteristic foliations in the $(z,s,t)$ and $H_0$ contact projections. 

To summarize the holonomy through a piecewise linear, high-dimensional box fold, we introduce more notation this will persist for the rest of the paper. Given a Weinstein domain $(W_0, \lambda_0)$, let $\psi^s$ denote the time-$s$ flow of the Liouville vector field $X_{\lambda_0}$. Define a distinguished collar neighborhood of $\partial W_0$ as follows: 
\[
N^{s_0}(\partial W_0) := \bigcup_{s\in (-s_0,0]} \psi^{s}(\partial W_0). 
\]
In words, $N^{s_0}(\partial W_0)$ is the set of points in $W_0$ that reach the boundary after a time-$s_0$ flow of the Liouville vector field.

The following proposition is the generalization of Lemma \ref{lemma:PL_box_fold_holonomy_low} when $z_0 \geq e^{s_0}t_0$. 

\begin{proposition}\label{prop:PL_box_fold_holonomy_high}
Let $\Pi^{PL}$ be a high-dimensional box fold with $z_0 \geq e^{s_0}t_0$, and let $h^{PL}:\{z=0\}\times \{s=s_0\} \times H_0 \dashrightarrow \{z=0\}\times \{s=0\} \times H_0$ be the partially-defined holonomy map given by the oriented characteristic foliation of $\Pi^{PL}$. Let $x\in H_0$ be the entry point of a flowline in $H_0$, and let $t(x)$ and $W_0(x)$ be the $t$-coordinate and $W_0$-coordinate of $x$, respectively. 
\begin{enumerate}
    \item If either $t(x) \in (e^{-s_0}t_0, t_0)$ or $W_0(x) \in N^{s_0}(\partial W_0)$, then the flowline through $x$ is trapped in backward time.\label{prop:PL_box_fold_holonomy_high:trap} 
    
    \item For all $(t,w)\in H_0$ in the domain of $h^{PL}$, 
    \[
    h^{PL}(t,w) = (e^{s_0}t, \psi^{s_0}(w)).
    \]\label{prop:PL_box_fold_holonomy_high:holonomy}
\end{enumerate}
\end{proposition}

\begin{proof}
We begin by proving~\ref{prop:PL_box_fold_holonomy_high:trap}.  Suppose first that $W_0(x)\in N^{s_0}(\partial W_0)$. Since $z_0 \geq e^{s_0}t_0$, the flowline enters $\Pi^{PL}$ along $\underline{s=s_0}$ and travels via $-\partial_t + e^{s_0}\, \partial_z$ to $\underline{t=0}$. Here the foliation is directed by $-\partial_s + X_{\lambda_0}$. Because $W_0(x)\in N^{s_0}(\partial W_0)$, the flowline reaches $\underline{\partial W_0}$ before reaching $\underline{s=0}$. Here it follows $\partial_t - R_{\eta_0}$ up to $\underline{t=t_0}$. By Lemma \ref{lemma:t_0_trap}, the flowline is trapped.

Next, suppose that $t(x) \in (e^{-s_0}t_0, t_0)$. We may further suppose that $W_0(x) \notin N^{s_0}(\partial W_0)$. Again by the assumption that $z_0 \geq e^{s_0}t_0$, the flowline enters $\Pi^{PL}$ and travels across $\underline{s=s_0}$ via $-\partial_t + e^{s_0}\, \partial_z$ to $\underline{t=0}$, attaining a $z$-coordinate of $e^{s_0}t(x) > t_0$. Because $W_0(x) \notin N^{s_0}(\partial W_0)$, it then follows $-\partial_s + X_{\lambda_0}$ to $\underline{s=0}$. Here it follows $\partial_t - \partial_z$. Since the $z$-coordinate upon entry to $\underline{s=0}$ exceeds $t_0$, the flowline reaches $\underline{t=t_0}$ and is trapped by Lemma \ref{lemma:t_0_trap}.

Now we prove~\ref{prop:PL_box_fold_holonomy_high:holonomy}. By~\ref{prop:PL_box_fold_holonomy_high:trap} we necessarily have $t(x) < e^{-s_0}t_0$ and $W_0(x) \notin N^{s_0}(\partial W_0)$. The flowline travels along $\underline{s=s_0}$ via $-\partial_t + e^{s_0}\, \partial_z$ to $\underline{t=0}$ where it attains a $z$-coordinate of $e^{s_0}t(x) < t_0$. Here it follows $-\partial_s + X_{\lambda_0}$. Since $W_0(x) \notin N^{s_0}(\partial W_0)$, it reaches $\underline{s=0}$ with $W_0$-coordinate $\psi^{s_0}(W_0(x))$. Here it travels along $\partial_t - \partial_z$ to $\underline{z=0}$, where it exits the fold with $t$-coordinate $e^{s_0}t(x)$ and $W_0$-coordinate $\psi^{s_0}(W_0(x))$. This proves~\ref{prop:PL_box_fold_holonomy_high:holonomy}. 
\end{proof}

\begin{remark}
A more general statement (similar to Lemma~\ref{lemma:PL_box_fold_holonomy_low}, without assuming $z_0 \geq e^{s_0}t_0$) can be obtained with more casework. We do not need such a statement.
\end{remark}

\begin{figure}[htb]
    \centering
    \def\svgwidth{\linewidth}
    \graphicspath{{Figures/}}
\begingroup%
  \makeatletter%
  \providecommand\color[2][]{%
    \errmessage{(Inkscape) Color is used for the text in Inkscape, but the package 'color.sty' is not loaded}%
    \renewcommand\color[2][]{}%
  }%
  \providecommand\transparent[1]{%
    \errmessage{(Inkscape) Transparency is used (non-zero) for the text in Inkscape, but the package 'transparent.sty' is not loaded}%
    \renewcommand\transparent[1]{}%
  }%
  \providecommand\rotatebox[2]{#2}%
  \newcommand*\fsize{\dimexpr\f@size pt\relax}%
  \newcommand*\lineheight[1]{\fontsize{\fsize}{#1\fsize}\selectfont}%
  \ifx\svgwidth\undefined%
    \setlength{\unitlength}{996.33078564bp}%
    \ifx\svgscale\undefined%
      \relax%
    \else%
      \setlength{\unitlength}{\unitlength * \real{\svgscale}}%
    \fi%
  \else%
    \setlength{\unitlength}{\svgwidth}%
  \fi%
  \global\let\svgwidth\undefined%
  \global\let\svgscale\undefined%
  \makeatother%
  \begin{picture}(1,0.4824074)%
    \lineheight{1}%
    \setlength\tabcolsep{0pt}%
    \put(0,0){\includegraphics[width=\unitlength]{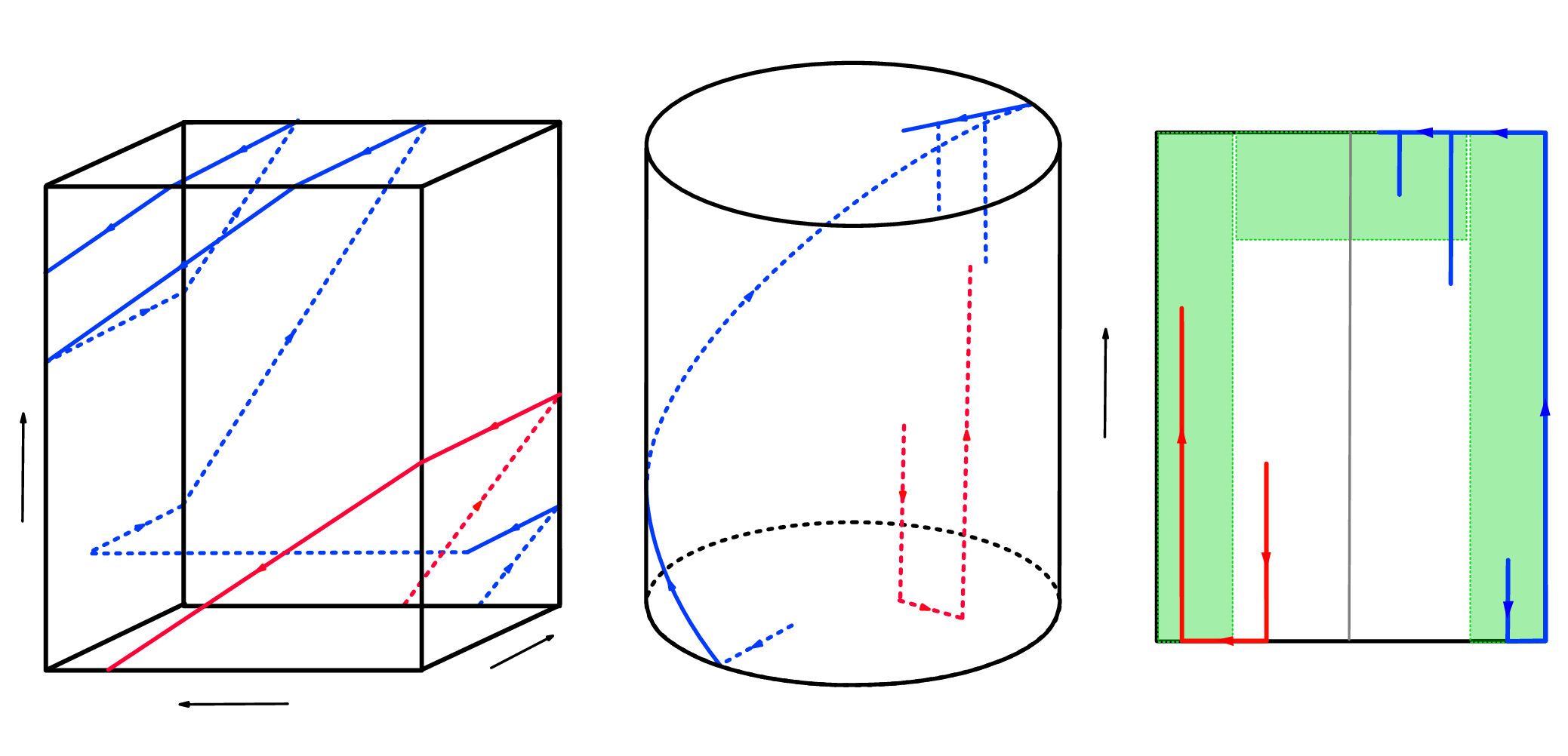}}%
    \put(0.01077623,0.22915429){\makebox(0,0)[lt]{\lineheight{1.25}\smash{\begin{tabular}[t]{l}$z$\end{tabular}}}}%
    \put(0.09710416,0.02997753){\makebox(0,0)[lt]{\lineheight{1.25}\smash{\begin{tabular}[t]{l}$t$\end{tabular}}}}%
    \put(0.35855645,0.07930577){\makebox(0,0)[lt]{\lineheight{1.25}\smash{\begin{tabular}[t]{l}$s$\end{tabular}}}}%
    \put(0.70078306,0.28156574){\makebox(0,0)[lt]{\lineheight{1.25}\smash{\begin{tabular}[t]{l}$t$\end{tabular}}}}%
    \put(0.51826112,0.01701654){\makebox(0,0)[lt]{\lineheight{1.25}\smash{\begin{tabular}[t]{l}$W_0$\end{tabular}}}}%
    \put(0.84260758,0.03799368){\makebox(0,0)[lt]{\lineheight{1.25}\smash{\begin{tabular}[t]{l}$W_0$\end{tabular}}}}%
  \end{picture}%
\endgroup%

    \caption{A visualization of Proposition~\ref{prop:PL_box_fold_holonomy_high}. In particular, the depiction of two flowlines entering the fold in various projections: on the far left is the $(z,s,t)$ projection, in the middle is the $H_0=[0,t_0]\times W_0$ projection, and the far right is a further projection of the middle picture for the sake of clarity. The shaded green regions on the far right are the trapping regions described in~\ref{prop:PL_box_fold_holonomy_high:trap}, and the gray line indicates $\skel(W_0,\lambda_0)$. The blue flowline enters the fold in $N^{s_0}(\partial W_0)$ and is ultimately trapped. The red flowline enters the fold and passes through with holonomy given by $h^{PL}$.}
    \label{fig:PL_high_dimensional_holonomy2}
\end{figure}

\subsection{Box holes}\label{subsection:box_holes}

Here we briefly discuss the notion of a \textit{box hole}. The primary function of a box fold installation as defined in \ref{subsection:PL_low} is to trap a portion of the flowlines entering the fold near the top of the Reeb chord $[0,t_0]$. Namely, the trapping region of a low-dimensional box fold (with $z_0 \geq t_0$) is $(e^{-s_0}t_0, t_0)$. In Section \ref{section:proof_main_prop}, it will be desirable to instead trap a portion of the flowlines entering a certain fold near the \textit{bottom} of the Reeb chord. This is possible by simply mirroring the installation of a box fold with a box-shaped hole.

\begin{definition}
Fix $z_0, s_0, t_0 > 0$. A \textbf{(piecewise linear, low-dimensional) box hole with parameters $z_0, s_0, t_0$}, denoted $\rotreflPi^{PL}$, is the surface 
\[
\rotreflPi^{PL} := \overline{\partial\left([-z_0,0] \times [0,s_0] \times [0,t_0]\right) \setminus \{z=0\}}.
\]
\end{definition}

\begin{figure}[htb]
    \centering
    \def\svgwidth{\linewidth}
    \graphicspath{{Figures/}}
\begingroup%
  \makeatletter%
  \providecommand\color[2][]{%
    \errmessage{(Inkscape) Color is used for the text in Inkscape, but the package 'color.sty' is not loaded}%
    \renewcommand\color[2][]{}%
  }%
  \providecommand\transparent[1]{%
    \errmessage{(Inkscape) Transparency is used (non-zero) for the text in Inkscape, but the package 'transparent.sty' is not loaded}%
    \renewcommand\transparent[1]{}%
  }%
  \providecommand\rotatebox[2]{#2}%
  \newcommand*\fsize{\dimexpr\f@size pt\relax}%
  \newcommand*\lineheight[1]{\fontsize{\fsize}{#1\fsize}\selectfont}%
  \ifx\svgwidth\undefined%
    \setlength{\unitlength}{995.94394483bp}%
    \ifx\svgscale\undefined%
      \relax%
    \else%
      \setlength{\unitlength}{\unitlength * \real{\svgscale}}%
    \fi%
  \else%
    \setlength{\unitlength}{\svgwidth}%
  \fi%
  \global\let\svgwidth\undefined%
  \global\let\svgscale\undefined%
  \makeatother%
  \begin{picture}(1,0.63936293)%
    \lineheight{1}%
    \setlength\tabcolsep{0pt}%
    \put(0,0){\includegraphics[width=\unitlength]{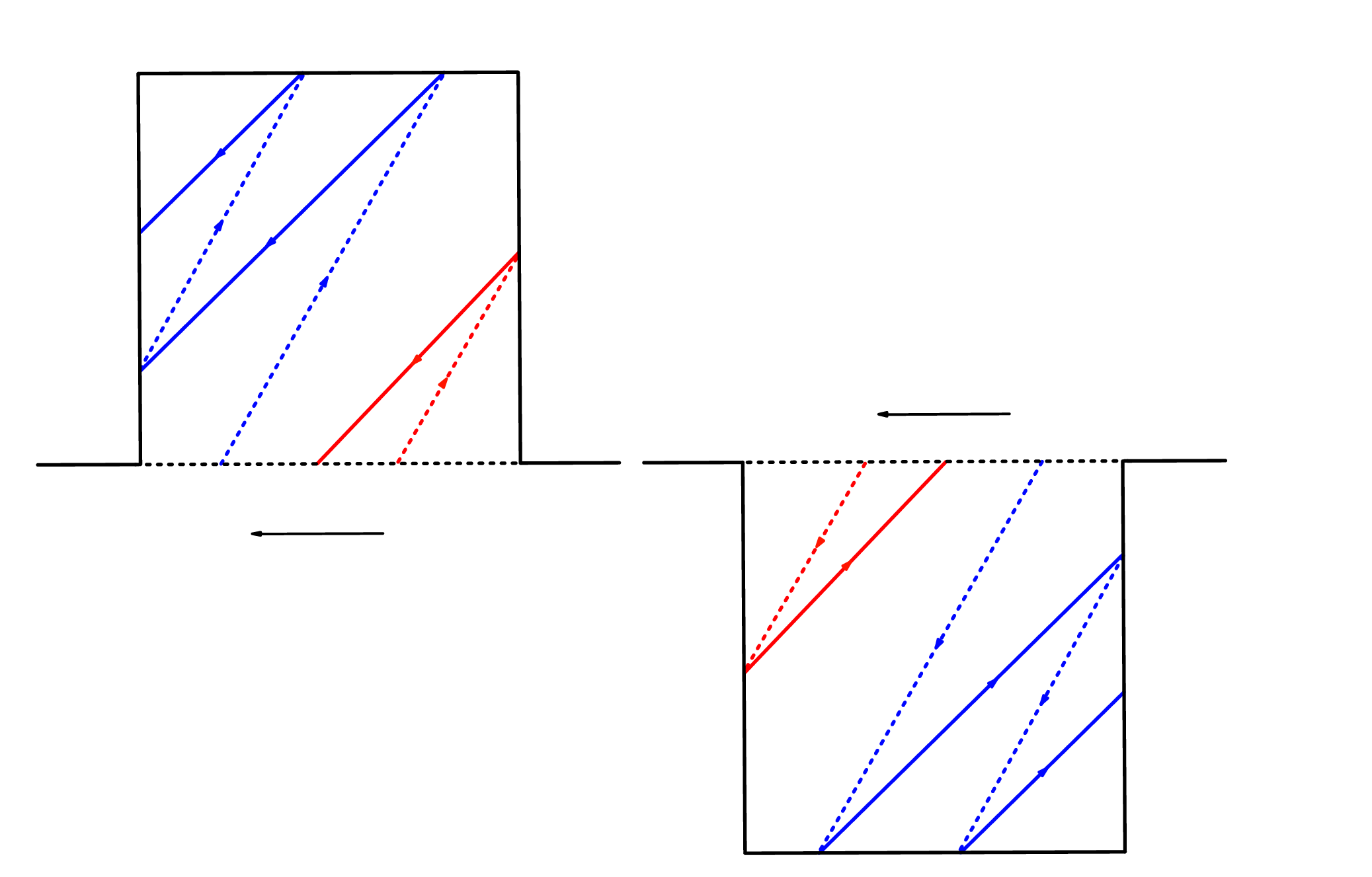}}%
    \put(0.16130984,0.24632576){\makebox(0,0)[lt]{\lineheight{1.25}\smash{\begin{tabular}[t]{l}$t$\end{tabular}}}}%
    \put(0.61954699,0.33378503){\makebox(0,0)[lt]{\lineheight{1.25}\smash{\begin{tabular}[t]{l}$t$\end{tabular}}}}%
    \put(0.39243594,0.58525154){\makebox(0,0)[lt]{\lineheight{1.25}\smash{\begin{tabular}[t]{l}$z=z_0$\end{tabular}}}}%
    \put(0.44784731,0.01464338){\makebox(0,0)[lt]{\lineheight{1.25}\smash{\begin{tabular}[t]{l}$z=-z_0$\end{tabular}}}}%
    \put(0.90124548,0.30303739){\makebox(0,0)[lt]{\lineheight{1.25}\smash{\begin{tabular}[t]{l}$z=0$\end{tabular}}}}%
  \end{picture}%
\endgroup%

    \caption{A depiction of a box fold on the left and a box hole on the right. The red flowlines pass through the folds, and the blue flowlines are trapped by the folds. Note that in a box hole, in contrast to a box fold, the flowlines entering near $t=0$ are trapped.}
    \label{fig:box_hole}
\end{figure}

By repeating the same analysis as in \ref{subsection:PL_low}, we obtain the following description of the trapping region and holonomy of a low-dimensional box hole. 

\begin{lemma}
Let $h^{PL}:\{z=0\} \times \{s=s_0\} \times [0,t_0] \dashrightarrow \{z=0\} \times \{s=s_0\} \times [0,t_0]$ be the partially-defined holonomy map given by the oriented characteristic foliation of $\rotreflPi^{PL}$. Assume that $z_0 \geq t_0$. The domain of $h^{PL}$ is $((1-e^{-s_0})t_0, t_0)$ and $h^{PL}(t) = t_0 - e^{s_0}(t - t_0).$
\end{lemma}

The extension to arbitrary dimensions is identical to \ref{subsection:PL_high}.

\subsection{Pre-chimney folds}\label{subsection:pre_chimney}

Finally, we close with an analysis of a different kind of piecewise linear fold, which we call a \textit{pre-chimney fold}. We will not ever install a pre-chimney fold as described here, but the analysis will be helpful in understanding both the function and purpose of a chimney fold in Section \ref{section:chimney_folds}.

A pre-chimney fold arises by allowing $(W_0,\lambda_0)$ to be a cobordism between manifolds with boundary. For the purpose of this discussion as it relates to this paper, we will consider the two-dimensional trivial cobordism 
\[
(W_0 = [-r_0,0]_r \times [0,\theta_0]_{\theta}, \, \lambda_0 = e^{r}\, d\theta).
\]
In practice we will identify such a region as a small neighborhood near the boundary of a Weinstein domain. 

With $(W_0, \lambda_0)$ as above, let $H_0= [0,t] \times W_0$ and consider the model 
\[
\left([0,s_0] \times H_0,\, e^s\, (dt + \lambda_0)\right).
\]
As always we then realize this piece of a Liouville domain  as the hypersurface $\{z=0\}$ inside its contactization: \[
\left([0,z_0] \times [0,s_0] \times H_0,\, dz + e^s\, (dt + \lambda_0)\right).
\]

\begin{definition}
A \textbf{(piecewise linear) pre-chimney fold}, denoted $pC\Pi^{PL}$, is the hypersurface 
\[
pC\Pi^{PL} := \overline{\partial \left([0,z_0] \times [0,s_0] \times H_0\right)\setminus \{z=0\}}.
\]
\end{definition}

The proof of Lemma \ref{lemma:PL_cf_high_dimensional} gives the backward-oriented characteristic foliation of $pC\Pi^{PL}$ in Table \ref{tab:prechimney}. 

\begin{table}[htb]
    \centering
    \begin{tabular}{ |c|c|c|c| } 
 \hline
 Side  & Characteristic foliation \\ 
 \hline 
 $\underline{z=z_0}$  & $  -\partial_s$ \\
 $\underline{s=0}$  & $\partial_t -\partial_z $ \\
 $\underline{s=s_0}$  & $-\partial_t + e^{s_0} \, \partial_z$ \\
 $\underline{t=0}$ & $-\partial_s + \partial_r$ \\
 $\underline{t=t_0}$ & $\partial_s - \partial_r$ \\
 $\underline{r=0}$ & $\partial_t -\partial_{\theta}$ \\ 
 $\underline{r=-r_0}$ & $-\partial_t +e^{r_0}\, \partial_{\theta}$ \\
 $\underline{\theta = 0}$ & $-\partial_r$ \\
 $\underline{\theta = \theta_0}$ & $\partial_r$ \\
 \hline
\end{tabular}
    \caption{The characteristic foliation of a pre-chimney fold.}
    \label{tab:prechimney}
\end{table}

Note that Table \ref{tab:prechimney} is consistent with the existing sides in the usual box fold case, where we have $X_{\lambda_0} = \partial_r$ and $R_{\eta_0} = \partial_{\theta}$.

As $(W_0, \lambda_0)$ is a trivial cobordism and does not have a skeleton, one might expect that a pre-chimney fold as defined above does not trap any flowlines in backward time. However, the dynamical analysis will reveal that the mechanism by which a fold based over a region $H_0$ traps flowlines in backward time depends on the characteristic foliation of $\partial H_0$. In particular, if $H_0$ is any contact region such that $\partial H_0$ has a characteristic foliation with (positive) critical points, it will trap some flowlines in backward time. In a pre-chimney fold, $\partial H_0$ is diffeomorphic to $S^2$ (up to corner rounding) and has a singular characteristic foliation with an index $0$ and index $2$ critical point.

\begin{lemma}\label{lemma:pre_chimney_t_0_trap}
Let $pC\Pi^{PL}$ be a pre-chimney fold with $t_0 \geq \theta_0$. If a flowline reaches $\underline{t=t_0}$, it is trapped by the fold in backward time.
\end{lemma}

\begin{proof}

Proving this lemma with backward time casework as in Lemma \ref{lemma:t_0_trap} using  Figure \ref{fig:pre_chimney} as reference is possible but tedious. Because of this, we will present a \textit{forward} time argument which is simpler, at the cost of obfuscating the nature of the trapping mechanism of $pC\Pi^{PL}$.

To exit the fold, a flowline must reach $\underline{s=0}$ with a $z$-coordinate no greater than $t_0$. In particular, it must necessarily reach $\underline{s=0}$ from $\underline{t=0}$ with a $z$-coordinate no greater than $t_0$. Consider such a flowline in \textit{forward} time, beginning on $\underline{t=0}$ with a $z$-coordinate no greater than $t_0$. We will argue that the flowline does not traverse $\underline{t=t_0}$, which will prove the lemma. 

In forward time, the flowline follows $\partial_s - \partial_r$ to either $\underline{s=s_0}$ or $\underline{r=-r_0}$.

\vspace{2mm}
\noindent \textit{Case 1: the flowline reaches $\underline{s=s_0}$ in forward time.} 
\vspace{2mm}

Here it follows $\partial_t - e^{s_0}\, \partial_z$ and subsequently reaches $\underline{s=s_0}\cap \underline{z=0}$, because the initial $z$-coordinate of the flowline was no larger than $t_0$. In this case, the flowline passes through the entire fold without traversing across $\underline{t=t_0}$.

\vspace{2mm}
\noindent \textit{Case 2: the flowline reaches $\underline{r=-r_0}$ in forward time.} 
\vspace{2mm}

Along $\underline{r=-r_0}$ the forward time foliation is directed by $\partial_t - e^{r_0}\, \partial_{\theta}$. Because $t_0 \geq \theta_0$, the flowline then reaches $\underline{\theta=0}$, where it follows $\partial_r$ to $\underline{r=0}$. Here the forward time foliation is $-\partial_t + \partial_{\theta}$, so the flowline returns to $\underline{t=0}$, and re-enters either Case 1 or Case 2. Note that the $s$-coordinate has increased from its initial value at the beginning of the proof, and the $z$-coordinate has not changed. Note also that the flowline has not traversed across $\underline{t=t_0}$.

\vspace{2mm}

The point of the above casework is that, in forward time, a flowline beginning on $\underline{t=0}$ with $z$-coordinate no larger than $t_0$ will cycle through Case 2 sufficiently many times until it reaches Case 1. Thus, any such flowline will pass through the fold without reaching $\underline{t=t_0}$, and therefore any flowline that does reach $\underline{t=t_0}$ is trapped in backward time. 
\end{proof}

It is possible to identify the trapping region of a pre-chimney fold precisely, but for the purpose of our arguments in Section \ref{section:chimney_folds} we will only need Lemma \ref{lemma:pre_chimney_t_0_trap}. For example, in a piecewise-linear pre-chimney fold with $t_0 \geq \theta_0$ and $z_0 \geq e^{s_0}t_0$, one can show that the trapping region of the fold is 
\[
\left((0,t_0)_t \times (-r_0, 0)_r \times (e^{-s_0}\theta_0, \theta_0)_{\theta} \right) \cup \left((e^{-s_0}t_0, t_0)_t \times (-r_0, 0)_r \times (0,\theta_0)_{\theta}\right) \subset H_0. 
\]
A flowline that is trapped ultimately swirls around $\underline{z=z_0} \cap \underline{t=t_0}$ in the $(z,s,t)$ projection and swirls around $\underline{t=t_0} \cap \underline{\theta = \theta_0}$ in the $H_0$ projection. See Figure \ref{fig:pre_chimney}.

\begin{figure}[htb]
    \centering
    \def\svgwidth{\linewidth}
            \graphicspath{{Figures/}}
            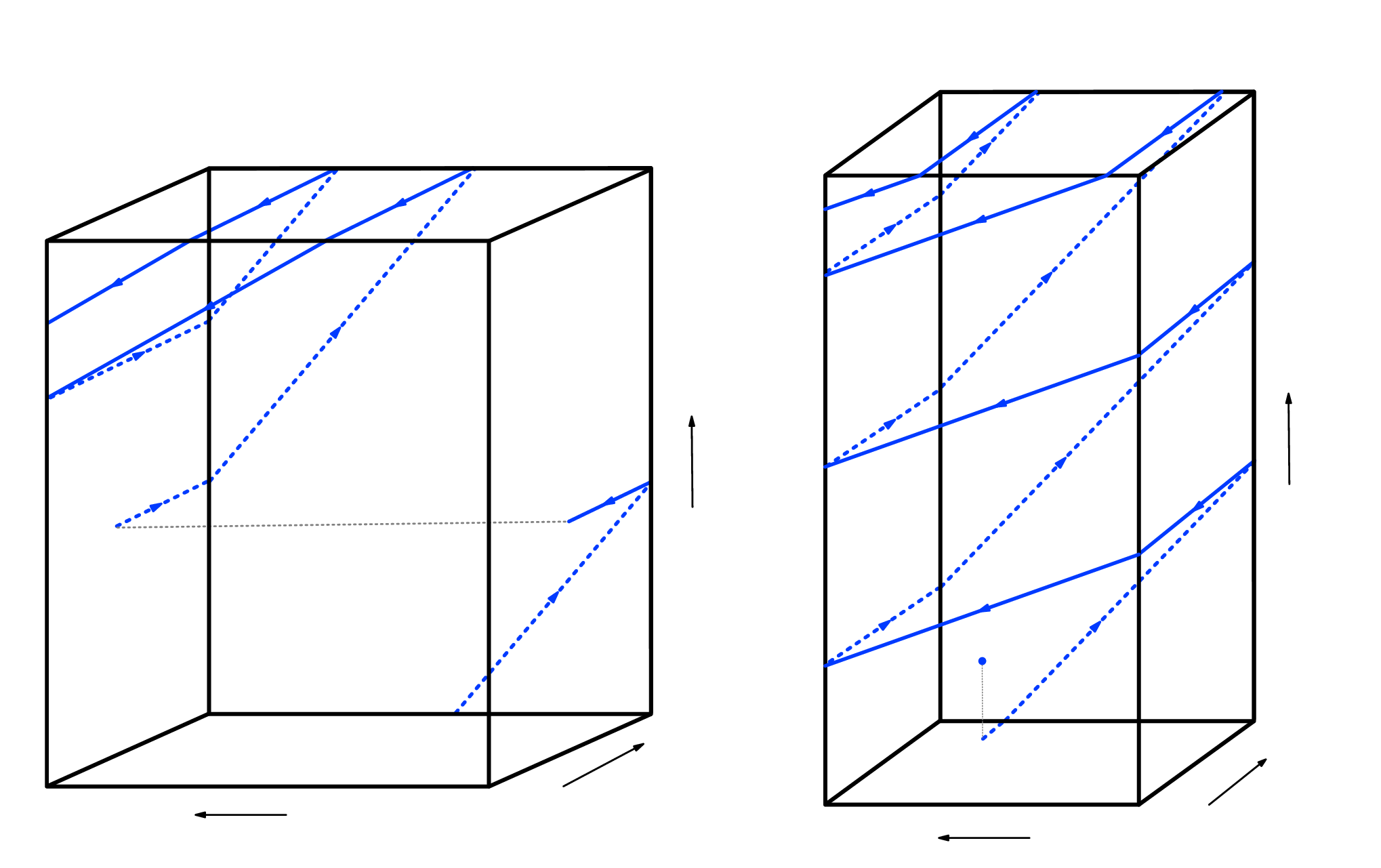
    \caption{A sample flowline that is trapped by a pre-chimney fold. On the left is the $(z,s,t)$ projection, and on the right is $H_0$ with coordinates $(t,r,\theta)$. The flowline enters the fold at $x_1$, travels to $x_2 \in \underline{t=0} \cap \underline{s=s_0}$, and then follows the characteristic foliation of $\partial H_0$ all the way to $x_3\in \underline{t=t_0} \cap \underline{r=0}$. The flowline essentially follows the characteristic foliations of both contact projections, ultimately swirling around $\underline{t=t_0} \cap \underline{z=z_0}$ on the left and $\underline{t=t_0}\cap \underline{\theta=\theta_0}$ on the right.}
    \label{fig:pre_chimney}
\end{figure}

\section{Smooth box folds}\label{section:box_folds_smooth}

In this section we construct the smooth box fold, which is a smooth, graphical approximation of the piecewise linear hypersurface constructed in Section \ref{section:box_folds_pl}. Our goal is to prove the following theorem, which summarizes all of the properties of such folds that we need for the duration of the paper. 

\begin{theorem}[Existence and behavior of smooth box folds]\label{theorem:new_smooth_box_fold}
Let $(W_0, \lambda_0)$ be a Weinstein domain of dimension $2n-2 \geq 2$. For any $s_0, t_0 > 0$, let $V_0 := [0,s_0]\times [0,t_0] \times W_0$. For any smooth function $F:V_0 \to \R$, let $\lambda_{F} := dF + e^s\, (dt + \lambda_0)$ denote a Liouville form on $V_0$, let $X_{\lambda_F} = \partial_s + X_{F}$ be its Liouville vector field, and let 
\[
h_{F}:\{s=s_0\} \times [0,t_0] \times W_0 \dashrightarrow \{s=0\} \times [0,t_0] \times W_0
\]
be the partially-defined holonomy map given by backward flow of $X_{\lambda_F}$. 

Fix $0 < \epsilon \ll \min(1, s_0, t_0)$. There is a smooth function $G_{\epsilon}:V_0 \to [0,\infty)$, compactly supported in the interior of $V_0$, with the following properties. 

\begin{enumerate}
    \item (Weinstein compatibility) \label{property:high_Weinstein_compat}
    
    \noindent The Liouville vector field $X_{\lambda_{G_{\epsilon}}}$ is Morse, with critical points of index $k$ and $k+1$ for each critical point of index $k$ in the underlying Weinstein domain $(W_0, \lambda_0)$.

    \item (Trapping properties)\label{property:high_trap2} 
    
    \noindent Let $(t_{\mathrm{init}}, p_{\mathrm{init}}) \in \{s=s_0\}$ be the initial point of a flowline of $X_{\lambda_{G_{\epsilon}}}$. Define subsets of $\{s=s_0\}$ as follows: 
    \begin{align*}
        U_{\mathrm{trap},1}^{\epsilon} &:= [e^{-s_0}t_0 + \epsilon, t_0 - \epsilon] \,\times\, (W_0 \setminus N^{s_0+\epsilon}(\partial W_0)), \\
        U_{\mathrm{trap},2}^{\epsilon} &:=[\epsilon, t_0-\epsilon] \,\times\, \left(N^{s_0 - 2\epsilon}(\partial W_0)\setminus N^{\epsilon}(\partial W_0)\right).
    \end{align*}
    If $(t_{\mathrm{init}}, p_{\mathrm{init}}) \in U_{\mathrm{trap},1}^{\epsilon} \cup U_{\mathrm{trap},2}^{\epsilon}$ then the flowline converges to a critical point of $X_{\lambda_{F_{\epsilon}}}$ in backward time.

    \item (Holonomy properties)\label{property:high_hol3}  
    
    \noindent Let $(t_{\mathrm{init}}, p_{\mathrm{init}}) \in \{s=s_0\}$ be the initial point of a flowline of $X_{\lambda_{G_{\epsilon}}}$. Assume $(t_{\mathrm{init}}, p_{\mathrm{init}})$ is in the domain of $h_{G_{\epsilon}}$. Let $(t_{\mathrm{fin}}, p_{\mathrm{fin}}) := h_{G_{\epsilon}}(t_{\mathrm{init}}, p_{\mathrm{init}}) \in \{s=0\}$ denote the exit point of the flowline in backward time. 
    \begin{enumerate}
        \item The estimate $\norm{p_{\mathrm{fin}}}_{W_0}\leq e^{s_0}\norm{p_{\mathrm{init}}}_{W_0}$ holds.\label{thm-part:new_smooth_holonomy}
    
        \item If $t_{\mathrm{init}} \in [t_0 - \epsilon, t_0]$, then $t_{\mathrm{fin}}\in [t_0 - \epsilon, t_0]$.\label{property:high_hol3b} 

        \item If $t_{\mathrm{init}} \in [0,\epsilon]$ and $t_{\mathrm{fin}} \gg e^{s_0 + \epsilon}\epsilon$, then $p_{\mathrm{init}}\in N^{s_0+\epsilon}(\partial W_0)$ and hence we may identify 
        \[
        p_{\mathrm{init}} = (r_{\mathrm{init}}, q_{\mathrm{init}}) \in \left((-\infty, 0]_r \times \partial W_0, \, \lambda_0 = e^r\, \eta_0\right)
        \]
        where $\eta_0$ is the contact form induced on $\partial W_0$ by $\lambda_0$. Moreover, we have 
        \[
        p_{\mathrm{fin}} = (r_{\mathrm{fin}}, R_{\eta_0}^{-t_{*}}(q_{\mathrm{init}}))
        \]
        where $R_{\eta_0}^{t}: \partial W_0 \to \partial W_0$ is the time-$t$ flow of the Reeb vector field $R_{\eta_0}$ and $t_{*}$ satisfies $\epsilon \ll t_{*} \leq e^{\epsilon}t_0$.\label{property:high_hol3c}  
        
    \end{enumerate}
\end{enumerate}
\end{theorem}

Despite being an unsurprising smooth generalization of the behavior of piecewise linear box folds (for example, sending $\epsilon \to 0$ in \ref{property:high_trap2} gives the exact piecewise linear trapping regions), this is not a priori clear and a careful proof of Theorem \ref{theorem:new_smooth_box_fold} involves some technical work. Thus, we outline the progression of the section below. 

\begin{itemize}
    \item In \ref{subsec:smooth2d} we define the smooth box fold in dimension $2$. This is done in two steps. First, we ``tip'' the sides of the piecewise linear box fold to obtain a graphical surface, and then we smooth all corners and edges via convolution.  

    \item In \ref{subsec:nearlysmooth} we consider the \textit{nearly smooth} box fold, which is a non-graphical, piecewise smooth hypersurface in a higher-dimensional model. It has one vertical side and one codimension-$1$ corner. 
    
    \item In \ref{subsec:smooth}, in analogy to \ref{subsec:smooth2d}, we ``tip'' the vertical side of the nearly smooth box fold and smooth the corners with a convolution. This produces the smooth, graphical box fold featured in the statement of Theorem \ref{theorem:new_smooth_box_fold}.

    \item Along the way in \ref{subsec:smooth2d}, \ref{subsec:nearlysmooth}, and \ref{subsec:smooth}, we state and prove various technical aspects of the dynamics of the auxiliary folds. In \ref{subsec:smooth_proof} we tie all of this together and prove Theorem \ref{theorem:new_smooth_box_fold}.
\end{itemize}

\begin{remark}[Box holes]
    Theorem \ref{theorem:new_smooth_box_fold} concerns box folds. By considering the function $-G_{\epsilon}$, one may also record a version for smooth box holes, which is nearly identical save for the natural adjustments according to the behavior described in \ref{subsection:box_holes}. We spare the reader a repetition of the adjusted theorem. In the future we will freely reference Theorem \ref{theorem:new_smooth_box_fold} when installing box holes, with the implicit understanding that we are applying the theorem statement with its modified conclusions to $-G_{\epsilon}$. 
\end{remark}

\begin{remark}[Parameter precision]
    In the statement of Theorem \ref{theorem:new_smooth_box_fold}, it is possible to give sharper and stronger bounds on the trapping regions and holonomy properties than what is stated in \ref{property:high_trap2} and \ref{property:high_hol3}, for instance by introducing additional levels of parameter dependency. This will be most clear from the work in \ref{subsec:smooth2d}. For our applications, the statements depending on the lone $\epsilon$ parameter, while generously far from being sharp, are sufficient. We have opted to state Theorem \ref{theorem:new_smooth_box_fold} and associated results like Theorem \ref{prop:mainprop}, Proposition \ref{prop:2D_smooth_box_fold}, and Proposition \ref{prop:new_chimney_smooth} in this way for the sake of clarity and ease of use. 
\end{remark}

\subsection{Smooth box folds in dimension $2$}\label{subsec:smooth2d}

In this subsection we prove the following proposition, which is the $2$-dimensional version of Theorem~\ref{theorem:new_smooth_box_fold}. 

\begin{proposition}\label{prop:2D_smooth_box_fold}
Fix $s_0, t_0 > 0$. For any smooth function $F:[0,s_0]\times [0,t_0] \to \R$, let $\lambda_{F} := dF + e^s\, dt$ denote a Liouville form on $[0,s_0]\times [0,t_0]$, let $X_{\lambda_F} = \partial_s + X_{F}$ be its Liouville vector field, and let 
\[
h_{F}:\{s=s_0\} \times [0,t_0] \dashrightarrow \{s=0\} \times [0,t_0] 
\]
be the partially-defined holonomy map given by backward flow of $X_{\lambda_F}$. 

Fix a general smoothing threshold parameter $0 < \epsilon \ll \min(1, s_0, t_0)$.  There is a smooth function $F_{\epsilon}:\mathbb{R}^2 \to [0,\infty)$, supported in the interior of $[0,s_0]\times [0,t_0]$, along with positive auxiliary parameters $0 < \delta_{\mathrm{con}} <\delta_{\mathrm{tip}}^2 \ll \epsilon^2$ that may be chosen arbitrarily small, with the following properties.           

\begin{enumerate}
    \item (Weinstein compatibility)\label{property:2d_Weinstein_compat}
    
    \noindent The Liouville vector field $X_{\lambda_{F_{\epsilon}}}$ is Morse, with one critical point of index $0$ and one critical point of index $1$. 

    \item (Trapping properties)

    \noindent Let $(s_{\mathrm{init}}, t_{\mathrm{init}})$ be the initial point of a flowline of $X_{\lambda_{F_{\epsilon}}}$. 

    \begin{enumerate}
        \item If $(s_{\mathrm{init}}, t_{\mathrm{init}}) \in \{s=s_0\} \times [e^{-s_0}t_0 + \epsilon, t_0 - \epsilon]$, then the flowline converges to a critical point of $X_{\lambda_{F_{\epsilon}}}$ in backward time. \label{property:2D_trapping_interval}

        \item Define a region $\tilde{D}_4 \subseteq [0,s_0]\times [0,t_0]$ by
        \[
        \tilde{D}_4 := \{\delta_{\mathrm{tip}} + \delta_{\mathrm{con}}\leq s \leq s_0 - (\delta_{\mathrm{tip}} + \delta_{\mathrm{con}})\} \cap \{t_0 - (\delta_{\mathrm{tip}} - \delta_{\mathrm{con}}) \leq t \leq t_0 - (\delta_{\mathrm{tip}}^2 + \delta_{\mathrm{con}})\}.
        \]
        If $(s_{\mathrm{init}}, t_{\mathrm{init}}) \in \tilde{D}_4$, then the flowline converges to a critical point of $X_{\lambda_{F_{\epsilon}}}$ in backward time. \label{property:2D_trapping_region}
    \end{enumerate}

    \item (Holonomy properties)

    \noindent Let $(s_{\mathrm{init}}, t_{\mathrm{init}}) \in \{s=s_0\}\times[0,t_0]$ be the initial point of a flowline of $X_{\lambda_{F_{\epsilon}}}$. Assume $t_{\mathrm{init}}$ is in the domain of $h_{F_{\epsilon}}$. Let $t_{\mathrm{fin}} := h_{F_{\epsilon}}(t_{\mathrm{init}})$ denote the exit point of the flowline in backward time. 

    \begin{enumerate}
        \item The estimate $t_{\mathrm{fin}} \leq e^{s_0+\epsilon}t_{\mathrm{init}}$ holds. Moreover, for sufficiently small $\epsilon$ and $t_{\mathrm{init}} \in [0,\epsilon]$, this estimate holds for the function $F^{\tau}_{\epsilon}:=\tau F_{\epsilon}$, for any $0\leq \tau \leq 1$. \label{property:2D_holonomy_bound}
        \item If $t_{\mathrm{init}} \in [t_0 - \epsilon, t_0]$, then $t_{\mathrm{fin}}\in [t_0 - \epsilon, t_0]$.\label{property:2D_holonomy_top}
        \item Define a region $\tilde{D}_2 \subseteq [0,s_0]\times [0,t_0]$ by
        \[
        \tilde{D}_2 := \{\delta_{\mathrm{tip}} + \delta_{\mathrm{con}}\leq s \leq s_0 - (\delta_{\mathrm{tip}} + \delta_{\mathrm{con}})\} \cap \{ \delta_{\mathrm{tip}}^2 + \delta_{\mathrm{con}} \leq t\leq \delta_{\mathrm{tip}} - \delta_{\mathrm{con}}\}.
        \]
        Then in $\tilde{D}_2$, $X_{\lambda_{F_{\epsilon}}}$  is positively parallel to $\partial_s$. Moreover, if $t_{\mathrm{init}}\in [\epsilon,t_0 - \epsilon]$, then in backward time the flowline enters $\tilde{D}_2$ along $\{s=s_0 - (\delta_{\mathrm{tip}} + \delta_{\mathrm{con}})\}$ and exits along $\{s=\delta_{\mathrm{tip}} + \delta_{\mathrm{con}}\}$. \label{property:2D_holonomy_region}
    \end{enumerate}
\end{enumerate}
\end{proposition}

Throughout this subsection we will consider fixed parameters $s_0,t_0,z_0>0$, with $z_0\geq t_0$, as well as $0 < \epsilon \ll \min(1,s_0,t_0)$.

\subsubsection{Graphical approximation}

\begin{figure}
\centering
\begin{tikzpicture}[scale=0.35]
\draw (-1,-1) -- (25,-1) -- (25,17) -- (-1,17) -- (-1,-1);
\draw (0,0) -- (24,0) -- (24,16) -- (0,16) -- (0,0);
\draw (3,3) -- (21,3) -- (21,13) -- (3,13) -- (3,3);

\draw (0,0) -- (3,3);
\draw (24,0) -- (21,3);
\draw (24,16) -- (21,13);
\draw (0,16) -- (3,13);

\draw[thick,->] (26,-2) -- (-2,-2);
\draw[thick,->] (26,-2) -- (26,18);

\node at (-2.5,-2) {$t$};
\node at (26,18.5) {$s$};

\draw[dashed] (25,-1) -- (25,-2.5);
\draw[dashed] (24,-1) -- (24,-2.5);
\draw[dashed] (21,-1) -- (21,-2.5);
\draw[dashed] (3,-1) -- (3,-2.5);
\draw[dashed] (0,-1) -- (0,-3.5);
\draw[dashed] (-1,-1) -- (-1,-2.5);
{\small
\node at (25,-3) {$0$};
\node at (24,-3) {$\delta_{\mathrm{tip}}^2$};
\node at (21,-3) {$\delta_{\mathrm{tip}}$};
\node at (3,-3) {$t_0-\delta_{\mathrm{tip}}$};
\node at (0.5,-4) {$t_0-\delta_{\mathrm{tip}}^2$};
\node at (-1,-3) {$t_0$};
}

\draw[dashed] (25,-1) -- (26.5,-1);
\draw[dashed] (25,0) -- (26.5,0);
\draw[dashed] (25,3) -- (26.5,3);
\draw[dashed] (25,13) -- (26.5,13);
\draw[dashed] (25,16) -- (26.5,16);
\draw[dashed] (25,17) -- (26.5,17);
{\small
\node at (27,-1) {$0$};
\node at (27.5,0) {$\delta_{\mathrm{tip}}^2$};
\node at (27.5,3) {$\delta_{\mathrm{tip}}$};
\node at (28.5,13) {$s_0-\delta_{\mathrm{tip}}$};
\node at (28.5,16) {$s_0-\delta_{\mathrm{tip}}^2$};
\node at (27.25,17) {$s_0$};
}

\node at (12,8) {$D^{PL}_1$};
\node at (22.5,8) {$D^{PL}_2$};
\node at (12,1.5) {$D^{PL}_3$};
\node at (1.5,8) {$D^{PL}_4$};
\node at (12,14.5) {$D^{PL}_5$};
\end{tikzpicture}
\caption{The regions used in the definition of $F^{PL}_{\delta_{\mathrm{tip}}}$.}
\label{fig:pl-regions}
\end{figure}

Our first step towards Proposition~\ref{prop:2D_smooth_box_fold} is to construct, for any $0<\delta_{\mathrm{tip}}\ll\epsilon$, a piecewise linear function $F^{PL}_{\delta_{\mathrm{tip}}}\colon\mathbb{R}^2\to[0,z_0]$ supported on $[\delta_{\mathrm{tip}}^2,s_0-\delta_{\mathrm{tip}}^2]\times[\delta_{\mathrm{tip}}^2,t_0-\delta_{\mathrm{tip}}^2]$ which satisfies the trapping and holonomy conditions of Proposition~\ref{prop:2D_smooth_box_fold}.  To this end, we consider the following partition of the plane $\mathbb{R}^2_{s,t}$:
\begin{align*}
D^{PL}_1 &= [\delta_{\mathrm{tip}},s_0-\delta_{\mathrm{tip}}]\times[\delta_{\mathrm{tip}},t_0-\delta_{\mathrm{tip}}]\\
D^{PL}_2 &= \{(s,t)\,|\,t\leq s\leq s_0-t \text{~and~} t\in[\delta_{\mathrm{tip}}^2,\delta_{\mathrm{tip}}]\}\\
D^{PL}_3 &= \{(s,t)\,|\,s\leq t\leq t_0-s \text{~and~} s\in[\delta_{\mathrm{tip}}^2,\delta_{\mathrm{tip}}]\}\\
D^{PL}_4 &= \{(s,t)\,|\,t_0-t\leq s\leq s_0-(t_0-t) \text{~and~} t\in[t_0-\delta_{\mathrm{tip}},t_0-\delta_{\mathrm{tip}}^2]\}\\
D^{PL}_5 &= \{(s,t)\,|\,s_0-s\leq t\leq t_0-(s_0-s) \text{~and~} s\in[s_0-\delta_{\mathrm{tip}},s_0-\delta_{\mathrm{tip}}^2]\}\\
D^{PL}_6 &= \mathbb{R}^2\setminus([\delta_{\mathrm{tip}}^2,s_0-\delta_{\mathrm{tip}}^2]\times[\delta_{\mathrm{tip}}^2,t_0-\delta_{\mathrm{tip}}^2]).
\end{align*}
See Figure~\ref{fig:pl-regions}.  Using this partition, we may define $F^{PL}_{\delta_{\mathrm{tip}}}\colon\mathbb{R}^2\to[0,z_0]$ by
\[
F^{PL}_{\delta_{\mathrm{tip}}}(s,t) := \left\{\begin{matrix}
    z_0, & (s,t)\in D^{PL}_1\\
    \tfrac{z_0}{\delta_{\mathrm{tip}}(1-\delta_{\mathrm{tip}})}(t-\delta_{\mathrm{tip}}^2), & (s,t)\in D^{PL}_2\\
    \tfrac{z_0}{\delta_{\mathrm{tip}}(1-\delta_{\mathrm{tip}})}(s-\delta_{\mathrm{tip}}^2), & (s,t)\in D^{PL}_3\\
    \tfrac{z_0}{\delta_{\mathrm{tip}}(\delta_{\mathrm{tip}}-1)}(t-(t_0-\delta_{\mathrm{tip}}^2)), & (s,t)\in D^{PL}_4\\
    \tfrac{z_0}{\delta_{\mathrm{tip}}(\delta_{\mathrm{tip}}-1)}(s-(s_0-\delta_{\mathrm{tip}}^2)), & (s,t)\in D^{PL}_5\\
    0, & (s,t)\in D^{PL}_6
\end{matrix}\right. .
\]
Notice that the graph of $F^{PL}_{\delta_{\mathrm{tip}}}$ converges to the piecewise linear box fold as $\delta_{\mathrm{tip}}$ tends to $0$.  Though undefined where the regions of our partition intersect, the Liouville 1-form $\lambda_{\mathrm{tip}}=dF^{PL}_{\delta_{\mathrm{tip}}}+e^s\,dt$ is defined on the interior of each $D^{PL}_j$, and we may compute the resulting Liouville vector field:

\vspace{2mm}
\begin{center}
\begin{tabular}{ |c|c|c|c| } 
 \hline
 Region & $dF^{PL}_{\delta_{\mathrm{tip}}} + e^s\, dt$ & $X_{\lambda_{\mathrm{tip}}}$ \rm \\ 
 \hline 
 $D^{PL}_1$ & $e^s\, dt$ & $\partial_s$ \\
 $D^{PL}_2$ & $(e^s+\tfrac{z_0}{\delta_{\mathrm{tip}}(1-\delta_{\mathrm{tip}})})\, dt$ & $(1+e^{-s}\tfrac{z_0}{\delta_{\mathrm{tip}}(1-\delta_{\mathrm{tip}})})\partial_s$ \\
 $D^{PL}_3$ & $\tfrac{z_0}{\delta_{\mathrm{tip}}(1-\delta_{\mathrm{tip}})}\, ds + e^s\, dt$ & $\partial_s - e^{-s}\tfrac{z_0}{\delta_{\mathrm{tip}}(1-\delta_{\mathrm{tip}})} \partial_t$ \\
 $D^{PL}_4$ & $(e^s+\tfrac{z_0}{\delta_{\mathrm{tip}}(\delta_{\mathrm{tip}}-1)})\, dt$ & $(1-e^{-s}\tfrac{z_0}{\delta_{\mathrm{tip}}(1-\delta_{\mathrm{tip}})})\partial_s$ \\
 $D^{PL}_5$ & $\tfrac{z_0}{\delta_{\mathrm{tip}}(\delta_{\mathrm{tip}}-1)}\, ds + e^s\, dt$ & $\partial_s + e^{-s}\tfrac{z_0}{\delta_{\mathrm{tip}}(1-\delta_{\mathrm{tip}})} \partial_t$ \\
 $D^{PL}_6$ & $e^s\,dt$ & $\partial_s$ \\
 \hline
\end{tabular}
\end{center}
\vspace{2mm}

While $X_{\lambda_{\mathrm{tip}}}$ is not piecewise linear, it is very nearly so.  In particular, $X_{\lambda_{\mathrm{tip}}}$ is parallel to a linear foliation of $\mathbb{R}^2$ on $\mathbb{R}^2\setminus(D^{PL}_3\cup D^{PL}_5)$, and on $D^{PL}_3$ and $D^{PL}_5$ we have very tight bounds for the slope of $X_{\lambda_{\mathrm{tip}}}$:
\begin{equation}\label{eq:2D-graphical-linear-bounds}
\begin{aligned}
D^{PL}_3: & -\frac{e^{-\delta_{\mathrm{tip}}^2}z_0}{\delta_{\mathrm{tip}}(1-\delta_{\mathrm{tip}})} < -\frac{e^{-s}z_0}{\delta_{\mathrm{tip}}(1-\delta_{\mathrm{tip}})} < -\frac{e^{-\delta_{\mathrm{tip}}}z_0}{\delta_{\mathrm{tip}}(1-\delta_{\mathrm{tip}})}\\
D^{PL}_5: & \frac{e^{-(s_0-\delta_{\mathrm{tip}}^2)}z_0}{\delta_{\mathrm{tip}}(1-\delta_{\mathrm{tip}})} < \frac{e^{-s}z_0}{\delta_{\mathrm{tip}}(1-\delta_{\mathrm{tip}})} < \frac{e^{-(s_0-\delta_{\mathrm{tip}})}z_0}{\delta_{\mathrm{tip}}(1-\delta_{\mathrm{tip}})}.
\end{aligned}
\end{equation}
An analysis such as was conducted in Section~\ref{section:box_folds_pl} then yields the following description of the dynamics of $X_{\lambda_{\mathrm{tip}}}$ (c.f. Lemma~\ref{lemma:PL_box_fold_holonomy_low}).

\begin{lemma}\label{lemma:graphical_box_fold_holonomy}
Suppose that $s_0,t_0,z_0,\delta_{\mathrm{tip}}>0$ are fixed, with $z_0\geq t_0$.  Let
\[
h^{PL}_{\delta_{\mathrm{tip}}}:\{s=s_0\} \times [0,t_0] \dashrightarrow \{s=0\}\times [0,t_0]
\]
denote the partially-defined holonomy map given by backwards flow of the piecewise-defined Liouville vector field of $\lambda_{\mathrm{tip}}=dF^{PL}_{\delta_{\mathrm{tip}}}+e^s\,dt$.  If $\delta_{\mathrm{tip}}$ is sufficiently small, then there is a value $\tilde{t}$ satisfying
\begin{equation}\label{ineq:2D_trapping_interval}
e^{2\delta_{\mathrm{tip}}^2-s_0}(t_0-\delta_{\mathrm{tip}}^2-\delta_{\mathrm{tip}})+\delta_{\mathrm{tip}}^2
\leq \tilde{t} \leq
e^{\delta_{\mathrm{tip}}+\delta_{\mathrm{tip}}^2-s_0}(t_0-2\delta_{\mathrm{tip}}^2)+\delta_{\mathrm{tip}}
\end{equation}
and a smooth function $\tilde{s}\colon[\delta_{\mathrm{tip}}^2,\tilde{t}]\to[\delta_{\mathrm{tip}}^2,\delta_{\mathrm{tip}}]$ satisfying $\tilde{s}(\delta_{\mathrm{tip}}^2)=\delta_{\mathrm{tip}}^2$ such that
\begin{equation}\label{eq:2D_graphical_holonomy}
h^{PL}_{\delta_{\mathrm{tip}}}(t) = \begin{cases}
    t, & t\in (0,\delta_{\mathrm{tip}}^2) \\
    e^{s_0-\tilde{s}(t)-\delta_{\mathrm{tip}}^2}(t-\tilde{s}(t))+\tilde{s}(t), & t\in (\delta_{\mathrm{tip}}^2,\tilde{t}) \\
    t, & t\in (t_0-\delta_{\mathrm{tip}}^2,t_0)
\end{cases}.
\end{equation}
Flowlines which pass through $(s_0,t^*)$, with $t^*\in(\tilde{t},t_0-\delta_{\mathrm{tip}}^2)$, do not leave $[0,s_0]\times[0,t_0]$ in backward time.
\end{lemma}
\begin{proof}
Analogous to the proof of Lemma~\ref{lemma:PL_box_fold_holonomy_low}, we proceed by investigating the flow within each region $D^{PL}_j$.  For sufficiently small $\delta_{\mathrm{tip}}>0$, the slope bounds given by Equation~\ref{eq:2D-graphical-linear-bounds} ensure that points in the domain of $h^{PL}_{\delta_{\mathrm{tip}}}$ lie on flowlines which pass through $D^{PL}_2$.  The function $\tilde{s}(t)$ gives the $s$-coordinate of this flowline as it passes through $D^{PL}_2\cap D^{PL}_3$, and from this we may compute $h^{PL}_{\delta_{\mathrm{tip}}}(t)$.  Because $h^{PL}_{\delta_{\mathrm{tip}}}$ depends on $\tilde{s}(t)$, the bounds on $\tilde{t}$ then follow from those on $\tilde{s}(t)$.
\end{proof}

\begin{remark}
Observe that as $\delta_{\mathrm{tip}}$ tends to $0$, $h^{PL}_{\delta_{\mathrm{tip}}}$ converges to $h^{PL}(t)=e^{s_0}t$, with domain $(0,e^{-s_0}t_0)$, as predicted by Lemma~\ref{lemma:PL_box_fold_holonomy_low}.
\end{remark}

\begin{remark}\label{remark:small-z0}
The hypothesis in Lemma~\ref{lemma:graphical_box_fold_holonomy} that $z_0\geq t_0$ is for convenience only, to avoid the complicated statement of $h^{PL}_{\delta_{\mathrm{tip}}}$ indicated by Lemma~\ref{lemma:PL_box_fold_holonomy_low}.  The same analysis can be carried out in case $z_0<t_0$, resulting in a subinterval of $[0,t_0]$ where backward-time holonomy approximates the map $t\mapsto t+(1-e^{-s_0}z_0)$.  However, this holonomy will only prevail where $t\geq e^{-s_0}z_0$, meaning that the estimate $t_{\mathrm{fin}} \leq e^{s_0+\epsilon}t_{\mathrm{init}}$ in part~\ref{property:2D_holonomy_bound} of Proposition~\ref{prop:2D_smooth_box_fold} continues to hold.
\end{remark}

\begin{corollary}\label{cor:2D_PL_properties}
For sufficiently small $\delta_{\mathrm{tip}}$, $\delta_{\mathrm{con}}$ satisfying $0 < \delta_{\mathrm{con}} <\delta_{\mathrm{tip}}^2 \ll \epsilon^2$, $F^{PL}_{\delta_{\mathrm{tip}}}$ satisfies the trapping and holonomy conclusions of Proposition~\ref{prop:2D_smooth_box_fold}.
\end{corollary}
\begin{proof}
By choosing $\delta_{\mathrm{tip}}$ sufficiently small in Inequality~\ref{ineq:2D_trapping_interval} we see that $\tilde{t}$ can be made arbitrarily close to $e^{-s_0}t_0$, leading us to Property~\ref{property:2D_trapping_interval} in Proposition~\ref{prop:2D_smooth_box_fold}.  With $\delta_{\mathrm{tip}}$ as in Lemma~\ref{lemma:graphical_box_fold_holonomy}, Property~\ref{property:2D_trapping_region} will hold for any $\delta_{\mathrm{con}}>0$.  Properties~\ref{property:2D_holonomy_bound} and~\ref{property:2D_holonomy_top} follow from Equation~\ref{eq:2D_graphical_holonomy} and Remark~\ref{remark:small-z0}; the proof of Lemma~\ref{lemma:graphical_box_fold_holonomy}, which tracks the flowlines of $X_{\lambda_{\mathrm{tip}}}$ as they pass through $D^{PL}_2$, confirms Property~\ref{property:2D_holonomy_region} for any $\delta_{\mathrm{con}}>0$.
\end{proof}

\subsubsection{Smooth approximation}
We now smooth our graphical approximation $F^{PL}_{\delta_{\mathrm{tip}}}$ via convolution with a mollifier.  Namely, we consider a smooth cutoff function $\varphi\colon\mathbb{R}^2\to[0,\infty)$ defined by
\[
\varphi(s,t) = \left\{\begin{matrix}
    k\,\exp(-\tfrac{1}{1-s^2-t^2}), & s^2+t^2<1\\
    0, & s^2+t^2\geq 1
\end{matrix}\right.,
\]
where the constant $k>0$ is chosen to ensure that $\int_{\mathbb{R}^2}\varphi\,dA=1$.  For any $\delta_{\mathrm{con}}>0$ we may then define $\varphi_{\delta_{\mathrm{con}}}$ by
\[
\varphi_{\delta_{\mathrm{con}}}(s,t) = \frac{1}{\delta_{\mathrm{con}}^2}\varphi\left(\frac{s}{\delta_{\mathrm{con}}},\frac{t}{\delta_{\mathrm{con}}}\right).
\]
Notice that $\varphi_{\delta_{\mathrm{con}}}$ is supported on the ball $B(0,\delta_{\mathrm{con}})$, and that $\int_{\mathbb{R}^2}\varphi_{\delta_{\mathrm{con}}}\,dA=1$.

Finally, we consider $\delta_{\mathrm{con}}>0$ such that $\delta_{\mathrm{con}}<\delta_{\mathrm{tip}}^2$ and $\delta_{\mathrm{tip}}+\delta_{\mathrm{con}}<\epsilon$, and we define $F_\epsilon$ by convolution:
\[
F_\epsilon(x) = (F^{PL}_{\delta_{\mathrm{tip}}}*\varphi_{\delta_{\mathrm{con}}})(x) = \int_{\mathbb{R}^2}F^{PL}_{\delta_{\mathrm{tip}}}(y)\,\varphi_{\delta_{\mathrm{con}}}(y-x)\,dy.
\]
Because $\varphi_{\delta_{\mathrm{con}}}$ is smooth, the same is true of $F_\epsilon$.  We now investigate the Liouville vector field $X_{\lambda_\epsilon}$ of the resulting form $\lambda_\epsilon = dF_\epsilon + e^s\,dt$.

We begin by identifying regions $D^{\mathrm{sm}}_j$, $1\leq j\leq 6$, which closely approximate the interiors of the regions $D^{PL}_j$ defined above:
\[
D^{\mathrm{sm}}_j = \{(s,t)\,|\, B((s,t),\delta_{\mathrm{con}})\subset D^{PL}_j\},
\quad\text{for~}1\leq j\leq 6.
\]
Here $B((s,t),\delta_{\mathrm{con}})$ denotes the ball in $\mathbb{R}^2$ of radius $\delta_{\mathrm{con}}>0$ which is centered at $(s,t)$.  Because $F^{PL}_{\delta_{\mathrm{tip}}}$ is linear on each $D^{PL}_j$, we find that $F_\epsilon\equiv F^{PL}_{\delta_{\mathrm{tip}}}$ on each $D^{\mathrm{sm}}_j$, and thus our analysis on these regions proceeds as before.  It remains to understand the dynamics of $X_{\lambda_\epsilon}$ near nonempty intersections $D^{PL}_i\cap D^{PL}_j$, $i\neq j$.

\begin{figure}
\centering
\begin{tikzpicture}[scale=0.35]
\draw (-1,-1) -- (25,-1) -- (25,17) -- (-1,17) -- (-1,-1);
\draw (0,0) -- (24,0) -- (24,16) -- (0,16) -- (0,0);
\draw (3,3) -- (21,3) -- (21,13) -- (3,13) -- (3,3);

\draw (0,0) -- (3,3);
\draw (24,0) -- (21,3);
\draw (24,16) -- (21,13);
\draw (0,16) -- (3,13);

\draw[thick,->] (26,-2) -- (-2,-2);
\draw[thick,->] (26,-2) -- (26,18);

\node at (-2.5,-2) {$t$};
\node at (26,18.5) {$s$};

\draw[dashed] (25,-1) -- (25,-2.5);
\draw[dashed] (24,-1) -- (24,-2.5);
\draw[dashed] (21,-1) -- (21,-2.5);
\draw[dashed] (3,-1) -- (3,-2.5);
\draw[dashed] (0,-1) -- (0,-3.5);
\draw[dashed] (-1,-1) -- (-1,-2.5);
{\small
\node at (25,-3) {$0$};
\node at (24,-3) {$\delta_{\mathrm{tip}}^2$};
\node at (21,-3) {$\delta_{\mathrm{tip}}$};
\node at (3,-3) {$t_0-\delta_{\mathrm{tip}}$};
\node at (0.5,-4) {$t_0-\delta_{\mathrm{tip}}^2$};
\node at (-1,-3) {$t_0$};
}

\draw[dashed] (25,-1) -- (26.5,-1);
\draw[dashed] (25,0) -- (26.5,0);
\draw[dashed] (25,3) -- (26.5,3);
\draw[dashed] (25,13) -- (26.5,13);
\draw[dashed] (25,16) -- (26.5,16);
\draw[dashed] (25,17) -- (26.5,17);
{\small
\node at (27,-1) {$0$};
\node at (27.5,0) {$\delta_{\mathrm{tip}}^2$};
\node at (27.5,3) {$\delta_{\mathrm{tip}}$};
\node at (28.5,13) {$s_0-\delta_{\mathrm{tip}}$};
\node at (28.5,16) {$s_0-\delta_{\mathrm{tip}}^2$};
\node at (27.25,17) {$s_0$};
}

\node at (12,8) {$\widehat{D}_1$};
\node at (22.5,8) {$\widehat{D}_2$};
\node at (12,1.5) {$\widehat{D}_3$};
\node at (1.5,8) {$\widehat{D}_4$};
\node at (12,14.5) {$\widehat{D}_5$};

\draw[densely dashed]   (0.25,-.25) -- 
                (23.75,-.25) .. controls (24.5,-.25) .. 
                (23.75,.5) -- 
                (21.5,2.75) .. controls (21,3.25) ..
                (20.75,3.25) --
                (3.25,3.25) .. controls (3,3.25) ..
                (2.75,3) --
                (0.25,0.5) .. controls (-0.5,-.25) ..
                (0.25,-0.25);

\draw[densely dashed]   (3.25,12.75) -- 
                (20.75,12.75) .. controls (21,12.75) .. 
                (21.25,13) -- 
                (23.75,15.25) .. controls (24.5,16.25) ..
                (23.75,16.25) --
                (0.25,16.25) .. controls (-0.5,16.25) ..
                (0.25,15.5) --
                (2.75,13) .. controls (3,12.75) ..
                (3.25,12.75);

\draw[densely dotted]   (0.25,0) -- 
                (2.75,2.5) .. controls (3.25,3) .. 
                (3.25,3.25) -- 
                (3.25,12.75) .. controls (3.25,13) ..
                (2.75,13.5) --
                (0.25,16) .. controls (-0.25,16.25) ..
                (-0.25,15.75) --
                (-0.25,0.25) .. controls (-0.25,-0.25) ..
                (0.25,0);
\draw[red] (-0.15,0.3) -- (-0.15,15.7);
\draw[red] (3.15,3.3) -- (3.15,12.7);

\draw[densely dotted]   (23.75,0) -- 
                (21.25,2.5) .. controls (20.75,3) .. 
                (20.75,3.25) -- 
                (20.75,12.75) .. controls (20.75,13) ..
                (21.25,13.5) --
                (23.75,16) .. controls (24.25,16.25) ..
                (24.25,15.75) --
                (24.25,0.25) .. controls (24.25,-0.25) ..
                (23.75,0);

\draw[densely dashdotted]   (3.25,2.75) -- 
                (20.75,2.75) .. controls (21.25,2.75) .. 
                (21.25,3.25) -- 
                (21.25,12.75) .. controls (21.25,13.25) ..
                (20.75,13.25) --
                (3.25,13.25) .. controls (2.75,13.25) ..
                (2.75,12.75) --
                (2.75,3.25) .. controls (2.75,2.75) ..
                (3.25,2.75);
\end{tikzpicture}
\caption{The domains of influence $\widehat{D}_1,\ldots,\widehat{D}_5$, with the critical points of $X_{\lambda_\epsilon}$ seen in $\widehat{D}_4$.}
\label{fig:2d-smooth-behavior}
\end{figure}

Towards this understanding, let us consider regions
\[
\widehat{D}_j = \{(s,t)\,|\, B((s,t),\delta_{\mathrm{con}})\cap D^{PL}_j \neq \emptyset\},
\quad\text{for~}1\leq j\leq 6.
\]
See Figure~\ref{fig:2d-smooth-behavior}.  The region $\widehat{D}_j$ represents the ``domain of influence" for $D^{PL}_j$ in the convolution $F_\epsilon=F^{PL}_{\delta_{\mathrm{tip}}}\ast\varphi_{\delta_{\mathrm{con}}}$.  For instance, the Liouville vector field $X_{\lambda_\epsilon}$ will have positive $\partial_t$-component in the interior of $\widehat{D}_5$, negative $\partial_t$-component in the interior of $\widehat{D}_3$, and zero $\partial_t$-component outside of these regions.  Similarly, the $\partial_s$-component of $X_{\lambda_\epsilon}$ will be positive outside of $\widehat{D}_4$.

The smooth function $F_\epsilon$ differs from $F^{PL}_{\delta_{\mathrm{tip}}}$ --- and thus $X_{\lambda_\epsilon}$ differs from $X_{\lambda_{\mathrm{tip}}}$ --- precisely where distinct sets $\widehat{D}_i$, $\widehat{D}_j$ intersect.  However, with the exception of $\widehat{D}_1\cap \widehat{D}_4$ and $\widehat{D}_4\cap\widehat{D}_6$, these overlaps represent intersections which are crossed by $X_{\lambda_{\mathrm{tip}}}$, and the net effect of replacing $X_{\lambda_{\mathrm{tip}}}$ with $X_{\lambda_\epsilon}$ is zero.  For instance, consider the intersection $\widehat{D}_3\cap\widehat{D}_6$.  Away from $\widehat{D}_2$ and $\widehat{D}_4$ we have $\tfrac{\partial F_\epsilon}{\partial t}\equiv 0$, and thus
\[
X_{\lambda_\epsilon} = \left(1 + \tfrac{\partial F_\epsilon}{\partial t}\right)\partial_s - \tfrac{\partial F_\epsilon}{\partial s}\partial_t = \partial_s - \left(\tfrac{\partial F^{PL}_{\delta_{\mathrm{tip}}}}{\partial s}\ast\varphi_{\delta_{\mathrm{con}}}\right)\partial_t.
\]
Now consider a flowline of $X_{\lambda_\epsilon}$ which enters $(\widehat{D}_3\cap\widehat{D}_6)\setminus(\widehat{D}_2\cup\widehat{D}_4)$ along $s=\delta_{\mathrm{tip}}^2-\delta_{\mathrm{con}}$ and exits along $s=\delta_{\mathrm{tip}}^2+\delta_{\mathrm{con}}$.  The $t$-coordinate $t_{\mathrm{exit}}$ at which this flowline exits  will differ from its entrance $t$-coordinate $t_{\mathrm{ent}}$ by
\begin{align*}
\left[-\int_{\delta_{\mathrm{tip}}^2-\delta_{\mathrm{con}}}^{\delta_{\mathrm{tip}}^2+\delta_{\mathrm{con}}} \left(\tfrac{\partial F^{PL}_{\delta_{\mathrm{tip}}}}{\partial s}\ast\varphi_{\delta_{\mathrm{con}}}\right)\, ds \right]_{t_{\mathrm{ent}}}^{t_{\mathrm{exit}}} &= \left(F^{PL}_{\delta_{\mathrm{tip}}}\ast\varphi_{\delta_{\mathrm{con}}}\right)(\delta_{\mathrm{tip}}^2-\delta_{\mathrm{con}},t_{\mathrm{ent}}) - \left(F^{PL}_{\delta_{\mathrm{tip}}}\ast\varphi_{\delta_{\mathrm{con}}}\right)(\delta_{\mathrm{tip}}^2+\delta_{\mathrm{con}},t_{\mathrm{exit}})\\
    &= F^{PL}_{\delta_{\mathrm{tip}}}(\delta_{\mathrm{tip}}^2-\delta_{\mathrm{con}},t_{\mathrm{ent}}) - F^{PL}_{\delta_{\mathrm{tip}}}(\delta_{\mathrm{tip}}^2+\delta_{\mathrm{con}},t_{\mathrm{exit}}).
\end{align*}
Because $F^{PL}_{\delta_{\mathrm{tip}}}$ is $t$-invariant in this region, we see that the change in $t$-coordinate induced by $X_{\lambda_\epsilon}$ agrees with that induced by $X_{\lambda_{\mathrm{tip}}}$, so the net effect of smoothing on this region is zero.  A similar analysis applies to $\widehat{D}_1\cap\widehat{D}_3$, $\widehat{D}_1\cap\widehat{D}_5$, and $\widehat{D}_5\cap\widehat{D}_6$.  On $\widehat{D}_2\cap\widehat{D}_3$, $\widehat{D}_2\cap\widehat{D}_5$, $\widehat{D}_3\cap\widehat{D}_4$, and $\widehat{D}_4\cap\widehat{D}_5$ we may repeat the analysis with respect to the coordinates $x=s+t$, $y=s-t$ to once again see that smoothing does not affect our overall trapping and holonomy analysis.  In the regions $D^{PL}_1$, $D^{PL}_2$, and $D^{PL}_4$, $X_{\lambda_{\mathrm{tip}}}$ has positive $\partial_s$-component and zero $\partial_t$-component, and thus the smoothing which occurs on $\widehat{D}_1\cap\widehat{D}_2$ and $\widehat{D}_2\cap\widehat{D}_6$ will simply be an interpolation in the $\partial_s$-component, with no affect on our dynamical analysis.

Finally, the regions $(\widehat{D}_1\cap\widehat{D}_4)\setminus(\widehat{D}_3\cup\widehat{D}_5)$ and $(\widehat{D}_4\cap\widehat{D}_6)\setminus(\widehat{D}_3\cup\widehat{D}_5)$ behave rather differently, but again do not greatly affect our overall analysis.  The large magnitude of $X_{\lambda_{\mathrm{tip}}}$ in $D^{PL}_4$ ensures that $X_{\lambda_\epsilon}$ has negative $\partial_s$-component throughout this region.  In $D^{PL}_1\cap\widehat{D}_4$ we interpolate from this negative $\partial_s$-component to the vector field $\partial_s$, resulting in a curve of points with nearly constant $t$-value.  This critical curve acts as a spiral source for the flowlines of $X_{\lambda_\epsilon}$, and there is a second critical curve in $D^{PL}_6\cap\widehat{D}_4$ which acts as a degenerate saddle.  See Figure~\ref{fig:2d-smooth-behavior}.

At last we observe that each nonempty triple intersection $\widehat{D}_i\cap\widehat{D}_j\cap\widehat{D}_k$ will meet either $D^{PL}_1$ or $D^{PL}_6$.  For sufficiently small $\epsilon>0$, the contribution of these regions to $X_{\lambda_\epsilon}$ within $[\delta_{\mathrm{tip}}^2,s_0-\delta_{\mathrm{tip}}^2]\times[\delta_{\mathrm{tip}}^2,t_0-\delta_{\mathrm{tip}}^2]$ will be dominated by those of $D^{PL}_2$ or $D^{PL}_4$, and thus our understanding of the pairwise intersections $\widehat{D}_i\cap\widehat{D}_j$ suffices.

\begin{proof}[Proof of Proposition~\ref{prop:2D_smooth_box_fold}]
Corollary~\ref{cor:2D_PL_properties} and the discussion above ensure that $X_{\lambda_\epsilon}$ satisfies the trapping and holonomy properties of Proposition~\ref{prop:2D_smooth_box_fold}.  As constructed, our $X_{\lambda_\epsilon}$ fails to be Morse, since it has two curves of critical points.  However, we may apply an arbitrarily small perturbation to $F_\epsilon$, supported in isolating neighborhoods of these critical curves, so that the resulting vector field $X_{\lambda_\epsilon}$ is Morse.  Because the perturbation is supported on isolating neighborhoods of the critical curves, our trapping and holonomy analysis is unaffected.
\end{proof}

\begin{remark}\label{remark:2D-cp-height}
The critical points of $X_{\lambda_\epsilon}$ have $z$-values very near $z=0$ and $z=z_0$, respectively, as they approximate the critical edges of the piecewise linear box fold.  For instance, suppose $(s^*,t^*)$ is a critical point in $D_4^{PL}$ of the Liouville vector field of $F^{PL}_{\delta_{\mathrm{tip}}}\ast\varphi_{\delta_{\mathrm{con}}}$.  Then
\[
(\tfrac{\partial F^{PL}_{\delta_{\mathrm{tip}}}}{\partial t}\ast\varphi_{\delta_{\mathrm{con}}})(s^*,t^*) = -\int_{B^*} \dfrac{e^{-s}z_0}{\delta_{\mathrm{tip}}(1-\delta_\mathrm{tip})}\,\varphi_{\delta_{\mathrm{con}}}(s-s^*,t-t^*)\,dA < -\tfrac{e^{-s_0}z_0}{\delta_{\mathrm{tip}}(1-\delta_\mathrm{tip})} \int_{B^*} \varphi_{\delta_{\mathrm{con}}}(s-s^*,t-t^*)\,dA,
\]
where $B^*=B((s^*,t^*),\delta_{\mathrm{con}})\cap D^{PL}_4$.  Because $(s^*,t^*)$ is a critical point of the Liouville vector field, we have $(\tfrac{\partial F^{PL}_{\delta_{\mathrm{tip}}}}{\partial t}\ast\varphi_{\delta_{\mathrm{con}}})(s^*,t^*)=-1$, and thus
\[
\int_{B^*} \varphi_{\delta_{\mathrm{con}}}(s-s^*,t-t^*)\,dA < \tfrac{e^{s_0}\delta_{\mathrm{tip}}(1-\delta_\mathrm{tip})}{z_0}.
\]
One then finds that
\[
(F^{PL}_{\delta_{\mathrm{tip}}}\ast\varphi_{\delta_{\mathrm{con}}})(s^*,t^*) < \tfrac{z_0\,\delta_{\mathrm{con}}}{\delta_{\mathrm{tip}}(1-\delta_\mathrm{tip})}\,\int_{B^*} \varphi_{\delta_{\mathrm{con}}}(s-s^*,t-t^*)\,dA < \tfrac{z_0\,\delta_{\mathrm{con}}}{\delta_{\mathrm{tip}}(1-\delta_\mathrm{tip})}\,\tfrac{e^{s_0}\delta_{\mathrm{tip}}(1-\delta_\mathrm{tip})}{z_0} = \delta_{\mathrm{con}}\,e^{s_0}.
\]
\end{remark}

\subsection{Nearly smooth box folds}\label{subsec:nearlysmooth}

Let $(W_0, \lambda_0)$ be a Weinstein domain of dimension $2n-2\geq 2$. Consider the model contact manifold 
\[
\left(M_0:= [0,z_0] \times [0,s_0] \times [0,t_0] \times W_0,\, \alpha_0 = dz + e^s(dt + \lambda_0) \right).
\]
Let $V_0 = \{z=0\}$ and fix $\epsilon > 0$. The goal of this subsection is to define a hypersurface $\Pi^{NS}\subset M_0$ which is the union of two smoothly embedded hypersurfaces with boundary meeting along a codimension-$1$ corner, and is a continuous perturbation of $V_0$. 

Let $F_{\epsilon}:[0,s_0]\times [0,t_0]\to [0,z_0]$ be the smooth function constructed in the previous subsection and let $\Pi_0:= \{z=F_{\epsilon}(s,t)\}\subset [0,z_0]\times [0,s_0]\times [0,t_0]$ denote the corresponding $2$-dimensional box fold. Define
\begin{align*}
    \Pi_{\mathrm{top}} &:= \Pi_0 \times W_0,\\
    \Pi_{\mathrm{side}} &:= 
    \{0\leq z \leq F_{\epsilon}(s,t)\}\times \partial W_0.
\end{align*}
Note that $\Pi^{NS} := \Pi_{\mathrm{top}} \cup \Pi_{\mathrm{side}}$ has a codimension-$1$ corner given by $\Pi_{\mathrm{top}} \cap \Pi_{\mathrm{side}} = \Pi_0 \times \partial W_0$.

The following lemma is the nearly smooth generalization of Lemma \ref{lemma:PL_cf_high_dimensional}.

\begin{lemma}\label{lemma:NS_cf}
Let $X_{\lambda_0}$ denote the Liouville vector field of $(W_0^{2n-2}, \lambda_0)$, let $\eta_0 := \lambda_0\mid_{\partial W_0}$ be the induced contact form on the boundary of $W_0$, and let $R_{\eta_0}$ denote the Reeb vector field on $\partial W_0$ of $\eta_0$. Let
\[
X_0 := -e^{s}\partial_s -X_{F_{\epsilon}}^{ds\, dt}= -\left(e^s + \frac{\partial F_{\epsilon}}{\partial t}\right)\, \partial_s + \frac{\partial F_{\epsilon}}{\partial s}\, \partial_t
\]
be the vector field directing the backward oriented characteristic foliation of $\Pi_0$. Then the backward oriented characteristic foliation of $\Pi^{NS}$ is given by Table \ref{table:NS}. 
\end{lemma}

\begin{table}[ht]
    \centering
    \begin{tabular}{ |c|c|c|c| }
 \hline
 Side  & Characteristic foliation \rm \\ 
 \hline 
 $\Pi_{\mathrm{top}}$  & $  X_0 + \frac{\partial F_{\epsilon}}{\partial t}\, X_{\lambda_0}$ \\
 $\Pi_{\mathrm{side}}$  & $  \partial_t - R_{\eta_0}$ \\
 \hline
\end{tabular}
    \caption{The characteristic foliation of a nearly smooth box fold.}
    \label{table:NS}
\end{table}

\begin{proof}
The foliation of $\Pi_{\mathrm{side}}$ is given by Lemma \ref{lemma:PL_cf_high_dimensional}, in particular by the exact same calculation for $\underline{\partial W_0}$. On $\Pi_{\mathrm{top}}$ we take as volume form $\Omega = -e^{(n-1)s}\,ds\, dt\, (d\lambda_0)^{n-1}.$ Let $\beta := \alpha_0 \mid_{\Pi_{\mathrm{top}}} = dF_{\epsilon} + e^s\, (dt + \lambda_0)$. Then $d\beta = e^s\, ds\, (dt + \lambda_0) + e^s\, d\lambda_0$ and so 
\begin{align*}
    \beta\, (d\beta)^{n-1} &= \left[dF_{\epsilon} + e^s\, (dt + \lambda_0)\right]\, \left[e^s\, ds\, (dt + \lambda_0) + e^s\, d\lambda_0\right]^{n-1} \\
    &= \left[dF_{\epsilon} + e^s\, (dt + \lambda_0)\right]\,\left[e^{(n-1)s}\, (d\lambda_0)^{n-1} + (n-1)e^{(n-1)s}\, ds\, (dt + \lambda_0)\, (d\lambda_0)^{n-2} \right] \\
    &= e^{(n-1)s}\, dF_{\epsilon}\, (d\lambda_0)^{n-1} + (n-1)e^{(n-1)s}\, dF_{\epsilon}\, ds \, \lambda_0\, (d\lambda_0)^{n-2} + e^{ns}\, dt\, (d\lambda_0)^{n-1} \\
    &= e^{(n-1)s}\left[dF_{\epsilon}\, (d\lambda_0)^{n-1} + (n-1)\frac{\partial F_{\epsilon}}{\partial t}\, dt\, ds \, \lambda_0\, (d\lambda_0)^{n-2} + e^s\, dt\, (d\lambda_0)^{n-1}\right].
\end{align*}
Note that 
\begin{align*}
    \iota_{-e^s\, \partial_s} \Omega &= e^{(n-1)s}\, e^s\, dt\, (d\lambda_0)^{n-1} \\
    \iota_{-X_{F_{\epsilon}}^{ds\, dt}} \Omega &= e^{(n-1)s}\, dF_{\epsilon}\, (d\lambda_0)^{n-1} \\ 
    \iota_{\frac{\partial F_{\epsilon}}{\partial t}\, X_{\lambda_0}}\Omega &= e^{(n-1)s}\, (n-1)\frac{\partial F_{\epsilon}}{\partial t}\, dt\, ds \, \lambda_0\, (d\lambda_0)^{n-2}
\end{align*}
and so $\iota_{X_0 + \frac{\partial F_{\epsilon}}{\partial t}\, X_{\lambda_0}} \Omega = \beta \, (d\beta)^{n-1}.$ By Lemma \ref{cflemma}, it follows that $X_0 + \frac{\partial F_{\epsilon}}{\partial t}\, X_{\lambda_0}$ directs the characteristic foliation of $\Pi_{\mathrm{top}}$.
\end{proof}

Recall that the holonomy $h^{PL}:[0,t_0]\times W_0 \dashrightarrow [0,t_0]\times W_0$ of a piecewise linear box fold is given by $h^{PL}(t, p) =(e^{s_0}t, \psi^{s_0}(p))$, where $\psi^{s}:W_0 \to W_0$ is the time-$s$ flow of the Liouville vector field $X_{\lambda_0}$. In particular, for $x\in [0,t_0]\times W_0$ in the domain of $h^{PL}$ we have $\norm{h^{PL}(x)}_{W_0} = e^{s_0}\norm{x}_{W_0}$. In the nearly smooth case, we have: 

\begin{proposition}\label{prop:NS_holo_norm}
Let $h: [0,t_0] \times W_0 \dashrightarrow [0,t_0] \times W_0$ be the partially-defined holonomy map given by backward passage through a nearly smooth box fold $\Pi^{NS}$ with $s$-thickness $s_0$. Then for any $x\in [0,t_0] \times W_0$, we have
\begin{equation*}
\norm{h(x)}_{W_0} \leq e^{s_0}\norm{x}_{W_0}.    
\end{equation*}
\end{proposition}

\begin{proof}
Let $x\in [0,t_0] \times W_0$ be in the domain of $h$ and let $\gamma:[0,\tau_{\infty}] \to \Pi^{NS}$ be the (piecewise smooth) parametrized flowline of the backward oriented characteristic foliation through $x$, so that $\gamma(0) = x$ and $\gamma(\tau_{\infty}) = h(x)$. Here $\tau_{\infty}\in (0,\infty)$ is some finite, positive number. In this notation, our goal is to prove the estimate 
\begin{equation}\label{eq:NS_hol}
    \norm{\gamma(\tau_{\infty})}_{W_0} \leq e^{s_0}\norm{\gamma(0)}_{W_0}.
\end{equation}

\vspace{2mm}
\noindent \emph{Step 1: Decompose the interval $[0,\tau_{\infty}]$.}
\vspace{2mm}

We begin by decomposing the time interval $[0,\tau_{\infty}]$ into subintervals of four distinct types. First, let $I_{\mathrm{side}} := \gamma^{-1}(\Pi_{\mathrm{side}})$. Since trajectories along $\Pi_{\mathrm{side}}$ are directed by $\partial_t - R_{\eta_0}$ and hence are determined by the $\partial_t$-trajectories, $I_{\mathrm{side}} \subset [0,\tau_{\infty}]$ is a disjoint union of closed intervals. Next, let $I_{\mathrm{top}} = (0,\tau_{\infty}) \setminus I_{\mathrm{side}}$ be the set of times for which $\gamma(\tau) \in \mathrm{int}(\Pi_{\mathrm{top}})$. By Lemma \ref{lemma:NS_cf}, the characteristic foliation on $\Pi_{\mathrm{top}}$ is directed by 
\[
X_{\mathrm{top}}:= -\left(e^s + \frac{\partial F_{\epsilon}}{\partial t}\right)\, \partial_s + \frac{\partial F_{\epsilon}}{\partial s}\, \partial_t + \frac{\partial F_{\epsilon}}{\partial t}\, X_{\lambda_0}.
\]
We then define three subsets of $I_{\mathrm{top}}$ according to the (negative) sign of the coefficient of $\partial_s$ as follows. First, for $\gamma(\tau) \in \Pi_{\mathrm{top}}$, let $s(\tau)$ and $t(\tau)$ denote the $s$- and $t$-components of $\gamma$, respectively, so that $s(0) = s_0$ and $s(\tau_{\infty}) = 0$. 
\begin{align*}
    I_{\mathrm{top}}^0 &:= \set{\tau\in I_{\mathrm{top}}}{e^{s(\tau)} + \frac{\partial F_{\epsilon}}{\partial t}(s(\tau), t(\tau)) = 0}, \\
    I_{\mathrm{top}}^+ &:= \set{\tau\in I_{\mathrm{top}}}{e^{s(\tau)} + \frac{\partial F_{\epsilon}}{\partial t}(s(\tau), t(\tau)) > 0}, \\
    I_{\mathrm{top}}^- &:= \set{\tau\in I_{\mathrm{top}}}{e^{s(\tau)} + \frac{\partial F_{\epsilon}}{\partial t}(s(\tau), t(\tau)) < 0}. 
\end{align*}
Note that $I_{\mathrm{top}}^0 \subset (0,\tau_{\infty})$ is a closed and bounded set, hence is a disjoint union of closed intervals. In total we have a disjoint decomposition 
\[
(0,\tau_{\infty}) = I_{\mathrm{side}} \sqcup I_{\mathrm{top}}^0 \sqcup I_{\mathrm{top}}^+ \sqcup I_{\mathrm{top}}^-
\]
where $I_{\mathrm{side}}$ and $I_{\mathrm{top}}^0$ are collections of disjoint closed intervals, while $I_{\mathrm{top}}^+$ and $I_{\mathrm{top}}^-$ are collections of disjoint open intervals. Write 
\[
I_{\mathrm{top}}^+ \sqcup I_{\mathrm{top}}^- =: \bigsqcup_{j=1}^N (a_j, b_j)
\]
where $a_1 = 0$, $b_N = \tau_{\infty}$, and $b_j < a_{j+1}$ for all $1\leq j\leq N$. 

\vspace{2mm}
\noindent \emph{Step 2: Estimate the norm in each subinterval.}
\vspace{2mm}

We will prove the estimate \eqref{eq:NS_hol} with successive estimates on $\norm{\gamma(a_j)}_{W_0}$ and $\norm{\gamma(b_j)}_{W_0}$. Before proceeding with the casework, we introduce one more piece of notation. For each $1\leq j\leq N$, let $\Delta s_j = s(b_j) - s(a_j)$. By definition of our decomposition, the $s$-component of $\gamma$ only changes in $I_{\mathrm{top}}^+ \sqcup I_{\mathrm{top}}^-$ and so 
\[
\sum_{j=1}^N \Delta s_j = -s_0. 
\]
Now we derive norm estimates on each of the four types of intervals. 

\vspace{2mm}
\noindent \emph{Case I: Intervals in $I_{\mathrm{side}}$.}
\vspace{2mm}

Let $[\tau_0, \tau_1] \subset I_{\mathrm{side}}$ be a connected component of $I_{\mathrm{side}}$. By Lemma \ref{lemma:NS_cf}, the characteristic foliation is directed by $\partial_t - R_{\eta_0}$, so $\norm{\gamma(\tau_1)}_{W_0} = \norm{\gamma(\tau_0)}_{W_0}$. 

\vspace{2mm}
\noindent \emph{Case II: Intervals in $I_{\mathrm{top}}^0$.}
\vspace{2mm}

Let $[\tau_0, \tau_1] \subset I_{\mathrm{top}}^0$ be a connected component of $I_{\mathrm{top}}^0$. Then by definition we have $e^s + \frac{\partial F_{\epsilon}}{\partial t} = 0$, hence $\frac{\partial F_{\epsilon}}{\partial t} < 0$, so the characteristic foliation is directed by the vector field 
\[
X_{\mathrm{top}} = \frac{\partial F_{\epsilon}}{\partial s}\, \partial_t - \left|\frac{\partial F_{\epsilon}}{\partial t}\right|\, X_{\lambda_0}.
\]
As the coefficient in front of $X_{\lambda_0}$ is negative, it follows that $\norm{\gamma(\tau_1)}_{W_0} \leq \norm{\gamma(\tau_0)}_{W_0}$.

\vspace{2mm}
\noindent \emph{Case III: Intervals in $I_{\mathrm{top}}^+$.}
\vspace{2mm}

Let $(a_j, b_j)\subset I_{\mathrm{top}}^+$ be a connected component of $I_{\mathrm{top}}^+$. Since $e^s + \frac{\partial F_{\epsilon}}{\partial t} > 0$, we may rescale $X_{\mathrm{top}}$ by $\frac{1}{e^s + \frac{\partial F_{\epsilon}}{\partial t}}$ and direct the characteristic foliation by 
\[
- \partial_s + \frac{\frac{\partial F_{\epsilon}}{\partial s}}{e^s + \frac{\partial F_{\epsilon}}{\partial t}}\, \partial_t + \frac{\frac{\partial F_{\epsilon}}{\partial t}}{e^s + \frac{\partial F_{\epsilon}}{\partial t}} \, X_{\lambda_0}.
\]
Since $\frac{\partial F_{\epsilon}}{\partial t} < e^s + \frac{\partial F_{\epsilon}}{\partial t}$ and $e^s + \frac{\partial F_{\epsilon}}{\partial t} > 0$, the coefficient in front of $X_{\lambda_0}$ satisfies $\frac{\partial F_{\epsilon}}{\partial t}\cdot\left(e^s + \frac{\partial F_{\epsilon}}{\partial t}\right)^{-1} \leq 1$. (The coefficient could be negative, which is fine.) Next, note that $\Delta s_j < 0$ since the leading term of the vector field is $-\partial_s$. Thus, the flowline experiences at worst a time-$|\Delta s_j|$ flow in the $X_{\lambda_0}$-direction on the time interval $(a_j,b_j)$, so 
\[
\norm{\gamma(b_j)}_{W_0} \leq e^{|\Delta s_j|}\norm{\gamma(a_j)}_{W_0} = e^{-\Delta s_j}\norm{\gamma(a_j)}_{W_0}.
\]
\vspace{2mm}
\noindent \emph{Case IV: Intervals in $I_{\mathrm{top}}^-$.}
\vspace{2mm}

Finally, let $(a_j, b_j)\subset I_{\mathrm{top}}^-$ be a connected component of $I_{\mathrm{top}}^-$. The analysis is similar to Case IV. Since $e^s + \frac{\partial F_{\epsilon}}{\partial t} < 0$, we rescale $X_{\mathrm{top}}$ by $-\frac{1}{e^s + \frac{\partial F_{\epsilon}}{\partial t}}$ and direct the characteristic foliation by 
\[
\partial_s - \frac{\frac{\partial F_{\epsilon}}{\partial s}}{e^s + \frac{\partial F_{\epsilon}}{\partial t}}\, \partial_t - \frac{\frac{\partial F_{\epsilon}}{\partial t}}{e^s + \frac{\partial F_{\epsilon}}{\partial t}} \, X_{\lambda_0}.
\]
Note that  $\frac{\partial F_{\epsilon}}{\partial t} < e^s + \frac{\partial F_{\epsilon}}{\partial t} < 0$ and so the coefficient in front of $X_{\lambda_0}$ is $-\left|\frac{\partial F_{\epsilon}}{\partial t}\right|\cdot\left|e^s + \frac{\partial F_{\epsilon}}{\partial t}\right|^{-1} < -1$. Next, note that $\Delta s_j > 0$ since the leading term of the vector field is $\partial_s$. Thus, the flowline experiences at least a time-$\Delta s_j$ flow in the $-X_{\lambda_0}$-direction on the time interval $(a_j,b_j)$, and so 
\[
\norm{\gamma(b_j)}_{W_0} \leq e^{-\Delta s_j}\norm{\gamma(a_j)}_{W_0}.
\]
This completes the subinterval casework.

\vspace{2mm}
\noindent \emph{Step 3: Compose the subinterval estimates.}
\vspace{2mm}

We complete the proof by chaining together each of the above subinterval estimates. First, by Case I and Case II, for any interval $[\tau_0, \tau_1]\subset I_{\mathrm{side}} \sqcup I_{\mathrm{top}}^0$ we have $\norm{\gamma(\tau_1)}_{W_0}\leq \norm{\gamma(\tau_0)}_{W_0}$. In particular this means that for all $2\leq j \leq N$ we have $\norm{\gamma(a_j)}_{W_0}\leq \norm{\gamma(b_{j-1})}_{W_0}$. Second, by Case III and Case IV, for all $1\leq j\leq N$ we have $\norm{\gamma(b_j)}_{W_0} \leq e^{-\Delta s_j}\norm{\gamma(a_j)}_{W_0}.$ Combining these estimates gives $\norm{\gamma(b_j)}_{W_0} \leq e^{-\Delta s_j}\norm{\gamma(b_{j-1})}_{W_0}$ for all $2\leq j\leq N$ and hence 
\[
\norm{\gamma(b_N)}_{W_0} \leq e^{-\Delta s_N}\norm{\gamma(b_{N-1})}_{W_0} \leq \cdots \leq e^{-\sum_{j=2}^N \Delta s_j} \norm{\gamma(b_1)}_{W_0} \leq e^{-\sum_{j=1}^N \Delta s_j} \norm{\gamma(a_1)}_{W_0}.
\]
Since $a_1 = 0$, $b_N = \tau_{\infty}$, and $-\sum_{j=1}^N \Delta s_j = s_0$, this proves the estimate \eqref{eq:NS_hol} as desired.
\end{proof}

The next two propositions concern the trapping behavior and holonomy properties of nearly smooth folds. 

\begin{proposition}[Nearly smooth trapping properties]\label{prop:NS_trap}
Let $h^{\epsilon}: [0,t_0] \times W_0 \dashrightarrow [0,t_0] \times W_0$ be the partially-defined holonomy map given by backward passage through $\Pi^{NS}$, where $\epsilon >0$ is the $2$-dimensional smoothing parameter.  
\begin{enumerate}

\item Any flowline entering $[e^{-s_0}t_0 + \epsilon, t_0 - \epsilon] \times (W_0 \setminus N^{s_0}(\partial W_0))$ is trapped in backward time.\label{property:NS_trap_1}
    
\item For the auxiliary parameter $\delta_{\mathrm{tip}}$ as in Proposition \ref{prop:2D_smooth_box_fold} chosen sufficiently small, any flowline entering $\Pi^{NS}$ in $[\epsilon, t_0 - \epsilon] \times N^{s_0-\epsilon}(\partial W_0)$ is trapped in backward time.\label{property:NS_trap_2} 
\end{enumerate}

\end{proposition}

\begin{proof}
Let $x\in [e^{-s_0}t_0 + \epsilon, t_0 - \epsilon] \times (W_0 \setminus N^{s_0}(\partial W_0))$. Since $W_0(x)\notin N^{s_0}(\partial W_0)$, the proof of Proposition ~\ref{prop:NS_holo_norm} implies that the flowline through $x$ is entirely contained in $\Pi_{\mathrm{top}}$. Thus, the long-term behavior of the flowline is completely determined by the vector field $X_0 + \frac{\partial F_{\epsilon}}{\partial t}\, X_{\lambda_0}$ directing the characteristic foliation of $\Pi_{\mathrm{top}}$, where $X_0$ is the vector field directing the characteristic foliation of the $2$-dimensional fold $\Pi_0$. Note that this factor of $X_0$ does not depend on the $W_0$-coordinate. Thus, by \ref{property:2D_trapping_interval} of Proposition \ref{prop:2D_smooth_box_fold}, the projection of the flowline to the $\Pi_0$-direction converges to a critical point of $X_0$, necessarily with $\frac{\partial F_{\epsilon}}{\partial t} < 0$. In the $W_0$-direction, the term $\frac{\partial F_{\epsilon}}{\partial t}\, X_{\lambda_0}$ then limits to a critical point on the skeleton of $W_0$ and hence the flowline is trapped. This proves \ref{property:NS_trap_1}. 

Now we consider \ref{property:NS_trap_2}. We proceed in two steps. 

\vspace{2mm}
\noindent \textit{Step 1: Every such flowline reaches $\Pi_{\mathrm{side}}$ with $(s,t)$-coordinate in $\tilde{D}_2$.}
\vspace{2mm}

As in the proof of Proposition \ref{prop:NS_holo_norm}, let $\gamma:[0,\tau_{\infty})\to \Pi^{NS}$ denote a parametrized piecewise smooth flowline of the backward oriented characteristic foliation and assume that $\gamma(0) \in \{s=s_0\}\times [\epsilon, t_0 - \epsilon] \times N^{s_0-\epsilon}(\partial W_0)$. Also as before, identify the time intervals $I_{\mathrm{side}}, I_{\mathrm{top}}^*$ for $*=0,-,+$. Since $0\in I_{\mathrm{top}}^+$, the first part of the trajectory is contained in $\Pi_{\mathrm{top}}$ and we may direct the characteristic foliation by 
\begin{equation}\label{eq:NS_trap}
    - \partial_s + \frac{\frac{\partial F_{\epsilon}}{\partial s}}{e^s + \frac{\partial F_{\epsilon}}{\partial t}}\, \partial_t + \frac{\frac{\partial F_{\epsilon}}{\partial t}}{e^s + \frac{\partial F_{\epsilon}}{\partial t}} \, X_{\lambda_0}.
\end{equation}
as in Case III of the proof of Proposition \ref{prop:NS_holo_norm}. By \ref{property:2D_holonomy_region} of Proposition \ref{prop:2D_smooth_box_fold}, the ($(s,t)$-coordinate of the) flowline enters the region $\tilde{D}_2$ along $\{s=s_0 - (\delta_{\mathrm{tip}} + \delta_{\mathrm{con}})\}$. In this region, by the definition of $F_{\epsilon}$, we have 
\[
 \frac{\partial F_{\epsilon}}{\partial t} = \frac{z_0}{\delta_{\mathrm{tip}}(1-\delta_{\mathrm{tip}})}.
\]
In $\tilde{D}_2$ we can then identify a lower bound for the (positive) coefficient of $X_{\lambda_0}$ in \eqref{eq:NS_trap}:
\[
\frac{\frac{\partial F_{\epsilon}}{\partial t}}{e^s + \frac{\partial F_{\epsilon}}{\partial t}} = \frac{1}{\frac{e^s\delta_{\mathrm{tip}}(1-\delta_{\mathrm{tip}})}{z_0} + 1} \geq \frac{1}{\frac{e^{s_0}\delta_{\mathrm{tip}}(1-\delta_{\mathrm{tip}})}{z_0} + 1}.
\]
By \ref{property:2D_holonomy_region} of Proposition \ref{prop:2D_smooth_box_fold}, upon entering $\tilde{D}_2$, the flowline follows \eqref{eq:NS_trap} until it either reaches $\Pi_{\mathrm{side}}$ or reaches $\{s=\delta_{\mathrm{tip}} + \delta_{\mathrm{con}}\}$. In order for the latter to occur, it must flow via \eqref{eq:NS_trap} for time $s_0 - 2(\delta_{\mathrm{tip}} + \delta_{\mathrm{con}})$ and hence, using the lower bound above, must flow forward in the $X_{\lambda_0}$ direction for at least time 
\begin{equation}\label{eq:NS_trap2}
\frac{s_0 - 2(\delta_{\mathrm{tip}} + \delta_{\mathrm{con}})}{\frac{e^{s_0}\delta_{\mathrm{tip}}(1-\delta_{\mathrm{tip}})}{z_0} + 1}.    
\end{equation}
By choosing $\delta_{\mathrm{tip}}$ in Proposition \ref{prop:2D_smooth_box_fold} sufficiently small relative to $\epsilon$, we may ensure that the quantity in \eqref{eq:NS_trap2} is bounded below by $s_0 - \epsilon$. Indeed, note that as $\delta_{\mathrm{tip}} \to 0$, \eqref{eq:NS_trap2} $\to s_0$. Since the $W_0$-coordinate of the entry point of the flowline is in $N^{s_0 - \epsilon}(\partial W_0)$, this implies that the flowline reaches $\Pi_{\mathrm{side}}$ via $\tilde{D}_2$ as desired. 

\vspace{2mm}
\noindent \textit{Step 2: Every such flowline as in Step 1 is trapped.}
\vspace{2mm}

Consider a flowline that reaches $\Pi_{\mathrm{side}}$ with $(s,t)$-coordinate in $\tilde{D}_2$. It follows the characteristic foliation $\partial_t - R_{\eta_0}$ until it returns to $\Pi_{\mathrm{top}}$. Since the function $F_{\epsilon}$ is symmetric across $t=\frac{t_0}{2}$ by construction, the flowline returns to $\Pi_{\mathrm{top}}$ with $(s,t)$-coordinate in $\tilde{D}_4$ as defined in Proposition \ref{prop:2D_smooth_box_fold}. If the flowline never reaches $\Pi_{\mathrm{side}}$, then by \ref{property:2D_trapping_interval} of Proposition \ref{prop:2D_smooth_box_fold} and the same argument as in \ref{property:NS_trap_1} of the present proposition, the flowline is trapped. 

If the flowline does reach $\Pi_{\mathrm{side}}$, it does so necessarily with $(s,t)$-coordinates in the unstable manifold of the index $0$ point of $X_{\lambda_{F_{\epsilon}}}$. In the worst case scenario, we simply repeat Step 1 but with an increased $t$-coordinate upon entry to $\tilde{D}_2$. This is the exact same analysis as in the high-dimensional piecewise linear case of Lemma \ref{lemma:t_0_trap}. Ultimately, the flowline is trapped as in the first scenario of Step 2. 
\end{proof}

\begin{proposition}[Nearly smooth holonomy properties]\label{prop:NS_holonomy}
Let $h^{\epsilon}: [0,t_0] \times W_0 \dashrightarrow [0,t_0] \times W_0$ be the partially-defined holonomy map given by backward passage through $\Pi^{NS}$, where $\epsilon >0$ is the $2$-dimensional smoothing parameter.
    \begin{enumerate}
        \item Any flowline entering $\Pi^{NS}$ with $t$-coordinate in $[t_0-\epsilon, t_0]$ which is not trapped in backward time exits the fold with $t$-coordinate in $[t_0-\epsilon, t_0]$.\label{property:NS_holo_1} 

    \item Any flowline entering $\Pi^{NS}$ with $t$-coordinate in $[0,\epsilon]$ which is not trapped in backward time either \label{property:NS_holo_2}
    \begin{enumerate}
        \item exits the fold with $t$-coordinate in $[0,e^{s_0+ \epsilon}\epsilon]$, or \label{property:NS_holo_2a}
        \item traverses once across $\Pi_{\mathrm{side}}$ via $\partial_t - R_{\eta_0}$ and then exits the fold. \label{property:NS_holo_2b}
    \end{enumerate}
    \end{enumerate}
\end{proposition}

\begin{remark}
The second part of Proposition \ref{prop:NS_holonomy} is important enough to restate in heuristic terms for emphasis. It says that if a flowline enters the fold with small $t$-coordinate and exits the fold with large $t$-coordinate, it must necessarily experience a flow in the negative Reeb direction $-R_{\eta_0}$ of $\partial W_0$ for a time which equals the corresponding net change in $t$. This feature is used in the design of a chimney fold in Section \ref{section:chimney_folds}.
\end{remark}

\begin{proof}
We begin with \ref{property:NS_holo_1}, which quickly follows from the $2$-dimensional statement given by \ref{property:2D_holonomy_top} of Proposition \ref{prop:2D_smooth_box_fold}. In particular, if the flowline is entirely contained in $\Pi_{\mathrm{top}}$, then its $t$-holonomy is entirely governed by the $2$-dimensional dynamics of $\Pi_0$. The only new possibility to consider in the high-dimensional nearly smooth fold case is if the flowline traverses $\Pi_{\mathrm{side}}$. However, the foliation on $\Pi_{\mathrm{side}}$ is directed by $\partial_t - R_{\eta_0}$ and hence only serves to further increase the $t$-coordinate.

Next we consider \ref{property:NS_holo_2}. There are two possibilities: either the flowline does not ever reach $\Pi_{\mathrm{side}}$, or it reaches $\Pi_{\mathrm{side}}$ at least once. If the flowline does not ever reach $\Pi_{\mathrm{side}}$, then \ref{property:2D_holonomy_bound} of Proposition ~\ref{prop:2D_smooth_box_fold} implies \ref{property:NS_holo_2a}. It remains to prove that if the flowline reaches $\Pi_{\mathrm{side}}$ at least once, then \ref{property:NS_holo_2b} must hold. 

So consider a flowline with initial $t$-coordinate in $[0,\epsilon]$ which is not trapped in backward time, and assume at some point it reaches $\Pi_{\mathrm{side}}$ for the first time. This necessarily occurs where $\frac{\partial F_{\epsilon}}{\partial t} > 0$, which implies that the $t$-coordinate at this moment is contained in $[0,\epsilon]$. The flowline then traverses $\Pi_{\mathrm{side}}$ via $\partial_t - R_{\eta_0}$. By $t$-symmetry of $F_{\epsilon}$, the flowline reaches $\Pi_{\mathrm{top}}$ with $t$-coordinate in $[t_0 - \epsilon, t_0]$. From here we appeal to the dynamics of the vector field $X_{\lambda_{F_{\epsilon}}}$ in the region $\{t_0 - \epsilon\leq t\leq t_0\}$ as described by Proposition ~\ref{prop:2D_smooth_box_fold}. For instance, by \ref{property:2D_trapping_region} of Proposition ~\ref{prop:2D_smooth_box_fold}, any flowline reaching $\tilde{D}_4$ (a region which approximates a large interior part of $\{t_0 - \epsilon\leq t\leq t_0\}$; see Figure ~\ref{fig:pl-regions}) is trapped in backward time. Since the flowline in consideration is not trapped, its $t$-coordinate upon reaching $\Pi_{\mathrm{top}}$ must be within $\delta_{\mathrm{tip}}^2$ of $t_0$. By the proof of \ref{property:2D_holonomy_top} of Proposition ~\ref{prop:2D_smooth_box_fold}, the resulting flowline remains in the region $\{t_0 - \epsilon\leq t\leq t_0\}$, very close to $t_0$, until it exits the fold. In particular, the flowline never reaches $\Pi_{\mathrm{side}}$ for a second time, as it can only do so from the region $\{0\leq t\leq \epsilon\}$. This gives \ref{property:NS_holo_2b} and completes the proof.
\end{proof}

\subsection{Smooth box folds}\label{subsec:smooth}

Finally, we complete the definition of the smooth box fold. We proceed in the same way as in \ref{subsec:smooth2d}: first, we define a graphical, piecewise smooth approximation to the nearly smooth fold $\Pi^{NS}$ by slightly ``tipping'' $\Pi_{\mathrm{side}}$, and then we perform the final step of smoothing with arbitrary precision via a convolution. 

It now suffices to work in a model where we further localize to a neighborhood of the boundary of $W_0$:
\[
\left(M_0:= [0,z_0] \times [0,s_0] \times [0,t_0] \times [-r_0,0] \times \Gamma_0,\, \alpha_0 = dz + e^s(dt + e^r\, \eta_0) \right).
\]
where $(\Gamma_0, \eta_0) = \partial (W_0, \lambda_0)$.

\subsubsection{Graphical approximation}

Let $\epsilon > 0$ be as in \ref{subsec:smooth2d}. Let $0 < \delta_{\mathrm{aux}} < \frac{\epsilon}{2}$ be an additional parameter. Let $\Pi_0 = \{z = F_{\epsilon}(s,t)\}\subset [0,z_0]\times [0,s_0] \times [0,t_0]$ and $\tilde{\Pi}_{\mathrm{top}} := \Pi_0 \times [-r_0,0]\times \Gamma_0\subset M_0$ be as in \ref{subsec:nearlysmooth}, where we have introduced a temporary $\sim$ to decorate the former top side. Finally, define 
\[
\tilde{\Pi}_{\mathrm{side}}^{\delta_{\mathrm{aux}}} := \{z = -\frac{z_0}{\epsilon - 2\delta_{\mathrm{aux}}}(r+\delta_{\mathrm{aux}})\} = \{-\delta_{\mathrm{aux}} - mz = r\} \subset M_0
\]
where in the last expression we have set $m := \frac{\epsilon-2\delta_{\mathrm{aux}}}{z_0}$. See Figure \ref{fig:tip1} for a side profile.

\begin{figure}[ht]
	\begin{overpic}[scale=0.19]{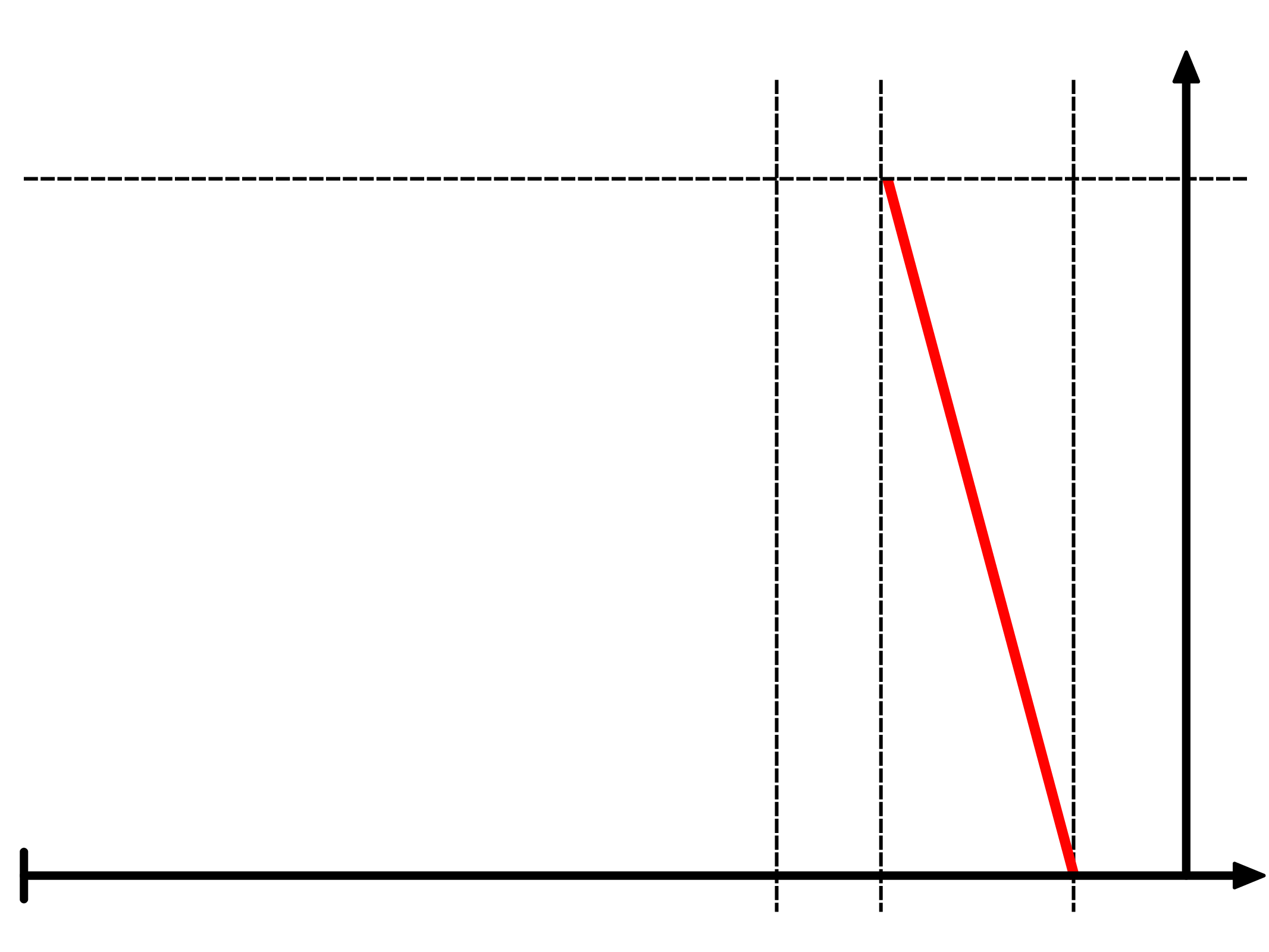}
	    \put(-2,0){\small $-r_0$}
        \put(57,-0.5){\small $-\epsilon$}
     \put(63,70){\tiny $-\epsilon + \delta_{\mathrm{aux}}$}
     \put(81,70){\tiny $-\delta_{\mathrm{aux}}$}
	    \put(74,53){\small \textcolor{red}{$\tilde{\Pi}_{\mathrm{side}}^{\delta_{\mathrm{aux}}}$}}
     \put(100.75,4.75){\small $r$}
     \put(100.25,60){\small $z_0$}
     \put(92,73){\small $z$}
	\end{overpic}
	\caption{The tipped side $\tilde{\Pi}_{\mathrm{side}}^{\delta_{\mathrm{aux}}}$.}
	\label{fig:tip1}
\end{figure}

Note that the intersection of the hypersurfaces $\tilde{\Pi}_{\mathrm{top}}$ and $\tilde{\Pi}_{\mathrm{side}}^{\delta_{\mathrm{aux}}}$ is given by $\tilde{\Pi}_{\mathrm{top}} \cap \tilde{\Pi}_{\mathrm{side}}^{\delta_{\mathrm{aux}}} = \{-\delta_{\mathrm{tip}} - mF_{\epsilon}(s,t) = r\}$. Thus, we define the two sides of the graphical approximation inside our model contact manifold as follows (without $\sim$ decorations): 
\begin{align*}
    \Pi_{\mathrm{top}} &:= \{z = F_{\epsilon}(s,t), \, r \leq -\delta_{\mathrm{aux}} - mF_{\epsilon}(s,t)\}, \\
    \Pi_{\mathrm{side}}^{\delta_{\mathrm{aux}}} &:= \{0 \leq z \leq F_{\epsilon}(s,t),\,  r = -\delta_{\mathrm{aux}}- mz\}.
\end{align*}
Note that $\Pi_{\mathrm{top}}$ as defined here is simply a subset of $\Pi_{\mathrm{top}}$ as defined in the nearly smooth case. In particular, the characteristic foliation on this side is still given by Lemma \ref{lemma:NS_cf}. The characteristic foliation of $\Pi_{\mathrm{side}}^{\delta_{\mathrm{aux}}}$ is given by the following lemma. 

\begin{lemma}\label{lemma:CF_high_side_tilt}
Let $\Pi_{\mathrm{side}}^{\delta_{\mathrm{aux}}}$ be constructed as above. The backward oriented characteristic foliation is directed by 
\[
X_{\mathrm{side}}^{\delta_{\mathrm{aux}}} := \partial_t - e^{\delta_{\mathrm{aux}} + mz}\, R_{\eta_0} - me^s\, \partial_s
\]
where $m = \frac{\epsilon-2\delta_{\mathrm{aux}}}{z_0}$ and $R_{\eta_0}$ is the Reeb vector field of the contact form $\eta_0$ on $\Gamma_0$. 
\end{lemma}

\begin{proof}
The computation is the same as all other computations appealing to Lemma \ref{cflemma}, so we spare many of the details. On $\Pi_{\mathrm{side}}^{\delta_{\mathrm{aux}}}$ we use the following volume form: 
\[
\Omega = (n-1)e^{(n-1)(s - \delta_{\mathrm{aux}} - mz)}\, dz\, ds\, dt\, \eta_0\, (d\eta_0)^{n-2}.
\]
The restriction of the contact form to $\Pi_{\mathrm{side}}^{\delta_{\mathrm{aux}}}$ is $\beta = dz + e^s(dt + e^{-\delta_{\mathrm{aux}} - mz}\, \eta_0)$ and one may verify that $\beta\, (d\beta)^{n-1}$ is given by
\[
(n-1)e^{(n-1)(s - \delta_{\mathrm{aux}} - mz)}\left(me^s \, dz\, dt\, \eta_0\, (d\eta_0)^{n-2} + e^{\delta_{\mathrm{aux}}+ mz}\, dz\, ds\, dt\, (d\eta_0)^{n-2} + dz\, ds\, \eta_0\, (d\eta_0)^{n-2}\right).
\]
Lemma \ref{cflemma} then implies that $X_{\mathrm{side}}^{\delta_{\mathrm{aux}}} = \partial_t - e^{\delta_{\mathrm{aux}} + mz}\, R_{\eta_0} - me^s\, \partial_s$ directs the characteristic foliation. 
\end{proof}

\begin{remark}\label{remark:smooth_tip_eps}
The characteristic foliation $\partial_t - e^{\delta_{\mathrm{aux}} + mz}\, R_{\eta_0} - me^s\, \partial_s$ is a slight deformation of the foliation in the nearly smooth case given by $\partial_t - R_{\eta_0}$, and we can quantify the effect of the deformation at the level of holonomy. First, note that $0 < \delta_{\mathrm{aux}} + mz < \epsilon$ and so the coefficient in front of $-R_{\eta_0}$ is bounded between $1$ and $e^{\epsilon} \approx 1 + \epsilon$. Second, note that 
\[
me^{s} = \left(\frac{\epsilon - 2\delta_{\mathrm{aux}}}{z_0}\right)e^s \leq \frac{\epsilon}{z_0}e^{s_0},
\]
so the total holonomy in the negative $s$-direction induced by $\Pi_{\mathrm{side}}^{\delta_{\mathrm{aux}}}$ is bounded by $t_0 \cdot \frac{\epsilon}{z_0}e^{s_0}$. With our usual assumption that $z_0 \geq e^{s_0}t_0$, this bounds the total negative $s$-direction holonomy along $\Pi_{\mathrm{side}}^{\delta_{\mathrm{aux}}}$ by $\epsilon$. 
\end{remark}

\subsubsection{Smooth approximation}

Note that the graphical, piecewise smooth hypersurface $\Pi_{\mathrm{top}} \cup \Pi_{\mathrm{side}}^{\delta_{\mathrm{aux}}}$ is given by the graph of the function $G_{\delta_{\mathrm{aux}}}^{PS}:[0,s_0]\times [0,t_0]\times [-r_0, 0] \times \Gamma_0 \to [0,z_0]$ defined as 
\begin{equation}\label{eq:defG}
G_{\delta_{\mathrm{aux}}}^{PS}(s,t,r,p) := \begin{cases}
    F_{\epsilon}(s,t) & r\leq -\delta_{\mathrm{aux}} - mF_{\epsilon}(s,t) \\
    -\frac{1}{m}(r + \delta_{\mathrm{aux}}) & -\delta_{\mathrm{aux}} - mF_{\epsilon}(s,t) \leq r\leq -\delta_{\mathrm{aux}} \\
    0 & -\delta_{\mathrm{aux}} \leq r \leq 0
\end{cases}.    
\end{equation}
The locus of points in the domain where $G_{\delta_{\mathrm{aux}}}^{PS}$ is not smooth is given by $S = \{r=-\delta_{\mathrm{aux}}- mF_{\epsilon}(s,t)\} \cup \{r=-\delta_{\mathrm{aux}}\}$. See Figure \ref{fig:tip2} and Figure \ref{fig:bigfig}.

\begin{figure}[ht]
	\begin{overpic}[scale=0.33]{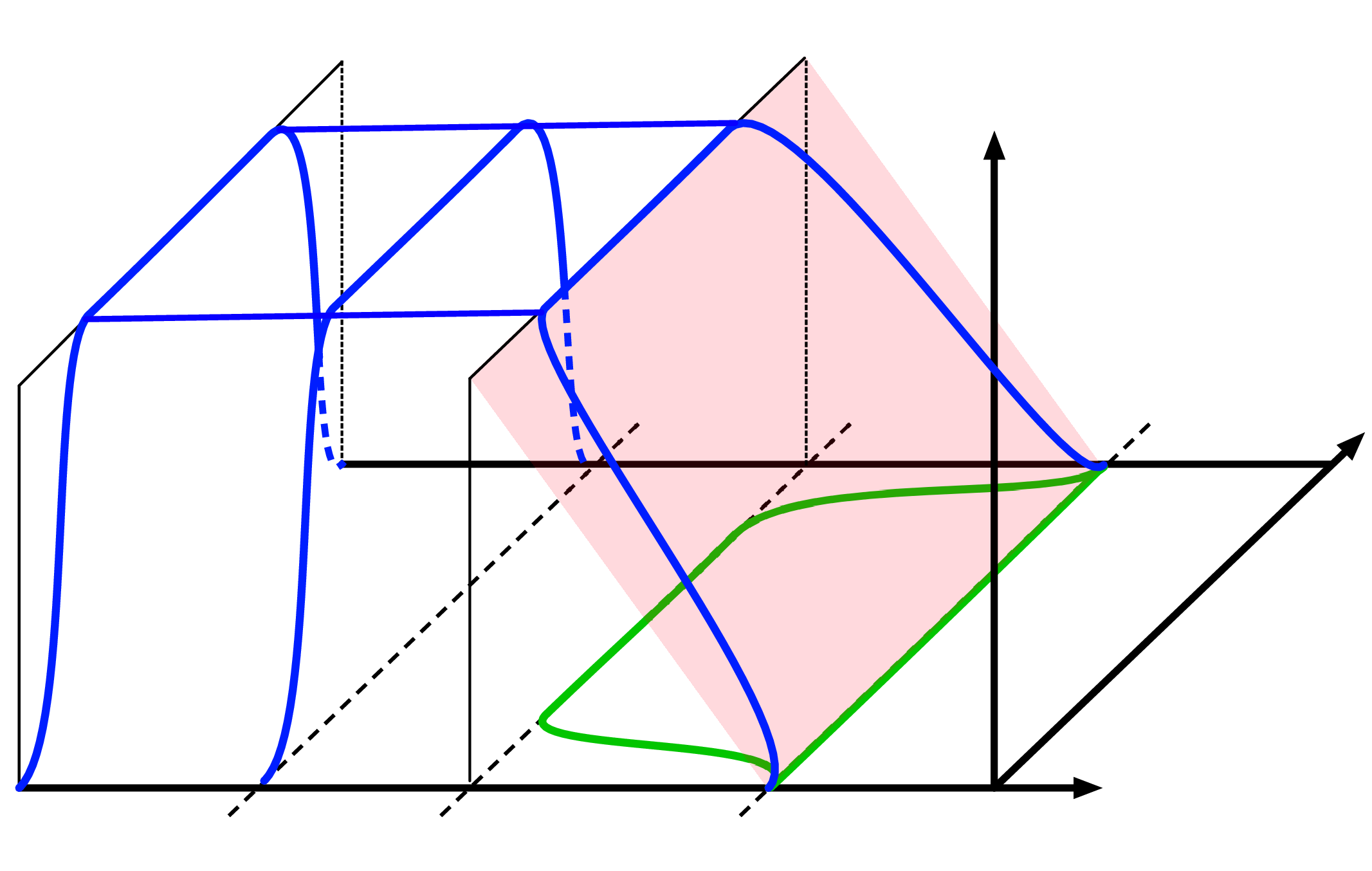}
	    \put(-2,3){\small $-r_0$}
        \put(13,1.5){\small $-\epsilon$}
     \put(27,1.5){\small $-\epsilon + \delta_{\mathrm{aux}}$}
     \put(49,1.5){\small $-\delta_{\mathrm{aux}}$}
	    \put(77,38){\small \textcolor{red}{$z = -\frac{1}{m}(r+\delta_{\mathrm{aux}})$}}
     
     \put(81,5.5){\small $r$}
     \put(72,55){\small $z$}
     \put(100,33){\small $s,t$}

     \put(67,12){\small \textcolor{green}{$S$}}
	\end{overpic}
	\caption{The graph of $G_{\delta_{\mathrm{aux}}}^{PS}$ indicated in blue. Shaded in red is the hyperplane $z = -\frac{1}{m}(r+\delta_{\mathrm{aux}})$. In green is the set $S\subset \{z=0\}$ where $G_{\delta_{\mathrm{aux}}}^{PS}$ is not smooth.}
	\label{fig:tip2}
\end{figure}

As in the $2$-dimensional case, let $\varphi:\R^3 \to [0, \infty)$ be the smooth cutoff function defined by 
\[
\varphi(s,t,r) = \begin{cases}
    k \exp\left(- \dfrac{1}{1 - s^2 - t^2 - r^2}\right) & s^2 + t^2 + r^2 < 1 \\
    0 & s^2 + t^2 + r^2 \geq 1
\end{cases}
\]
this time including the $r$-direction. (Note that there is no need to cutoff in the $\Gamma_0$-direction since $\Gamma_0$ is a closed manifold.) As before, choose $k>0$ so that $\int_{\R^3} \varphi\, dA = 1$. Let $\delta_{\mathrm{con}} > 0$ be as in \ref{subsec:smooth2d}, and assume without loss of generality that $\delta_{\mathrm{con}} \ll \min(\frac{\epsilon}{2} - \delta_{\mathrm{aux}}, \delta_{\mathrm{aux}})$. Set $\varphi_{\delta_{\mathrm{con}}}(s,t,r) := \frac{1}{\delta_{\mathrm{con}}^3}\varphi(\frac{s}{\delta_{\mathrm{con}}}, \frac{t}{\delta_{\mathrm{con}}}, \frac{r}{\delta_{\mathrm{con}}})$. Finally, define $G_{\epsilon} := G_{\delta_{\mathrm{aux}}}^{PS} * \varphi_{\delta_{\mathrm{con}}}$, where in the convolution we only integrate over $\R^3$. 

\begin{remark}[Accumulation of convolutions]\label{remark:accumulation}
Fix a point $w_0\in W_0$ such that $r(w_0) < -\epsilon$ and consider the function $\tilde{F}_{\epsilon}(s,t):= G_{\epsilon}(s,t,w_0)$. Strictly speaking, the convolution defining $G_{\epsilon}$ compounds the initial convolution defining $F_{\epsilon}$, so $\tilde{F}_{\epsilon}$ differs from $F_{\epsilon}$ in a $\delta_{\mathrm{con}}$-neighborhood of the first $\delta_{\mathrm{con}}$-neighborhood of the corners in \ref{subsec:smooth2d}. In particular, $\tilde{F}_{\epsilon}$ differs from $F_{\delta_{\mathrm{tip}}}^{PL}$ in a neighborhood of the corners of diameter $4\delta_{\mathrm{con}}$, rather than of diameter $2\delta_{\mathrm{con}}$. It is clear from the analysis in \ref{subsec:smooth2d} that this makes no real difference, in particular since $\delta_{\mathrm{con}}$ may be chosen arbitrarily small. We proceed without further acknowledging this fact or and make no distinction between $F_{\epsilon}$ and $\tilde{F}_{\epsilon}$. 
\end{remark}

\subsection{Proof of Theorem \ref{theorem:new_smooth_box_fold}}\label{subsec:smooth_proof}

In this subsection we prove Theorem \ref{theorem:new_smooth_box_fold} for the function $G_{\epsilon}$ constructed above. First we denote by $\hat{S}\subset [0,s_0] \times [0,t_0] \times [-r_0, 0] \times \Gamma_0$ a $\delta_{\mathrm{con}}$-neighborhood of $S = \{r=-\delta_{\mathrm{aux}}- mF_{\epsilon}(s,t)\} \cup \{r=-\delta_{\mathrm{aux}}\}$, the set of points where $G_{\delta_{\mathrm{aux}}}^{PS}$ fails to be smooth. This is (effectively, see Remark \ref{remark:accumulation}) the domain on which $G_{\epsilon}$ fails to be linear. Then define the following further subsets of $[0,s_0] \times [0,t_0] \times [-r_0, 0] \times \Gamma_0$: 
\begin{align*}
    S_{\mathrm{top}} &:= \{r\leq \delta_{\mathrm{aux}} - mF_{\epsilon}(s,t)\} - \hat{S},\\
    S_{\mathrm{side}} &:= \{\delta_{\mathrm{aux}} - mF_{\epsilon}(s,t) \leq r\leq -\delta_{\mathrm{aux}}\} - \hat{S}.
\end{align*}
These are the parts of the first two domains in \eqref{eq:defG} away from $\hat{S}$. In Figure \ref{fig:bigfig}, $S_{\mathrm{side}}$ is the interior of the green prism.

With this notation we record the Liouville vector field (and vector fields directing the backward flow) on the regions $S_{\mathrm{top}}, S_{\mathrm{side}}$, and $\hat{S}$.

\begin{lemma}\label{lemma:big_cf}
The Liouville vector field $X_{\lambda_{G_{\epsilon}}}$ on the neighborhood $[0,s_0] \times [0,t_0] \times [-r_0, 0]\times \Gamma_0$ is given by 
\begin{equation}\label{eq:LVF_high_new}
    X_{\lambda_{G_{\epsilon}}} = \left(1 + e^{-s}\frac{\partial G_{\epsilon}}{\partial t}\right)\, \partial_s - e^{-s}\frac{\partial G_{\epsilon}}{\partial s}\, \partial_t - e^{-s}\frac{\partial G_{\epsilon}}{\partial t}\, X_{\lambda_0} - e^{-s}\frac{\partial G_{\epsilon}}{\partial r}\left(-\partial_t + e^{-r}\, R_{\eta_{0}}\right)
\end{equation}
where $X_{\lambda_0} = \partial_r$. In particular, the vector field $-X_{\lambda_{G_{\epsilon}}}$ is positively parallel to the following vector fields on the corresponding regions.
\begin{enumerate}
    \item In $S_{\mathrm{top}}$:\label{formula:Stop}
    \[
    X_{\mathrm{top}}:= -\left(e^s + \frac{\partial F_{\epsilon}}{\partial t}\right)\, \partial_s + \frac{\partial F_{\epsilon}}{\partial s}\, \partial_t + \frac{\partial F_{\epsilon}}{\partial t}\, X_{\lambda_0}.
    \]
    \item In $S_{\mathrm{side}}$:\label{formula:Sside}
    \[
    X_{\mathrm{side}}:= \partial_t - e^{-r}\, R_{\eta_0} - me^s\, \partial_s
    \]
    where $m= \frac{\epsilon - 2\delta_{\mathrm{aux}}}{z_0}$.
    \item Everywhere, but in particular in $\hat{S}$:\label{formula:Shat}
    \[
    \hat{X}:= -\left(e^s + \frac{\partial G_{\epsilon}}{\partial t}\right)\, \partial_s + \frac{\partial G_{\epsilon}}{\partial s}\, \partial_t + \frac{\partial G_{\epsilon}}{\partial t}\, X_{\lambda_0} +\left|\frac{\partial G_{\epsilon}}{\partial r}\right|\left(\partial_t - e^{-r}\, R_{\eta_{0}}\right).
    \]
    In particular, $\hat{X}$ is an interpolation between $X_{\mathrm{top}}$ and $X_{\mathrm{side}}$. 
\end{enumerate}
\end{lemma}

\begin{proof}
One can verify by an explicit calculation that $\iota_{X_{\lambda_{G_{\epsilon}}}} d\lambda = \lambda$, where $\lambda = dG_{\epsilon} + e^s(dt + e^{r}\, \eta_0)$. This gives \eqref{eq:LVF_high_new}. 

The formula in \ref{formula:Stop} follows from the fact that $\frac{\partial G_{\epsilon}}{\partial r} = 0$ in $S_{\mathrm{top}}$, together with the observation in Remark \ref{remark:accumulation}. Note that $X_{\mathrm{top}}$ is exactly the vector field that direct the characteristic foliation of $\Pi_{\mathrm{top}}$ in the nearly smooth fold, as expected. One may similarly prove \ref{formula:Sside} by way of Lemma \ref{lemma:CF_high_side_tilt} by observing that $X_{\mathrm{side}}$ agrees with the vector field in the lemma, up to choice of coordinates. Alternatively, in $S_{\mathrm{side}}$, \eqref{eq:LVF_high_new} gives 
    \[
    -X_{\lambda_{G_{\epsilon}}}   = -\partial_{s} + e^{-s} \frac{1}{m} \left(\partial_t - e^{-r}\, R_{\eta_{0}}\right)
    \]
    from which we get $X_{\mathrm{side}}$ after a further rescaling by $me^s$. Finally, \ref{formula:Shat} is immediate from \eqref{eq:LVF_high_new}.
\end{proof}

\begin{proof}[Proof of Theorem \ref{theorem:new_smooth_box_fold}.]
First we address the Weinstein compatibility statement \ref{property:high_Weinstein_compat}, which is straightforward. By Lemma \ref{lemma:big_cf}, critical points of $X_{\lambda_{G_{\epsilon}}}$ must occur when $\frac{\partial G_{\epsilon}}{\partial r} = 0$ and as such it suffices to identify critical points of a nearly smooth fold $\Pi^{NS}$ on $\Pi_{\mathrm{top}}$. By the nearly smooth characteristic foliation in Lemma \ref{lemma:NS_cf} together with \ref{property:2d_Weinstein_compat} of Proposition \ref{prop:2D_smooth_box_fold}, it follows that for every critical point of index $k$ of $X_{\lambda_0}$ in $W_0$, there are two critical points of index $k$ and $k+1$ of $X_{\lambda_{G_{\epsilon}}}$. The Morse property is also immediate from the Morse property of Proposition \ref{prop:2D_smooth_box_fold} and the fact that $(W_0, \lambda_0)$ is Weinstein.

\vspace{2mm}
\noindent \textit{Trapping properties \ref{property:high_trap2}.}
\vspace{2mm}

First, assume that a flowline enters $U_{\mathrm{trap},1}^{\epsilon}$, so that in particular the $W_0$-coordinate of entry is outside of $N^{s_0+\epsilon}(\partial W_0)$. For such a flowline, the $W_0$-norm estimate for nearly smooth folds in Proposition \ref{prop:NS_holo_norm} still applies and in particular there is no interaction with the $\epsilon$-neighborhood of $\partial W_0$ supporting the final tipping and convolution that defines $G_{\epsilon}$. The argument for \ref{property:NS_trap_1} of Proposition \ref{prop:NS_trap}, word for word, then identifies $U_{\mathrm{trap},1}^{\epsilon}$ as a trapping region in Theorem \ref{theorem:new_smooth_box_fold}.

The second trapping region $U_{\mathrm{trap},2}^{\epsilon}$ is the smooth counterpart to the nearly smooth region in \ref{property:NS_trap_2} of Proposition \ref{prop:NS_trap}. In the smooth case there are two additional factors to consider, one of which is the ``tipping'' of $\Pi_{\mathrm{side}}$ governed by $\delta_{\mathrm{aux}}$, and the other is the convolution smoothing governed by $\delta_{\mathrm{con}}$, resulting in the interpolation in $\hat{S}$ between $X_{\mathrm{top}}$ and $X_{\mathrm{side}}$. The first factor is the reason for the exponent $s_0 - 2\epsilon$. Indeed, by Remark \ref{remark:smooth_tip_eps}, there is an at-worst $s$-holonomy of magnitude $\epsilon$ as the flowline traverses $\Pi_{\mathrm{side}}^{\mathrm{aux}}$. By entering the fold in $N^{s_0 - 2\epsilon}(\partial W_0)$ as opposed to $N^{s_0 - \epsilon}(\partial W_0)$, this is accounted for and the argument in the proof of \ref{property:NS_trap_2} of Proposition \ref{prop:NS_trap} holds.

\begin{figure}[ht]
	\begin{overpic}[scale=0.44]{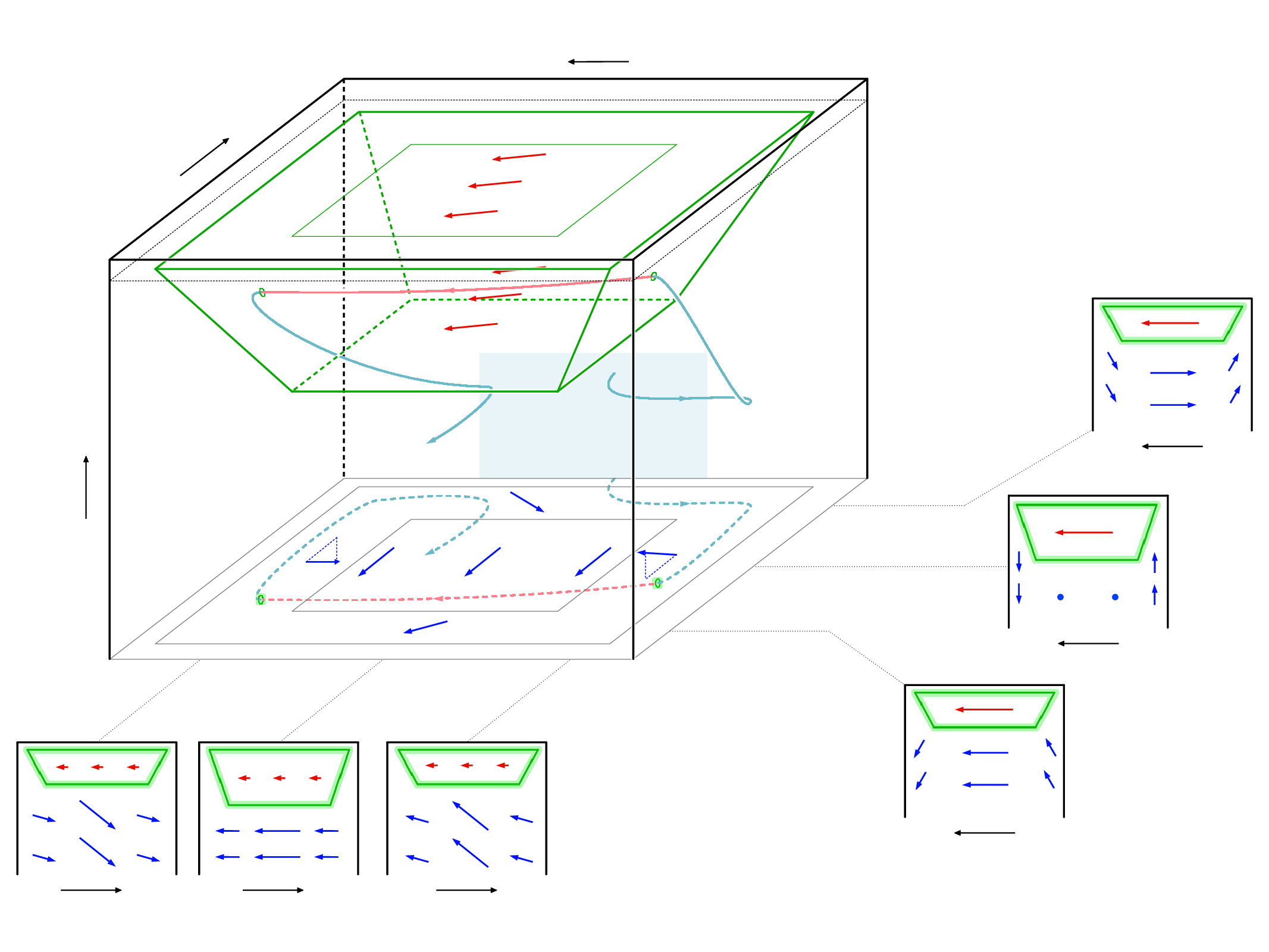}
	    \put(10.2,4.5){\small $s$}
     \put(24.5,4.5){\small $s$}
     \put(39.75,4.5){\small $s$}
    
     \put(73.75,9){\small $t$}
     \put(82,24){\small $t$}
     \put(88.5,39.5){\small $t$}

     \put(6.5,40){\small $r$}
     \put(18.75,64.5){\small $s$}
     \put(43.5,69.75){\small $t$}

     \put(42,60){\small \textcolor{red}{$X_{\mathrm{side}}$}} 
     \put(40,30){\small \textcolor{blue}{$X_{\mathrm{top}}$}} 
     
	\end{overpic}
	\caption{The dynamics of a chimney fold in the $[0,s_0]\times [0,t_0] \times[-r_0,0]$ projection. The green $2$-dimensional prism in the main figure represents $S$, the set of points where $G^{PS}_{\delta_{\mathrm{aux}}}$ is not smooth. In the various constant $t$-slices (below) and constant $s$-slices (right), the thicker shaded green region is $\hat{S}$. A sample flowline in light blue enters $U_{\mathrm{trap},2}^{\epsilon}$ (the rectangle shaded in light blue), then travels through $S_{\mathrm{side}}$ in pink before returning to $S_{\mathrm{top}}$ in light blue and eventually getting trapped. The projection of the flowline to the $(s,t)$-plane is given by the dashed curve.}
	\label{fig:bigfig}
\end{figure}

These consequences are most easily understood via the visualization in Figure \ref{fig:bigfig}. Figure \ref{fig:bigfig} depicts the region $[0,s_0]\times [0,t_0]\times [-r_0,0] \times \Gamma_0$ with the $\Gamma_0$-direction suppressed. The set $S$ where $G_{\delta_{\mathrm{aux}}}^{PS}$ fails to be smooth is outlined in green; $\hat{S}$ is simply a $\delta_{\mathrm{con}}$-small tubular neighborhood of this complex. The interior prism region contained inside $\hat{S}$ is $S_{\mathrm{side}}$. Various slices of the dynamics ($X_{\mathrm{side}}$ in red, $X_{\mathrm{top}}$ in blue) are indicated, and $\hat{X}$ is the natural interpolation of these vector fields in $\hat{S}$. 

The light-blue-and-pink trajectory represents a flowline entering $U_{\mathrm{trap},2}^{\epsilon}$. By entering sufficiently below the prism, it avoids the potential interference arising from the interpolation in $\hat{S}$ in the region indicated by the top right constant $s$-slice. On the other hand, by entering sufficiently close to to the prism, it traverses enough of the interior to emerge in the part of $S_{\mathrm{top}}$ where Step 2 of the proof of \ref{property:NS_trap_2} of Proposition \ref{prop:NS_trap} holds.

\vspace{2mm}
\noindent \textit{Holonomy properties \ref{property:high_hol3}.}
\vspace{2mm}

We first consider the $W_0$-norm estimate in \ref{thm-part:new_smooth_holonomy}. As remarked above, if the flowline avoids $N^{\epsilon}(\partial W_0)$ then Proposition \ref{prop:NS_holo_norm} gives the desired estimate without further remark. Otherwise, by Lemma \ref{lemma:big_cf}, the additional decay in the $r$-direction in $N^{\epsilon}(\partial W_0)$ of the function $G_{\epsilon}$ only introduces movement parallel to $\partial_t - e^{-r} R_{\eta_0}$, which does not affect the $W_0$-norm and hence Proposition \ref{prop:NS_holo_norm} still implies \ref{thm-part:new_smooth_holonomy}.

Similarly, \ref{property:high_hol3b} of Theorem \ref{theorem:new_smooth_box_fold} follows from the corresponding nearly smooth statement \ref{property:NS_holo_1} of Proposition \ref{prop:NS_holonomy}. According to Lemma \ref{lemma:big_cf}, the only difference in the dynamics of the nearly smooth fold with the backward flow of $X_{\lambda_{G_{\epsilon}}}$ is the term $\left|\frac{\partial G_{\epsilon}}{\partial r}\right|(\partial_t - e^{-r} R_{\eta_0})$. This only serves to increase the $t$-coordinate. Thus, if $t_{\mathrm{init}}\in [t_0 - \epsilon, t_0]$, this remark together with Proposition \ref{prop:NS_holonomy} implies that $t_{\mathrm{fin}}\in [t_0 - \epsilon, t_0]$.

Finally, \ref{property:high_hol3c} is the smooth statement corresponding to \ref{property:NS_holo_2} of Proposition \ref{prop:NS_holonomy}. It is slightly less precise, but it is sufficient for our purposes. Assume that $t_{\mathrm{init}}\in [0,\epsilon]$. We again use Figure \ref{fig:bigfig} as a reference and refer to the prism region $S_{\mathrm{side}}\cup \hat{S}$. The flowline either passes through the prism, or it doesn't. If the flowline does not pass through the prism, its dynamics are entirely dictated by $X_{\mathrm{top}}$ and \ref{property:2D_holonomy_bound} of Proposition \ref{prop:2D_smooth_box_fold} implies that $t_{\mathrm{fin}}\leq e^{s_0 + \epsilon}\epsilon$. This means that to achieve $t_{\mathrm{fin}}\gg e^{s_0 + \epsilon}\epsilon$ as in the assumption of \ref{property:high_hol3c}, it is necessary for the flowline to cover a significant $t$-distance inside the prism. 

In traversing a significant $t$-distance inside the prism, it can either travel through $S_{\mathrm{side}}$ or $\hat{S}$. We consider each situation separately. 

\vspace{2mm}
\textit{Traversal through $S_{\mathrm{side}}$.}
\vspace{2mm}

 Let $\Delta_{\mathrm{side}}t$ denote the $t$-holonomy arising from passage through $S_{\mathrm{side}}$, i.e., the strict interior of the prism. Here the flow is directed by $X_{\mathrm{side}} = \partial_t - e^{-r}\, R_{\eta_0} - me^s\, \partial_s$. Since $S_{\mathrm{side}} \subset \{r>-\epsilon\}$, while traveling along $S_{\mathrm{side}}$ the flowline experiences a flow in the $-R_{\eta_0}$-direction for some time $t_{\mathrm{side}}$, where $\Delta_{\mathrm{side}} t\leq t_{\mathrm{side}}\leq e^{\epsilon}\Delta_{\mathrm{side}} t$; see also Remark \ref{remark:smooth_tip_eps}.

\vspace{2mm}
\textit{Traversal through $\hat{S}$.}
\vspace{2mm}

Since we are only concerned with the flowline covering a large $t$-distance, we only need to consider the four regions of $\hat{S}$ in Figure \ref{fig:bigfig} which are those corresponding to the neighborhoods of the faces closest to $r=0$ (the top face), $r=-r_0$ (the bottom face), $s=0$ (the near face), and $s=s_0$ (the far face). Indeed, the remaining two regions have $t$-thickness on the order of $\delta_{\mathrm{con}}$.

By definition of $G_{\epsilon}$, in the regions of $\hat{S}$ near the top face and the bottom face, $\frac{\partial G_{\epsilon}}{\partial s}=0$ in their large interior parts. Therefore, in the vector field 
\[
\hat{X}= -\left(e^s + \frac{\partial G_{\epsilon}}{\partial t}\right)\, \partial_s + \frac{\partial G_{\epsilon}}{\partial s}\, \partial_t + \frac{\partial G_{\epsilon}}{\partial t}\, X_{\lambda_0} +\left|\frac{\partial G_{\epsilon}}{\partial r}\right|\left(\partial_t - e^{-r}\, R_{\eta_{0}}\right)
\]
directing the flow, the $t$-holonomy coming from the second term is negligible, so any $t$-holonomy is accompanied by $-R_{\eta_0}$-flow of comparable time, just as in $S_{\mathrm{side}}$. In particular, with $\hat{\Delta} t$ denoting the $t$-holonomy through the regions near the top face and the bottom face where $\frac{\partial G_{\epsilon}}{\partial s}=0$, the flowline experiences a flow in the $-R_{\eta_0}$-direction for some time $\hat{t}$, where $\hat{\Delta} t\leq \hat{t}\leq e^{\epsilon}\hat{\Delta} t$ by the same argument as in $S_{\mathrm{side}}$. 

Finally, via control of the $\delta_{\mathrm{con}}$ parameter it is actually not possible to traverse a large $t$-distance in the regions of $\hat{S}$ near the near face and far face. In these regions, $\frac{\partial G_{\epsilon}}{\partial t}$ is either $0$ or controllably small via $\delta_{\mathrm{con}}$, so the $\partial_s$-component $\hat{X}$ is close (via $\delta_{\mathrm{con}}$) to $-e^s\, \partial_s$. As the regions of $\hat{S}$ near the near face and far face have $s$-thickness on the order of $\delta_{\mathrm{con}}$, any $t$-holonomy here is negligible compared to other parameters, so we will ignore it. 

Altogether, the $t$-holonomy experienced by the flowline as it passes through the prism is controllably close (via $\delta_{\mathrm{con}}$) to $\Delta_{\mathrm{side}}t + \hat{\Delta} t$. Consequently, the flowline experiences a total holonomy in the $-R_{\eta_0}$-direction of time $t_*$, where $t_*$ satisfies 
\[
t_* > t_{\mathrm{side}} + \hat{t} \geq \Delta_{\mathrm{side}}t + \hat{\Delta} t.
\]
As $t_{\mathrm{fin}} \gg e^{s_0 + \epsilon}\epsilon$ necessitates a significant $t$-distance traversed through the prism, we may assume $\Delta_{\mathrm{side}}t + \hat{\Delta} t \gg \epsilon$. Moreover, we have $t_*\leq e^{\epsilon}(\Delta_{\mathrm{side}}t + \hat{\Delta} t) \leq e^{\epsilon}t_0.$ This gives 
\[
\epsilon \ll t_* \leq e^{\epsilon} t_0
\]
as desired.
\end{proof}

\section{Chimney folds}\label{section:chimney_folds}

A box fold is based over the symplectization of a contact handlebody, i.e., a Reeb-thickened Weinstein domain. A \textit{chimney fold} is a generalization wherein the fold is based over the symplectization of a more complicated contact manifold with boundary. The specific shape of the base is chosen to alter the shape of the trapping region in a way that allows us to prove Theorem~\ref{prop:mainprop}, as will become clear in Section \ref{section:proof_main_prop}. As with box folds, we begin by defining chimney folds in piecewise linear form in \ref{subsection:PL_chimney_fold} and discuss their smoothing in \ref{subsection:smooth_chimney}. In particular, in \ref{subsection:smooth_chimney} we state and prove Proposition \ref{prop:new_chimney_smooth}, which is the centerpiece of the section and the analogue of Theorem \ref{theorem:new_smooth_box_fold} for chimney folds.

\subsection{Piecewise linear chimney folds}\label{subsection:PL_chimney_fold}

For most of this subsection, we focus on chimney folds defined on $4$-dimensional Liouville domains, which is sufficient to describe the interesting aspects of their design. In \ref{subsubsec:high_dim_chim} we extend chimney folds to arbitrary dimensions, which is a straightforward process and does not introduce any significant complications. 

Fix $z_0, s_0, t_0 > 0$ and assume that $z_0 = e^{s_0}t_0$. Let $(W_{\mathrm{Ch}}, \lambda_{\mathrm{Ch}})$ be a $2$-dimensional Weinstein domain with connected boundary, so that $\partial W_{\mathrm{Ch}} \cong S^1$. Identify a collar neighborhood $N^{s_0}(\partial W_{\mathrm{Ch}})$ as in Sections~\ref{section:box_folds_pl} and~\ref{section:box_folds_smooth} and recall that $N^{s_0}(\partial W_{\mathrm{Ch}})$ is symplectomorphic to $((-s_0,0]_r \times \partial W_{\mathrm{Ch}}, e^r\, \eta_{\mathrm{Ch}})$ where $\eta_{\mathrm{Ch}}$ is the induced contact form on $\partial W_{\mathrm{Ch}}$. Let $\gamma_C \subset \partial W_{\mathrm{Ch}}$ be any connected arc, and define 
\[
C:= [-s_0, 0]_r \times \gamma_C \subset N^{s_0}(\partial W_{\mathrm{Ch}}).
\]
Next, we will identify a set $H_{\mathrm{Ch}}$ inside the contact handlebody $([0,t_0] \times W_{\mathrm{Ch}}, \, dt + \lambda_0)$ that supports the eventual chimney fold. Fix $0 < t_- \ll t_0$ and define 
\[
H_{\mathrm{Ch}} := ([0,t_-] \times W_{\mathrm{Ch}}) \cup ([0,t_0]\times C).
\]
Informally, $H_{\mathrm{Ch}}$ is a generalized contact handlebody with Reeb chords of two different lengths: over $W_{\mathrm{Ch}} \setminus C$ Reeb chords have length $t_-$, and over $C$ they have length $t_0$. Alternatively, $H_{\mathrm{Ch}}$ is an ordinary contact handlebody $[0,t_-]\times W_0$ together with an appended contact region that is the supporting region of a pre-chimney fold as described in \ref{subsection:pre_chimney}. 

\begin{remark}
We refer to the region $[0,t_0]\times C$ as the \textit{chimney} or \textit{chimney region} of $H_{\mathrm{Ch}}$, and to $[0,t_-]\times (W_{\mathrm{Ch}}\setminus C)$ as the \textit{stove} or \textit{stove region}.
\end{remark}

\begin{figure}[ht]
	\begin{overpic}[scale=0.43]{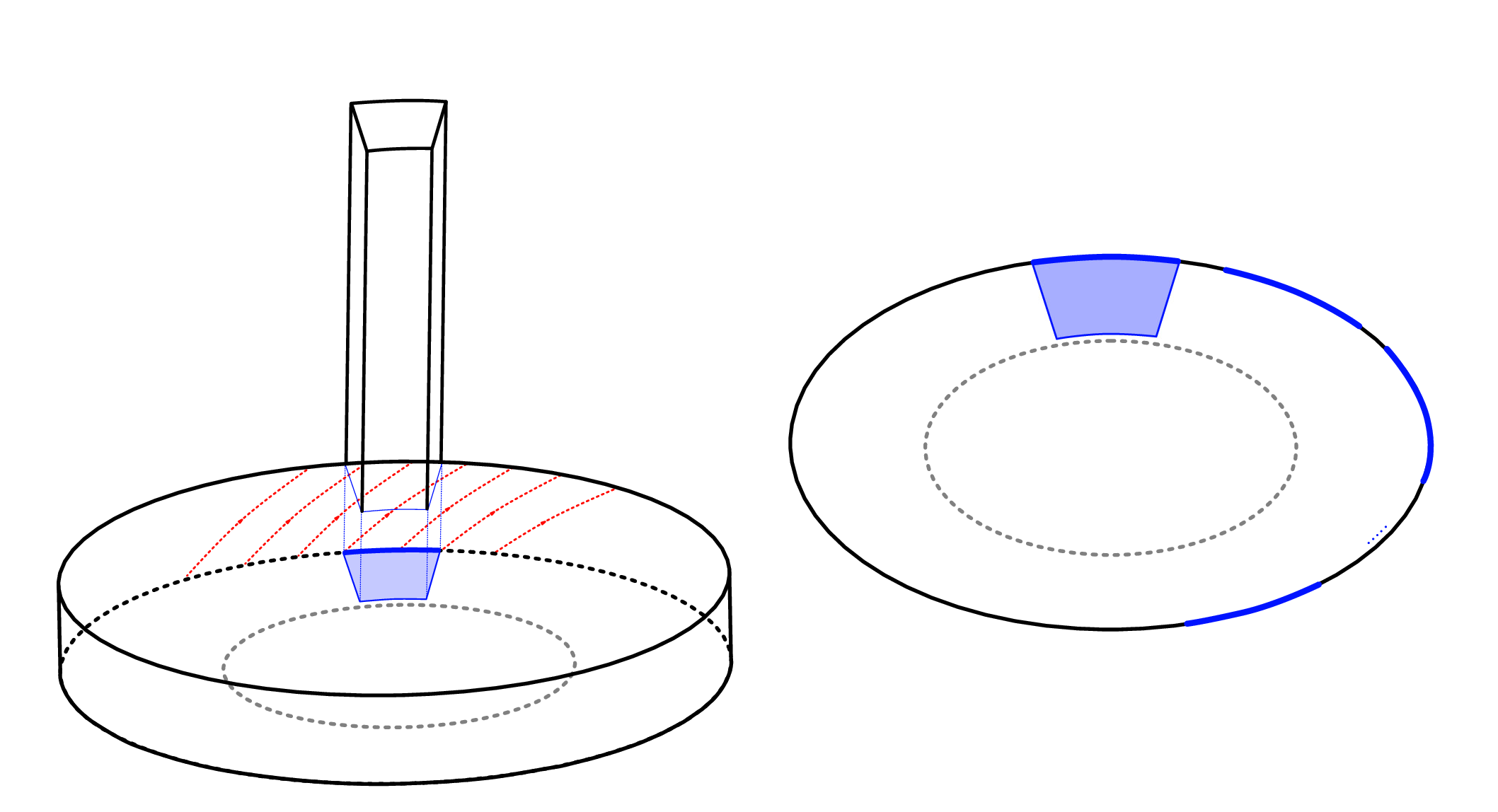}
	\put(-1.5,8){\small $t=0$}
    \put(-2,16){\small $t=t_-$}
    \put(17,46){\small $t=t_0$}
    \put(73,24){\small $W_0$}

    \put(73,34){\small $C$}
    \put(72.5,38.5){\small \textcolor{blue}{$\gamma_C$}}

    \put(82,32.5){\small \textcolor{blue}{$h(\gamma_C)$}}
    \put(88,25.5){\small \textcolor{blue}{$h^2(\gamma_C)$}}
    \put(79,15.5){\small \textcolor{blue}{$h^n(\gamma_C)$}}

    \put(25,-0.5){\small $H_{\mathrm{Ch}}$}
	\end{overpic}
	\caption{The supporting region $H_{\mathrm{Ch}}$ of a chimney fold. On the left, the characteristic foliation $\partial_t - R_{\eta_{\mathrm{Ch}}}$ of $\underline{\partial W_{\mathrm{Ch}}}$ is depicted by the dashed red lines; Assumption \ref{ass:chimney} about iterates of $\gamma_C$ under $h=h_{\partial W_{\mathrm{Ch}}}$ is depicted in the $W_{\mathrm{Ch}}$ projection on the right.}
	\label{fig:chimney}
\end{figure}

For a chimney fold to work as desired, we need to impose one more assumption that requires additional notation to state. Let $h_{\partial W_{\mathrm{Ch}}}: \{t=0\} \times \partial W_{\mathrm{Ch}} \to \{t=t_-\} \times \partial W_{\mathrm{Ch}}$ be the holonomy map given by the flow of the backward oriented characteristic foliation of $[0,t_-] \times \partial W_{\mathrm{Ch}}$ in the stove $([0,t_-] \times W_{\mathrm{Ch}}, \, dt + \lambda_{\mathrm{Ch}})$. Recall that this characteristic foliation is directed by $\partial_t - R_{\eta_{\mathrm{Ch}}}$. Abusing notation, we will write $h_{\partial W_{\mathrm{Ch}}}: \partial W_{\mathrm{Ch}} \to \partial W_{\mathrm{Ch}}$. 

\begin{example}
Suppose that $(W_{\mathrm{Ch}} = r_0\D^2, \lambda_{\mathrm{Ch}} = \frac{1}{2}r^2\, d\theta)$. The backward oriented characteristic foliation of $[0,t_-]\times \partial W_{\mathrm{Ch}}$ is directed by $\partial_t - \frac{2}{r_0^2}\, \partial_{\theta}$, so 
\[
h_{\partial W_{\mathrm{Ch}}}(\theta) = \theta - \frac{2}{r_0^2} t_-.
\]
\end{example}

The key assumption needed to define a chimney fold is the following. 

\begin{assumption}\label{ass:chimney}
    Let $n$ be the smallest integer so that $nt_- > t_0$. Then $\gamma_C \cap h^j_{\partial W_{\mathrm{Ch}}}(\gamma_C) = \emptyset$ for all $1\leq j \leq n$.
\end{assumption}

This assumption ensures two main features of the boundary holonomy $h_{\partial W_{\mathrm{Ch}}}$: 
\begin{enumerate}[label=(\arabic*)]
    \item the map $h_{\partial W_{\mathrm{Ch}}}$ displaces $\gamma_C$, and \label{feature:displace}
    \item the image of $\gamma_C$ does not circle around the boundary back to itself after up to $n$ iterates under $h_{\partial W_{\mathrm{Ch}}}$. \label{feature:image}
\end{enumerate}
See Figure \ref{fig:chimney}. In practice, we will choose $t_- \ll t_0$ to be very small. Thus,~\ref{feature:displace} will be achieved from a combination of $\gamma_C$ being small and $R_{\eta_{\mathrm{Ch}}}$ being large near $\gamma_C$, and~\ref{feature:image} will be achieved by ensuring that $\partial W_{\mathrm{Ch}}$ is long enough. 

\begin{remark}
If we drop the assumption $z_0 = e^{s_0}t_0$ on the initial parameters of the fold, Assumption \ref{ass:chimney} is more generally stated by requiring $n$ to satisfy $ne^{s_0}t_- > z_0$. Assuming $z_0 =e^{s_0}t_0$, this simplifies to $nt_- > t_0$. The cost of increasing $z_0$ further is an increase in the value of $n$. 
\end{remark}

Consider the contactization of the symplectization of $H_{\mathrm{Ch}}$:
\[
\left(\R_z \times [0,s_0] \times H_{\mathrm{Ch}},\, dz + e^s\, (dt + \lambda_{\mathrm{Ch}})\right)
\]

\begin{definition}
Fix $z_0, s_0, t_0 > 0$ with $z_0 = e^{s_0}t_0$, and suppose that $H_{\mathrm{Ch}}$ is defined as above with Assumption \ref{ass:chimney} satisfied. A \textbf{piecewise linear chimney fold with parameters $z_0, s_0, t_0, t_-$}, denoted $C\Pi^{PL}$, is the hypersurface 
\[
C\Pi^{PL} := \overline{\partial ([0,z_0] \times [0,s_0]\times H_{\mathrm{Ch}}) \setminus \{z=0\}}.
\]
\end{definition}

As with piecewise linear box folds, we adopt the following notation to refer to the various sides of $C\Pi^{PL}$:
\begin{align*}
    \underline{z=z_0} &:= \{z=z_0\} \times [0,s_0]\times H_{\mathrm{Ch}} \\
    \underline{s=0} &:= [0,z_0] \times \{s=0\} \times H_{\mathrm{Ch}} \\
\underline{s=s_0} &:= [0,z_0] \times \{s=s_0\} \times H_{\mathrm{Ch}} \\
\underline{t=0} &:= [0,z_0]\times [0,s_0] \times \{t=0\} \times W_{\mathrm{Ch}} \\
\underline{t=t_-} &:= [0,z_0]\times [0,s_0] \times \{t=t_-\} \times (W_{\mathrm{Ch}} \setminus C) \\
\underline{t=t_0} &:= [0,z_0]\times [0,s_0] \times \{t=t_0\} \times C \\ 
\underline{\partial W_{\mathrm{Ch}}} &:= [0,z_0] \times [0,s_0] \times [0,t_-] \times \partial W_{\mathrm{Ch}} \\
\underline{\partial C} &:= [0,z_0] \times [0,s_0] \times [t_-, t_0] \times \partial C.
\end{align*}
As in Section \ref{section:box_folds_pl}, we compute the oriented characteristic foliation of each side.

\vspace{2mm}
\begin{center}
\begin{tabular}{ |c|c|c|c| } 
 \hline
 Side  & Characteristic foliation \rm \\ 
 \hline 
 $\underline{z=z_0}$  & $  -\partial_s$ \\
 $\underline{s=0}$  & $\partial_t -\partial_z $ \\
 $\underline{s=s_0}$  & $-\partial_t + e^{s_0} \partial_z$ \\
 $\underline{t=0}$ & $-\partial_s + X_{\lambda_0}$ \\
 $\underline{t=t_-}$ & $\partial_s - X_{\lambda_0}$ \\
 $\underline{t=t_0}$ & $\partial_s - X_{\lambda_0}$ \\
 $\underline{\partial W_{\mathrm{Ch}}}$ & $\partial_t -R_{\eta_{\mathrm{Ch}}}$ \\ $\underline{\partial C}$ & $X_{\partial C}$ \\
 \hline
\end{tabular}
\end{center}
\vspace{2mm}

Here, $X_{\partial C}$ is simply a placeholder for the backward oriented characteristic foliation of $\underline{\partial C}$. For example, we could write $X_{\partial W_{\mathrm{Ch}}} = \partial_t -R_{\eta_{\mathrm{Ch}}}$. The explicit description of $X_{\partial C}$ is given by the table in \ref{subsection:pre_chimney} that gives the characteristic foliation of a pre-chimney fold. Speaking loosely, most flowlines of $X_{\partial C}$ swirl around $\underline{\partial C}$ and up in the $t$-direction toward $\underline{t=t_0}$. 

The following proposition is the key feature of a chimney fold, namely, that such a fold traps the entire chimney region.

\begin{proposition}\label{prop:PL_chimney_trap}
Suppose that $x \in \{z=0\} \times \{s=s_0\} \times (0,t_0) \times \text{int}(C)$. Then the flowline through $x$ of the characteristic foliation of $C\Pi^{PL}$ is trapped in backward time. 
\end{proposition}

To prove Proposition \ref{prop:PL_chimney_trap}, we begin with a lemma that elucidates how a chimney fold exhibits the behavior of both an ordinary box fold and a pre-chimney fold in different regions. 

\begin{lemma}\label{lemma:chimney_z_0_trap}
If a flowline reaches $\underline{z=z_0}$ in a chimney fold, it is trapped in backward time. 
\end{lemma}

\begin{proof}
The characteristic foliation of $\underline{z=z_0}$ is directed by $-\partial_s$, so upon reaching $\underline{z=z_0}$ the flowline will travel to $\underline{s=0}$. There are two cases to consider: the flowline reaches either $\underline{s=0} \cap C$ or $\underline{s=0} \cap (W_{\mathrm{Ch}} \setminus C)$.

\vspace{2mm}
\noindent\textit{Case 1: the flowline reaches $\underline{s=0} \cap C$.}
\vspace{2mm}

Along $\underline{s=0}$ the foliation is directed by $\partial_t - \partial_z$. Over $C$, the length of the Reeb chords is $t_0$. Since $z_0 = e^{s_0}t_0 > t_0$, the flowline reaches $\underline{t=t_0}$. Because a chimney fold has the structure of a pre-chimney fold near $\underline{t=t_0}$, and because the flowline came from $\underline{z=z_0}$, we may apply Lemma \ref{lemma:pre_chimney_t_0_trap} to conclude that the flowline is ultimately trapped. 

\vspace{2mm}
\noindent \textit{Case 2: the flowline reaches $\underline{s=0} \cap (W_{\mathrm{Ch}} \setminus C)$.}
\vspace{2mm}

This case follows from Lemma \ref{lemma:t_0_trap}, with one minor subtlety. Over $W_{\mathrm{Ch}} \setminus C$, the length of the Reeb chords is $t_- < z_0$. Thus, the flowline follows $\partial_t - \partial_z$ to $\underline{t=t_-}$. Here the foliation is directed by $\partial_s - X_{\lambda_{\mathrm{Ch}}}$. Because $C$ is defined by the backward flow of some arc in $\partial W_{\mathrm{Ch}}$, and because the flowline is currently in $W_{\mathrm{Ch}} \setminus C$, it follows $\partial_s - X_{\lambda_{\mathrm{Ch}}}$ to $\underline{s=s_0}$ (as opposed to potentially reaching $\underline{\partial C}$). The $W_{\mathrm{Ch}}$-norm has decreased, and the flowline is still in $W_{\mathrm{Ch}} \setminus C$. 

Along $\underline{s=s_0}$ the flowline follows $-\partial_t + e^{s_0}\, \partial_z$. Because the flowline originally reached $\underline{t=t_-}$ from $\underline{z=z_0}$, the flowline then reaches $\underline{z=z_0}$. We return to Case 2. Note that the $t$-coordinate has increased, and the $W_{\mathrm{Ch}}$-norm has decreased. The flowline cycles through Case 2 indefinitely and never exits the fold. 

\end{proof}

\begin{proof}[Proof of Proposition \ref{prop:PL_chimney_trap}.]

Let $x\in  \{z=0\} \times \{s=s_0\} \times (0,t_0) \times \text{int}(C)$ denote the entry point of the flowline. The foliation is directed by $-\partial_t + e^{s_0}\, \partial_z$ along $\underline{s=s_0}$. Since $t(x) < t_0$ and $z_0 = e^{s_0}t_0$, the flowline reaches $\underline{t=0}$ with $z$-coordinate $e^{s_0}t(x)$. Here it follows $-\partial_s + X_{\lambda_{\mathrm{Ch}}}$. Since the flowline began in $C \subset N^{s_0}(\partial W_{\mathrm{Ch}})$, it then reaches $\underline{\partial W_0}$ before $\underline{s=0}$.

Here it follows $\partial_t - R_{\eta_{\mathrm{Ch}}}$. By Assumption \ref{ass:chimney} --- in particular, the fact that $\gamma_C \cap h_{\partial W_{\mathrm{Ch}}}(\gamma_C) = \emptyset$ --- the flowline then reaches $\underline{t=t_-}$ (as opposed to $\underline{\partial C}$). 

Along $\underline{t=t_-}$ the flowline follows $\partial_s - X_{\lambda_{\mathrm{Ch}}}$. Since the flowline is currently in $W_{\mathrm{Ch}} \setminus C$, it reaches $\underline{s=s_0}$ (as opposed to $\underline{\partial C}$), as in the proof of Lemma \ref{lemma:chimney_z_0_trap}. Along $\underline{s=s_0}$, the flowline follows $-\partial_t + e^{s_0}\, \partial_z$. If the flowline reaches $\underline{z=z_0}$, it is ultimately trapped by Lemma \ref{lemma:chimney_z_0_trap}. Otherwise, it reaches $\underline{t=0}$ with $z$-coordinate $e^{s_0}t(x) + e^{s_0}t_-$.

At this point, we essentially return to the beginning of the proof: the flowline has reached $\underline{t=0}$, but now with an increased $z$-coordinate of $e^{s_0}t(x) + e^{s_0}t_-$. After $j$ cycles through this process, the flowline will either reach $\underline{z=z_0}$, or it will reach $\underline{t=0}$ with $z$-coordinate 
\[
e^{s_0}t(x) + je^{s_0}t_-.
\]
Furthermore, with each cycle through this process, the flowline travels along $\partial W_{\mathrm{Ch}}$ via $h_{\partial W_{\mathrm{Ch}}}$. Recall Assumption \ref{ass:chimney}: with $n$ the smallest integer satisfying $nt_- > t_0$, we have $\gamma_C \cap h_{\partial W_{\mathrm{Ch}}}^j(\gamma_C) = \emptyset$ for $1\leq j \leq n$. This assumption ensures that the above process terminates only by reaching $\underline{z=z_0}$, rather than eventually reaching $\underline{\partial C}$. Indeed, after $n$ cycles through this process, the $z$-coordinate of the flowline would otherwise be 
\[
e^{s_0}t(x) + ne^{s_0}t_- > ne^{s_0}t_- > e^{s_0}t_0 = z_0.
\]
See Figure \ref{fig:chimney2} for a visualization of this argument. 

\end{proof}

\begin{figure}[ht]
	\begin{overpic}[scale=0.43]{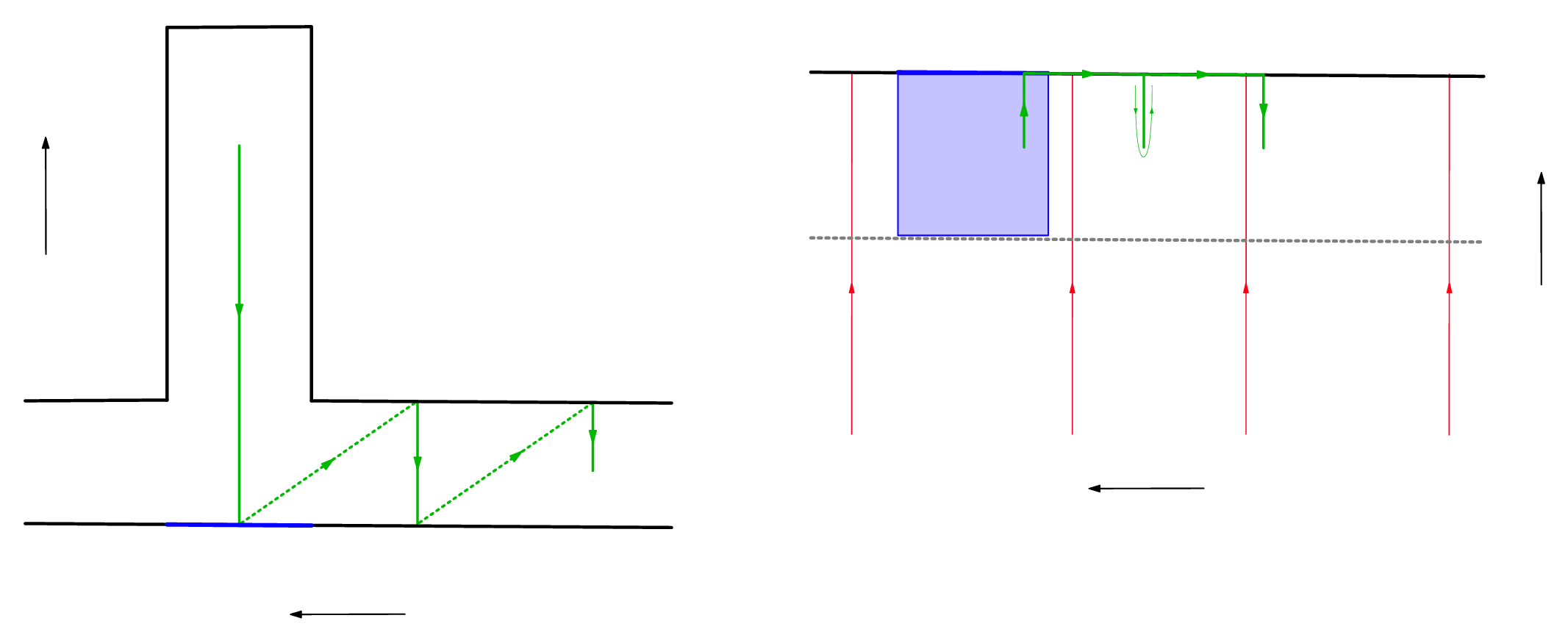}
	\put(2.5,33){\small $t$}
    \put(14,1.5){\small $R_{\eta_{\mathrm{Ch}}}$}
    \put(64,9.75){\small $R_{\eta_{\mathrm{Ch}}}$}
    \put(96,31.5){\small $X_{\lambda_{\mathrm{Ch}}}$}

    \put(61.5,30.5){\small $C$}
    \put(60.5,38){\small \textcolor{blue}{$\gamma_C$}}
    
	\end{overpic}
	\caption{A flowline entering $C\Pi^{PL}$ in the chimney that eventually gets trapped according to Proposition~\ref{prop:PL_chimney_trap}. It initially travels down to $\underline{t=0}$, then cycles through a process involving $h_{\partial W_{\mathrm{Ch}}}$ a number of times before reaching $\underline{z=z_0}$.}
	\label{fig:chimney2}
\end{figure}

\begin{remark}[Holonomy discussion]
Before moving on, we discuss what goes wrong without Assumption \ref{ass:chimney} in order to to clarify this aspect of the behavior of a chimney fold. This discussion is not logically necessary for the rest of the paper, in particular Section \ref{section:proof_main_prop}.

Recall that in an ordinary (piecewise linear) box fold, any flowline that reaches $\underline{\partial W_0}$ is ultimately trapped. Indeed, the foliation along this side is directed by $\partial_t - R_{\eta_0}$, so the flowline reaches $\underline{t=t_0}$ and is trapped by Lemma \ref{lemma:t_0_trap}. In contrast, it is not the case that any flowline reaching $\underline{\partial W_{\mathrm{Ch}}}$ in a chimney fold is trapped. Any flowline that reaches $\underline{\partial W_{\mathrm{Ch}}}$ via $C$ is trapped --- this is essentially the proof of Proposition \ref{prop:PL_chimney_trap} --- but it possible to reach $\underline{\partial W_{\mathrm{Ch}}}$ outside of $C$ and eventually exit the fold. 

For example, let 
\[
h^{-1}_{\partial W_{\mathrm{Ch}}}(C) := [-s_0, 0]_r \times h^{-1}_{\partial W_{\mathrm{Ch}}}(\gamma_C)
\]
and consider a flowline that enters the fold in $H_{\mathrm{Ch}}$ via $(0,e^{-s_0}t_-) \times h^{-1}_{\partial W_{\mathrm{Ch}}}(C)$. Suppose that the initial $r$-coordinate of the flowline is $-\bar{s}$ for some $0 < \bar{s} < s_0$. As the flowline enters the fold, it travels across $\underline{s=s_0}$ via $-\partial_t + e^{s_0}\, \partial_z$, reaching $\underline{t=0}$ first. Here it follows $-\partial_s + X_{\lambda_{\mathrm{Ch}}}$. Since $\bar{s} < s_0$, the flowline reaches $\underline{\partial W_{\mathrm{Ch}}}$. Note that the new $s$-coordinate is $s_0 - \bar{s}$. By definition of $h^{-1}_{\partial W_{\mathrm{Ch}}}(C)$, the flowline follows $\partial_t - R_{\eta_{\mathrm{Ch}}}$ and eventually reaches $\underline{\partial C}$, as opposed to $\underline{t=t_-}$; see Figure \ref{fig:chimney4}.

The flowline then follows the characteristic foliation of $\underline{\partial C}$, and a portion of such flowlines reaches $\underline{t=t_-}$ rather than spiraling up toward $\underline{t=t_0}$. Note that traversing the characteristic foliation of $\underline{\partial C}$ does not change the $s$-coordinate. Thus, a flowline that reaches $\underline{t=t_-}$ via $\underline{\partial C}$ will do so with $s$-coordinate $s_0 - \bar{s}$. On $\underline{t=t_-}$ the flowline follows $\partial_s - X_{\lambda_{\mathrm{Ch}}}$ to $\underline{s=s_0}$, again due to the definition of $C$. The $r$-coordinate (i.e., the $X_{\lambda_{\mathrm{Ch}}}$-coordinate) of the flowline is now $-s_0 - \bar{s}$. Here the flowline follows $-\partial_t + e^{s_0}\, \partial_z$ to $\underline{t=0}$ where it then follows $-\partial_s + X_{\lambda_{\mathrm{Ch}}}$. Because the $r$-coordinate upon reaching $\underline{t=0}$ is $-s_0 - \bar{s}$, the flowline reaches $\underline{s=0}$ first, and it does so with $r$-coordinate $-\bar{s}$. In particular, for most values of $\bar{s}$, the $W_{\mathrm{Ch}}$-coordinate of the flowline is now in $C$. Along $\underline{s=0}$ the flowline follows $\partial_t - \partial_z$. Since the length of the Reeb direction over $C$ is $t_0$, which is larger than than the current $z$-coordinate of the flowline, the flowline will travel up the chimney some distance and exit the fold.     
\end{remark}

\begin{figure}[htb]
            \centering
           \def\svgwidth{\linewidth}
            \graphicspath{{Figures/}}
            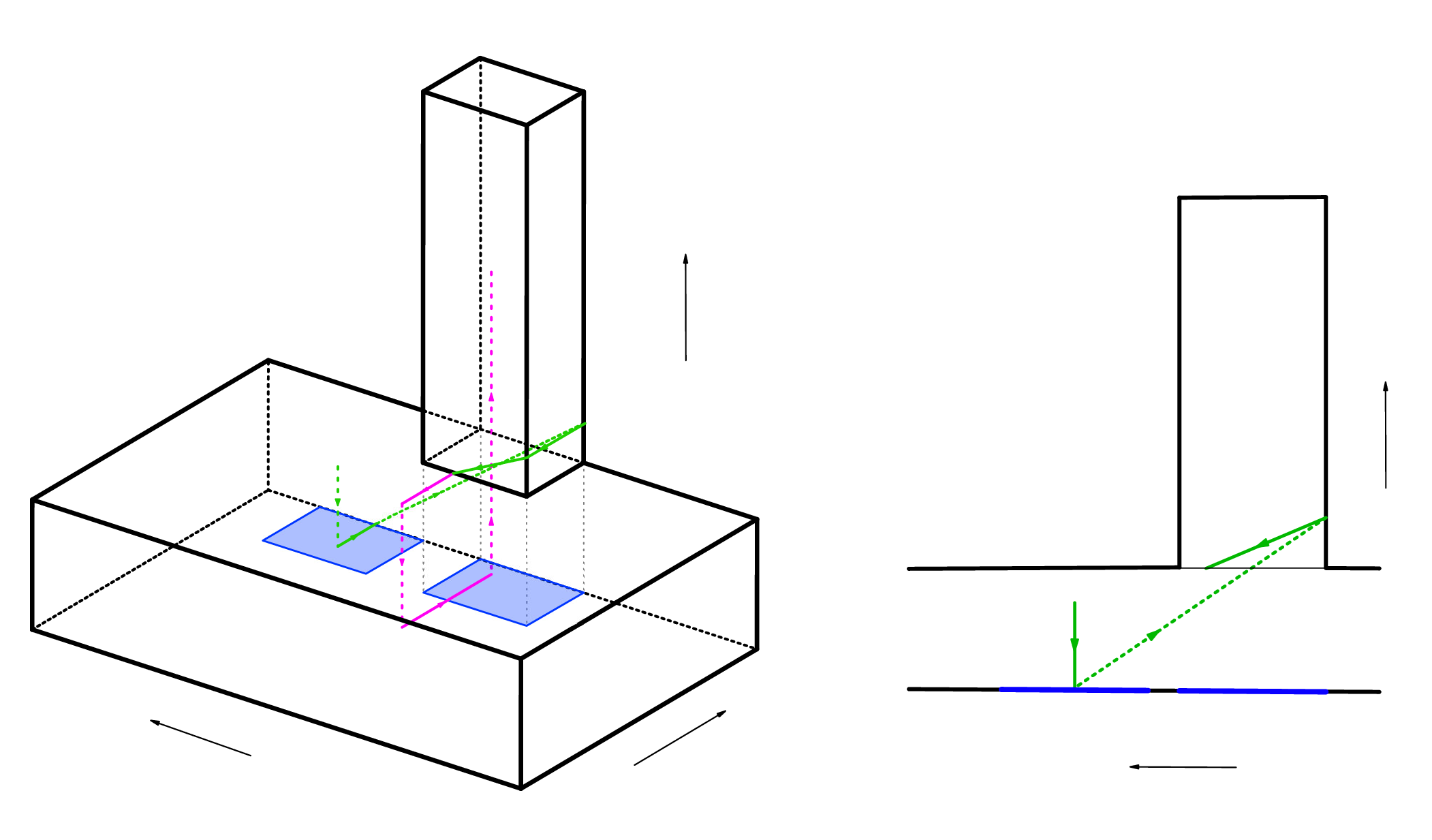
            \caption{A depiction of a flowline that enters the fold in the stove above $h^{-1}_{\partial W_{\mathrm{Ch}}}(C)$ and ultimately exits the fold after traveling up the chimney some distance. Two different phases of the flowline are color coded in green and pink for visual clarity, with the rightmost projection containing only the green phase. The flowline enters the fold at $x_1$ and exits the fold at $x_3$.}
            \label{fig:chimney4}
        \end{figure}

All this said, we summarize the behavior of a chimney fold. Points that enter in the chimney get funneled down and pass through the iterates $h_{\partial W_{\mathrm{Ch}}}^j(C)$ in the stove for $1\leq j \leq n$ before ultimately getting trapped. While this is happening, points entering through a collection of preimages $h_{\partial W_{\mathrm{Ch}}}^{-j}(C)$ in the stove pass through the various iterates and eventually travel up the chimney to exit the fold. As the $z$-coordinate of a flowline increases with each application of $h_{\partial W_{\mathrm{Ch}}}$, it follows that points entering $h^{-(j+1)}_{\partial W_{\mathrm{Ch}}}(C)$ with holonomy fill in a portion of the chimney above points entering $h^{-j}_{\partial W_{\mathrm{Ch}}}(C)$ with holonomy. Thus, the role of the $t$-thin stove is two-fold: it provides enough room to funnel points entering the chimney away to be trapped, and it also provides a repository of points that will ultimately fill in the trapping region in the chimney. Assumption \ref{ass:chimney} guarantees that these two features function independently of each other.

\subsubsection{High-dimensional piecewise linear chimney folds}\label{subsubsec:high_dim_chim}

The chimney folds we have discussed so far are based over 
\[
[0,s_0] \times H_{\mathrm{Ch}} \subset [0,s_0] \times [0,t_0] \times W_{\mathrm{Ch}}. 
\]
To extend this to higher dimensions, we next define a chimney fold based in a region of the form 
\[
[0,s_0] \times H_{\mathrm{Ch}} \times W_0 \subset [0,s_0] \times [0,t_0] \times W_{\mathrm{Ch}} \times W_0
\]
where $(W_0, \lambda_0)$ is a Weinstein domain of arbitrary dimension, different from $(W_{\mathrm{Ch}}, \lambda_{\mathrm{Ch}})$. This new $W_0$ direction corresponds to $W_0$ in the statement of Theorem~\ref{prop:mainprop}. The extension of a chimney fold to this setting is straightforward.

\begin{definition}
Fix $z_0, s_0, t_0, t_- > 0$ with $z_0 = e^{s_0}t_0$, and define $H_{\mathrm{Ch}} \subset [0,t_0] \times W_{\mathrm{Ch}}$ as before. A \textbf{(high-dimensional) piecewise linear chimney fold with parameters $z_0, s_0, t_0, t_-$}, still denoted $C\Pi^{PL}$, is the hypersurface 
\[
C\Pi^{PL} := \overline{\partial ([0,z_0] \times [0,s_0] \times H_{\mathrm{Ch}} \times W_0) \setminus \{z=0\}}.
\]
\end{definition}

The following proposition is the high-dimensional generalization of Proposition \ref{prop:PL_chimney_trap}. It says that nothing interesting happens to the trapping behavior of a chimney fold in this high-dimensional extension, at least as far as the chimney is concerned: every flowline entering the chimney (and anywhere in $W_0$) is trapped.

\begin{proposition}\label{prop:PL_chimney_trap_high}
Suppose that $x\in \{z=0\} \times \{s=s_0\} \times (0,t_0) \times \text{int}(C) \times \text{int}(W_0)$. Then the flowline through $x$ of the characteristic foliation of $C\Pi^{PL}$ is trapped in backward time. 
\end{proposition}

\begin{proof}
The proof is identical to the proof of Proposition \ref{prop:PL_chimney_trap}. Initially, the foliation along $\underline{s=s_0}$ is directed by $-\partial_t + e^{s_0}\, \partial_z$, so the flowline will reach $\underline{t=0}$. Here the foliation is $-\partial_s + X_{\lambda_{\mathrm{Ch}}} + X_{\lambda_0}$. The flowline will either reach $\underline{\partial W_{\mathrm{Ch}}}$ or $\underline{\partial W_0}$. In the former case, the same analysis in the proof of Proposition \ref{prop:PL_chimney_trap} shows that the flowline is trapped; the only additional behavior is some inconsequential back and forth movement in the $W_0$-direction before ultimately limiting towards $\skel(W_0, \lambda_0)$. 

In the latter case --- when the flowline reaches $\underline{\partial W_0}$ --- it then follows $\partial_t - R_{\eta_0}$, where $R_{\eta_0}$ is the Reeb vector field of $\eta_0 := \lambda_0\mid_{\partial W_0}$. Since the $W_{\mathrm{Ch}}$-coordinate of the flowline is in $C$, where Reeb chords have length $t_0$, the flowline follows $\partial_t - R_{\eta_0}$ and ultimately reaches $\underline{t=t_0}$. By Lemma \ref{lemma:t_0_trap} and Lemma \ref{lemma:pre_chimney_t_0_trap}, the flowline is trapped. 
\end{proof}

\subsection{Smooth chimney folds}\label{subsection:smooth_chimney}

As we did with box folds in Section \ref{section:box_folds_smooth}, we now pass from the piecewise linear chimney fold to a smooth, graphical version which is a legitimate Liouville homotopy. In particular, we prove the following proposition, which is the chimney fold analogue of Theorem \ref{theorem:new_smooth_box_fold}.

\begin{proposition}[Existence and behavior of smooth chimney folds]\label{prop:new_chimney_smooth}

Let $(W_{\mathrm{Ch}}, \lambda_{\mathrm{Ch}})$ and $(W_0, \lambda_0)$ be Weinstein domains of dimension $2$ and $2n-2\geq 2$, respectively. Fix $s_0, t_0 >0$ and let $V_0:= [0,s_0] \times [0,t_0] \times W_{\mathrm{Ch}} \times W_0$. For any smooth function $F:V_0 \to \R$, let $\lambda_{F} := dF + e^s\, (dt + \lambda_{\mathrm{Ch}} + \lambda_0)$ denote a Liouville form on $V_0$, let $X_{\lambda_F} = \partial_s + X_{F}$ be its Liouville vector field, and let 
\[
h_{F}:\{s=s_0\} \times [0,t_0] \times W_{\mathrm{Ch}} \times W_0 \dashrightarrow \{s=0\} \times [0,t_0] \times W_{\mathrm{Ch}} \times W_0
\]
be the partially-defined holonomy map given by backward flow of $X_{\lambda_F}$. 

Fix $0 < t_- \ll t_0$. Let $C\subset N^{s_0}(\partial W_{\mathrm{Ch}})$ be a flowbox of the Liouville vector field $X_{\lambda_{\mathrm{Ch}}}$ of length $s_0$ and width $\mu_0$, measured with respect to the Reeb flow of $\partial (W_{\mathrm{Ch}}, \lambda_{\mathrm{Ch}})$. Finally, fix $0< \epsilon \ll \min(1,s_0,t_-, \mu_0)$. There is a smooth function $F_{\epsilon}^{\mathrm{Ch}}: V_0 \to [0,\infty)$, compactly supported in the interior of $V_0$, with the following properties.

\begin{enumerate}
    \item (Weinstein compatibility)\label{chimney:weinstein}

    \noindent The Liouville vector field $X_{\lambda_{F_{\epsilon}^{\mathrm{Ch}}}}$ is Morse with $2(N_{\mathrm{Ch}} + 2)N_0$ critical points, where $N_{\mathrm{Ch}}$ and $N_0$ are the number of critical points of $(W_{\mathrm{Ch}}, \lambda_{\mathrm{Ch}})$ and $(W_0, \lambda_0)$, respectively. Moreover, the number of middle index critical points is $(N_{\mathrm{Ch}}^{\mathrm{crit}} + 1)N_0^{\mathrm{crit}}$, where $N_{\mathrm{Ch}}^{\mathrm{crit}}$ and $N_0^{\mathrm{crit}}$ are the number of middle index critical points of $(W_{\mathrm{Ch}}, \lambda_{\mathrm{Ch}})$ and $(W_0, \lambda_0)$, respectively.

    \item (Trapping properties)\label{chimney:trapping}

    \noindent Let $(t_{\mathrm{init}}, p_{\mathrm{init}}^{\mathrm{Ch}}, p_{\mathrm{init}}^0) \in \{s=s_0\} \times [0,t_0] \times W_{\mathrm{Ch}} \times W_0$ be the initial point of a flowline of $X_{\lambda_{F_{\epsilon}^{\mathrm{Ch}}}}$. Define a subset of $\{s= s_0\}$ by
    \[
    U_{\mathrm{trap}}^{\epsilon} := [\epsilon, t_0-\epsilon] \times C_{\epsilon} \times (W_0 \setminus N^{\epsilon}(\partial W_0))
    \]
    where $C_{\epsilon}\subset C$ is an $\epsilon$-retract of $C$. If $(t_{\mathrm{init}}, p_{\mathrm{init}}^{\mathrm{Ch}}, p_{\mathrm{init}}^0) \in U_{\mathrm{trap}}^{\epsilon}$, then the flowline converges to a critical point of $X_{\lambda_{F_{\epsilon}^{\mathrm{Ch}}}}$ in backward time. 

    \item (Holonomy properties) 

    \noindent Let $(t_{\mathrm{init}}, p_{\mathrm{init}}^{\mathrm{Ch}}, p_{\mathrm{init}}^0) \in \{s=s_0\} \times [0,t_0] \times W_{\mathrm{Ch}} \times (W_0 \setminus N^{s_0}(\partial W_0))$ be the initial point of a flowline of $X_{\lambda_{F_{\epsilon}^{\mathrm{Ch}}}}$. Assume $(t_{\mathrm{init}}, p_{\mathrm{init}}^{\mathrm{Ch}}, p_{\mathrm{init}}^0)$ is in the domain of $h_{F_{\epsilon}^{\mathrm{Ch}}}$. Let $(t_{\mathrm{fin}}, p_{\mathrm{fin}}^{\mathrm{Ch}}, p_{\mathrm{fin}}^0) := h_{F_{\epsilon}^{\mathrm{Ch}}}(t_{\mathrm{init}}, p_{\mathrm{init}}^{\mathrm{Ch}}, p_{\mathrm{init}}^0)\in \{s=0\}$ denote the exit point of the flowline in backward time.  

    \begin{enumerate}
        \item The estimate $\norm{p^0_{\mathrm{fin}}}_{W_0}\leq e^{s_0}\norm{p^0_{\mathrm{init}}}_{W_0}$ holds.\label{chimney:weinstein-bound}

        \item If $(t_{\mathrm{init}}, p_{\mathrm{init}}^{\mathrm{Ch}}) \in [0, \epsilon] \times C$, then $t_{\mathrm{fin}} \in [0, t_-]$.\label{chimney:bottom}

        \item If $t_{\mathrm{init}} \in [0, t_-]$ and $t_{\mathrm{fin}}> t_-$, then $p_{\mathrm{init}}^{\mathrm{Ch}} \notin C_{\epsilon}$ and $p_{\mathrm{fin}}^{\mathrm{Ch}} \in C$.\label{chimney:stove}

    \end{enumerate}    
\end{enumerate}
    
\end{proposition}

\begin{remark}
Properties~\ref{chimney:bottom} and~\ref{chimney:stove} are important enough to merit rephrasings in informal language for the sake of clarity. First, the assumption $p_{\mathrm{init}}^0 \notin N^{s_0}(\partial W_0)$ means that we are only interested in the holonomy of flowlines that enter sufficiently far from $\partial W_0$.  Then~\ref{chimney:bottom} states that such points near the bottom of the chimney region, if they are not trapped, exit in the stove. Likewise,~\ref{chimney:stove} says that any point which enters in the stove and exits above the stove must have entered outside of the chimney region.    
\end{remark}

To define the smooth chimney fold and prove Proposition \ref{prop:new_chimney_smooth}, we will pass directly through the definition of the smooth box fold in Section \ref{section:box_folds_smooth} and will appeal to Theorem \ref{theorem:new_smooth_box_fold}. With all of the notation as established in the statement of Proposition \ref{prop:new_chimney_smooth}, let 
\[
M_{\mathrm{aux}} := [0,t_-] \times W_{\mathrm{Ch}} \times W_0.
\]
In words, $M_{\mathrm{aux}}$ is an auxiliary manifold with the same shape as the stove region of the base of a chimney fold. Next, we define a region $H_{\mathrm{Ch}}^{\epsilon} \subset [0,t_0] \times W_{\mathrm{Ch}}$ as follows. Let $F_{\epsilon}^C:C \to [0,t_{0} - t_-]$ be a smooth box fold with height $t_0 - t_-$ and smoothing parameter $\epsilon$ as defined in \ref{subsec:smooth2d}, so that $F_{\epsilon}^C \equiv t_0 - t_-$ on $C_{\epsilon}$, and extend $F_{\epsilon}^C$ to a function on $W_{\mathrm{Ch}}$ by the $0$-function outside of $C$. Define $H_{\mathrm{Ch}}^{\epsilon} := \{0\leq t \leq t_- + F_{\epsilon}^C\}$. See the left side of Figure \ref{fig:chimney_new}. In words, $H_{\mathrm{Ch}}^{\epsilon}$ is a smooth approximation of the low-dimensional piecewise linear chimney region $H_{\mathrm{Ch}}$ obtained by installing a box fold to the top side of the stove, based in $C$, and with the appropriate height.

\begin{figure}[ht]
	\begin{overpic}[scale=0.4]{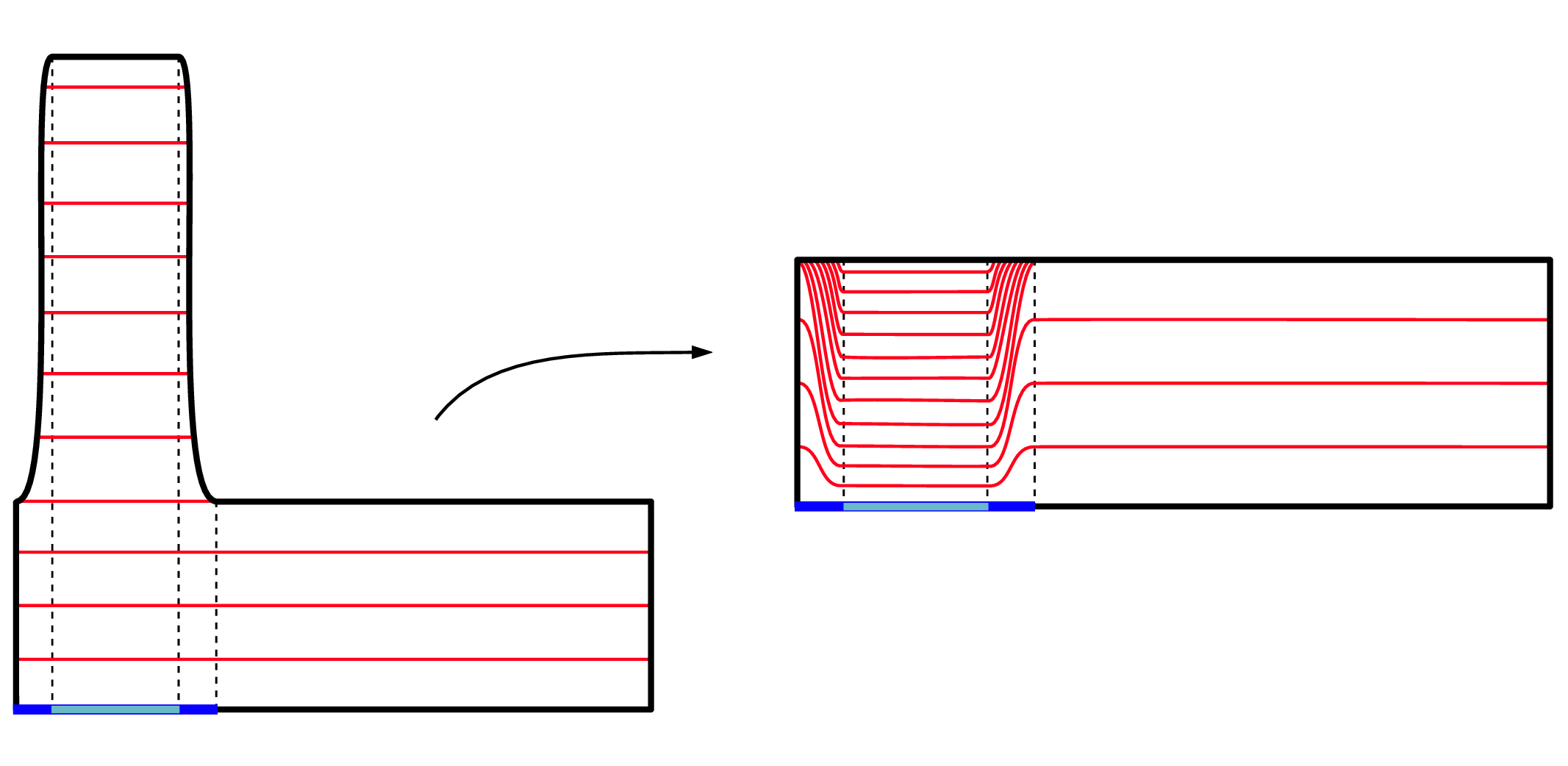}
	 \put(-1.5,16){\small $t_-$}
     \put(-1,3){\small $0$}
     \put(0,45){\small $t_0$}
     \put(16,0.5){\small $H_{\mathrm{Ch}}^{\epsilon}\times W_0$}   
        
     \put(100,32){\small $t_-$}
     \put(100,15.5){\small $0$}
     \put(74,13.5){\small $M_{\mathrm{aux}}$}

     \put(35,27){\small $\vartheta$}

	\end{overpic}
	\caption{In both figures, $C$ is indicated in dark blue and $C_{\epsilon}$ in light blue.}
	\label{fig:chimney_new}
\end{figure}

Let $\vartheta: H_{\mathrm{Ch}}^{\epsilon}\times W_0 \to M_{\mathrm{aux}}$ be a diffeomorphism of the form 
\[
\vartheta(t, p^{\mathrm{Ch}}, p^0) = (\theta(t, p^{\mathrm{Ch}}), p^{\mathrm{Ch}}, p^0)
\]
where $\theta: H_{\mathrm{Ch}}^{\epsilon} \to [0,t_-]$ satisfies the following properties. 
\begin{enumerate}
    \item For $p^{\mathrm{Ch}}\in C_{\epsilon}$, we have $\theta(t, p^{\mathrm{Ch}}) = \frac{t_-}{t_0}t$.

    \item For $p^{\mathrm{Ch}}\notin C$, we have $\theta(t, p^{\mathrm{Ch}}) = t$.
\end{enumerate}
See Figure \ref{fig:chimney_new}. Let $\alpha_{\mathrm{aux}}:= (\vartheta^{-1})^{\ast}(dt + \lambda_{\mathrm{Ch}} + \lambda_0)$ be the induced contact form on $M_{\mathrm{aux}}$. Note that, outside of $C$, $\alpha_{\mathrm{aux}} = dt + \lambda_{\mathrm{Ch}} + \lambda_0$, which is the usual contact form when viewing $M_{\mathrm{aux}}$ as a contact handlebody over $W_{\mathrm{Ch}} \times W_0$ with Reeb chords of length $t_-$. In $C_{\epsilon}$, however, $\alpha_{\mathrm{aux}} = \frac{t_0}{t_-}\, dt + \lambda_{\mathrm{Ch}} + \lambda_0$ and the Reeb chords have length $t_0$ as in a chimney fold. Over the region $C\setminus C_{\epsilon}$ there is some induced rotation of the Reeb vector field around $C$, but the precise behavior is inconsequential.

Choose $\epsilon'>0$ such that $\frac{t_0}{t_-}\epsilon' < \epsilon$, let $z_0 = e^{s_0}t_0$, and let $F_{\epsilon'}:[0,s_0]\times M_{\mathrm{aux}} \to [0,z_0]$ be the smooth box fold function as defined in Section \ref{section:box_folds_smooth} with smoothing parameter $\epsilon'$ and height $z_0$. Moreover, assume that Assumption \ref{ass:chimney} is satisfied. Finally, define $F^{\mathrm{Ch}}_{\epsilon}: [0,s_0]\times H_{\mathrm{Ch}}^{\epsilon}\times W_0 \to [0,z_0]$ by 
\[
F^{\mathrm{Ch}}_{\epsilon}(s, t, p^{\mathrm{Ch}}, p^0) := F_{\epsilon'}(s, \vartheta(t, p^{\mathrm{Ch}}, p^0)).
\]
This is the function that we will call a \textit{smooth chimney fold} and for which we will prove Proposition \ref{prop:new_chimney_smooth}. Given $F^{\mathrm{Ch}}_{\epsilon}$ as constructed above, we refer to the \textit{stove} or \textit{stove region} as the set 
\[
H_{\mathrm{Ch}}^{\epsilon} \cap \{t \leq t_- + \delta_{\mathrm{con}}e^{s_0}\}
\]
where $\delta_{\mathrm{con}} \ll \epsilon$ is the parameter used in the definition of the smooth box fold.  The smoothing of the chimney fold requires us to include in the stove some points which have $t$-value greater than $t_-$.  Namely, Remark~\ref{remark:2D-cp-height} tells us that the $t$-height of the index $1$ point of the box fold used to define the chimney portion of $H_{\mathrm{Ch}}^{\epsilon}$ is below the height $t_- + \delta_{\mathrm{con}}e^{s_0}$, and we define the stove region so that it contains this point.

\begin{remark}
The point of the above construction, and the proof of Proposition \ref{prop:new_chimney_smooth}, is that a chimney fold can be viewed as a ``box fold'' wherein the contact form on the underlying contact handlebody has been perturbed. As a result, the Reeb flow inside the region over $C$ has been modified --- in particular, the vertical Reeb flow has been massively slowed down over $C_{\epsilon}$ --- as has the characteristic foliation on the top side of $M_{\mathrm{aux}}$ (which dictates the nature of the eventual critical points). Accounting for these modifications, the content of Section \ref{section:box_folds_smooth} ``pulls back'' via the contactomorphism $\vartheta^{-1}$ to the chimney fold setting. 
\end{remark}

\begin{proof}[Proof of Proposition \ref{prop:new_chimney_smooth}.]

The Morse condition~\ref{chimney:weinstein} is immediate from~\ref{property:high_Weinstein_compat} of Theorem \ref{theorem:new_smooth_box_fold}. To count the number of critical points as stated, we count the critical points of the characteristic foliation of the top side $\{t=t_-\} \cong W_{\mathrm{Ch}} \times W_0$ of $(M_{\mathrm{aux}}, \alpha_{\mathrm{aux}})$. The restriction of the contact form to $\{t=t_-\}$ is $dF_{\epsilon}^C + \lambda_{\mathrm{Ch}} + \lambda_0$ where $F_{\epsilon}^C$ is the $2$-dimensional box folding function on $C$. The Liouville vector field of $dF_{\epsilon}^C + \lambda_{\mathrm{Ch}}$ on $W_{\mathrm{Ch}}$ has $N_{\mathrm{Ch}}+2$ critical points, of which $N_{\mathrm{Ch}}^{\mathrm{crit}}+1$ have middle index, and so $dF_{\epsilon}^C + \lambda_{\mathrm{Ch}} + \lambda_0$ has $(N_{\mathrm{Ch}}+2)N_0$ critical points, of which $(N_{\mathrm{Ch}}^{\mathrm{crit}}+1)N_0^{\mathrm{crit}}$ have middle index. Then by~\ref{property:2d_Weinstein_compat} of Proposition \ref{prop:2D_smooth_box_fold}, the Liouville vector field $X_{\lambda_{F_{\epsilon}^{\mathrm{Ch}}}}$ of the chimney fold function has $2(N_{\mathrm{Ch}}+2)N_0$ critical points, of which $(N_{\mathrm{Ch}}^{\mathrm{crit}}+1)N_0^{\mathrm{crit}}$ have middle index. 

Next we consider the trapping region described in~\ref{chimney:trapping}. Recall that in a piecewise linear chimney fold, the entire chimney region is trapped. By a piecewise smooth homeomorphism from the piecewise linear chimney fold base onto $M_{\mathrm{aux}}$ defined analogously to $\vartheta$, we may induce a piecewise linear contact form $\alpha_{\mathrm{aux}}^{PL}$ on $M_{\mathrm{aux}}$ for which a piecewise linear box fold based over $(M_{\mathrm{aux}}, \alpha_{\mathrm{aux}}^{PL})$ captures this behavior. By~\ref{chimney:trapping} of Theorem \ref{theorem:new_smooth_box_fold}, we may approximate the trapping behavior of this piecewise linear fold with $\epsilon'$ precision with a smooth box fold based over $(M_{\mathrm{aux}}, \alpha_{\mathrm{aux}})$. Then under the contactomorphism $\vartheta$ and the assumption that $\frac{t_0}{t_-}\epsilon' < \epsilon$, this gives a smooth approximation with $\epsilon$ precision to the trapping region of the piecewise linear chimney fold, which is the region described by~\ref{chimney:trapping}.

Finally we consider the holonomy properties. 

\begin{enumerate}
    \item[(\ref{chimney:weinstein-bound})] The $W_0$-norm estimate in~\ref{chimney:weinstein-bound} follows from~\ref{thm-part:new_smooth_holonomy} of Theorem \ref{theorem:new_smooth_box_fold}, as the modified contact form in $M_{\mathrm{aux}}$ does not affect its proof.

    Note that, since $p^0_{\mathrm{init}}\in W_0 \setminus N^{s_0}(\partial W_0)$ by assumption, this implies that the flowline experiences negligible holonomy in the $\partial W_0$-direction, which only occurs in an $\epsilon$-neighborhood of $\partial W_0$. This comment is relevant for~\ref{chimney:bottom} and~\ref{chimney:stove}.

    \item[(\ref{chimney:bottom})] The next property~\ref{chimney:bottom} is a consequence of~\ref{property:high_hol3c} of Theorem \ref{theorem:new_smooth_box_fold} together with Assumption \ref{ass:chimney}. Indeed, suppose for the sake of contradiction that $t_{\mathrm{fin}} > t_-$. Then necessarily by the support of the chimney fold we have $p^{\mathrm{Ch}}_{\mathrm{fin}}\in C$. By~\ref{property:high_hol3c} of Theorem~\ref{theorem:new_smooth_box_fold}, the flowline experiences a time-$(-t_*)$ flow in the $R_{\eta_{\mathrm{Ch}}}$-direction, where $t_{*}$ is sufficiently large enough to satisfy Assumption~\ref{ass:chimney}. Such a flow displaces $C$. This is a contradiction, and hence $t_{\mathrm{fin}}\in [0,t_-]$. 

    \item[(\ref{chimney:stove})] The final property is similar. The observation that $p^{\mathrm{Ch}}_{\mathrm{fin}}\in C$ is immediate from the support of the chimney fold. Next, suppose for the sake of contradiction that $p_{\mathrm{init}}^{\mathrm{Ch}}\in C_{\epsilon}$. By~\ref{chimney:trapping}, we have $t_{\mathrm{init}} \in [0,\epsilon]$. But then by~\ref{chimney:bottom}, $p_{\mathrm{init}}^{\mathrm{Ch}} \notin C$. Thus, we must have $p_{\mathrm{init}}^{\mathrm{Ch}}\notin C_{\epsilon}$.
    
\end{enumerate}
\end{proof}

\section{The blocking apparatus and proof of Theorem~\ref{prop:mainprop}}\label{section:proof_main_prop}
In this section we prove Theorem~\ref{prop:mainprop}, the local operation that drives the strategy of Section \ref{section:mitsumatsu}. The majority of the section is spent in \ref{subsec:low_dim_BA_new} where we construct and analyze the blocking apparatus in low dimensions. In \ref{subsec:real_proof}, we extend this work to higher dimensions and prove Theorem \ref{prop:mainprop}.

\subsection{The low-dimensional blocking apparatus}\label{subsec:low_dim_BA_new}

We begin by proving a low-dimensional version of Theorem~\ref{prop:mainprop} without the $W_0$-factor. 

\begin{proposition}[Low-dimensional blocking apparatus]\label{prop:reduced_mainprop}

Consider the Weinstein cobordism $(U=[0,s_0]\times[0,t_0]\times r_0\,\mathbb{D}^2,\lambda_U = e^s\,(dt+\lambda_{\mathrm{stab}}))$.  Fix $0<\epsilon\ll \min(1,s_0,\frac{1}{2}(1-e^{-s_0})t_0)$ sufficiently small. If $r_0>0$ is sufficiently large, a blocking apparatus can be installed in $U$, i.e., a Liouville homotopy $\lambda_U \rightsquigarrow \lambda_{\mathrm{BA}}$ can be applied to $U$, such that the following properties hold. 
\begin{enumerate}
    \item (Weinstein compatibility)\label{reduced_mainprop:compat}

    \noindent The Liouville vector field $X_{\lambda_{\mathrm{BA}}}$ is Morse with 
    $8$ critical points. Of these $8$ critical points, there is a single critical point of middle index. 

    \item (Trapping properties)\label{reduced_mainprop:trapping}

\noindent There is a neighborhood $U_{\mathrm{trap}}$ of $[\epsilon,t_0-\epsilon]\times \{(0,0)\}\subset[0,t_0]\times r_0\,\mathbb{D}^2$ such that any flowline passing through $\{s=s_0\}\times U_{\mathrm{trap}}\subseteq \partial_+U$ converges to a critical point of $X_{\lambda_{\mathrm{BA}}}$ in backward time.

    \item (Holonomy properties)\label{reduced_mainprop:holonomy}

\noindent Let $x\in \partial_+U$ be in the domain of the partially-defined holonomy map $h\colon\partial_+U\dashrightarrow\partial_-U$. Then the following properties hold. 
\begin{enumerate}
    \item For some constant $0<K<1$, we have  $\|\pi_{r_0\,\mathbb{D}^2}(h(x))\|_{\mathrm{stab}}\leq Ke^{\frac{s_0}{2}}\,\|\pi_{r_0\,\mathbb{D}^2}(x)\|_{\mathrm{stab}}$, where $\|\cdot\|_{\mathrm{stab}}$ is the usual Euclidean norm on $r_0\D^2$.\label{prop-part:reduced_mainprop_holonomy}
    \item If $t(x) \in [0,\epsilon] \cup [t_0 - \epsilon,t_0]$, then $t(h(x)) \in [0,\epsilon] \cup [t_0 - \epsilon, t_0]$.\label{prop-part:reduced_mainprop_t_holonomy}
    
\end{enumerate}
\end{enumerate}

\end{proposition}

\begin{remark}
The $\epsilon$ parameter in the statement of Proposition \ref{prop:reduced_mainprop} will be the smoothing parameter for all folds involved in the blocking apparatus. In particular, $\epsilon$ is the same $\epsilon$ that appears in the statements of Theorem \ref{theorem:new_smooth_box_fold} and Proposition \ref{prop:new_chimney_smooth}.    
\end{remark}

In \ref{subsubsec:low-dim_BA} we construct the low-dimensional blocking apparatus, and in \ref{subsubsec:low-dim_BA_proof} we use it to prove Proposition \ref{prop:reduced_mainprop}.

\subsubsection{Construction of the low-dimensional blocking apparatus}\label{subsubsec:low-dim_BA}

A blocking apparatus will consist of a chimney fold $C\Pi_1$ and an ordinary box hole $\rotreflPi_2$. The chimney fold will be installed so that the chimney region traps a large portion of the $t$-axis in $[0,t_0] \times r_0\D^2$, and the box hole will be installed behind the chimney fold (in the $s$-sense) so that its trapping region contains the stove of the chimney fold. Roughly, the point is that we use the chimney fold to achieve most of the desired trapping behavior, and we use the box hole to swallow all of the unwanted holonomy behavior induced by the stove of the chimney fold. 

For now we use the model
\[
\left([0,s_0]\times [0,t_0]\times \R^2, \, e^s\, (dt + \lambda_{\mathrm{stab}})\right)
\]
with an infinite stabilization direction. There are various parameters involved in the construction of a blocking apparatus, and we will choose $r_0$ at the end. We instantiate them in an intentional order below, and for each parameter we provide some informal explanation as to its role. Our convention in this section is as follows, in contrast to previous sections: parameters with a $0$ subscript (like $s_0, t_0$) correspond to the given model cobordism, parameters with a $1$ subscript correspond to the chimney fold $C\Pi_1$, and parameters with a $2$ subscript correspond to the box hole $\rotreflPi_2$. Furthermore, we will often use the radial coordinate $r = \sqrt{p^2 + q^2}$ on $\R^2$. 

Given $\epsilon > 0$ as in the statement of Proposition \ref{prop:reduced_mainprop}, we proceed as follows. We prompt the reader to consult Figure \ref{fig:BAadj} and Figure \ref{fig:BAadj2} while noting each parameter. 
\begin{enumerate}
    \item \underline{Choose $s_1,s_2 > 0$ such that $s_1 + s_2 < s_0$.}

    These are the symplectization lengths of $C\Pi_1$ and $\rotreflPi_2$, respectively. 

    \item \underline{Choose $t_2$ and $z_2$ satisfying $\frac{t_0}{2} < t_2 < t_0 - 2\epsilon$ and $z_2 \geq e^{s_2}t_2$.}

    These are the Reeb-thickness of the contact handlebody supporting the box hole $\rotreflPi_2$ and the $z$-depth of the hole, respectively. The lower bound on $t_2$ is not so important; it simply serves to define an $\epsilon$-independent bound. The upper bound ensures that the top of the box hole will not influence the holonomy induced by the top of the chimney fold. 

    \item \underline{Choose $\delta_1>\epsilon$ so that $3\delta_1 < (1-e^{-s_2})t_2$.}

    The parameter $\delta_1$ is a $t$-value that will position the tilted base of the region $H_1^{C_1}$ supporting the chimney fold $C\Pi_1$. It will also represent the Reeb-thickness of the stove of $H_1^{C_1}$. The required inequality ensures that the trapping region of the box hole $\rotreflPi_2$ is sufficiently $t$-thick to contain the entire stove of $H_1^{C_1}$.

   \item \underline{Choose $c_1, p_1 > 0$ such that $c_1p_1 < \delta_1 - \epsilon$.}
    
    The parameter $c_1$ will be the (negative) $p$-slope of the Weinstein base $W_1 \subset [0,t_0]\times \R^2_{p,q}$ that will define the chimney fold. The parameter $p_1$ will be the $p$-thickness of the Weinstein base $W_1$, ensuring that the support of the chimney fold is a subset of $[\epsilon,t_0]\times r_0\D^2$.

    \item \underline{Choose $0< \rho_1 < p_1$ sufficiently small to satisfy the requirements described below.}
    
    This will be the radius of the chimney. Let $W_1$ be the surface
    \[
    W_1 := \{t=\delta_1-c_1p\,|\, (p,q) \in [-p_1,p_1] \times [-q_1,0] \text{ or } r \leq \rho_1\} \subset [0,t_0]\times \R^2_{p,q},
    \]
    for some $q_1 > 0$ to be chosen later. Endow $W_1$ with the Liouville form induced by the restriction of $dt + \lambda_{\mathrm{stab}}$. We require $\rho_1$ to be small enough so that the following two properties hold. Note that both hold trivially as $\rho_1 \to 0$, and that neither depend on $q_1$. 
    \begin{enumerate}
        \item \underline{$\{r\leq \rho_1\} \cap W_1 \subset N^{s_1}(\partial W_1).$}
        
        This ensures that $C_1 := \{r \leq \rho_1\}\cap W_1$ is a viable chimney region. 
        
        \item \underline{$p\left(h_{\partial W_1}(p=-\rho_1, q=0)\right) > \rho_1$.}
        
        Here $h_{\partial W_1}$ is the boundary holonomy map of the stove as in Section \ref{section:chimney_folds}. This ensures part of Assumption ~\ref{ass:chimney} needed to define $C\Pi_1$. Towards this, we then let $\gamma_{C_1} \subset \partial W_1$ be an arc containing $\{r=\rho_1\} \cap \{q \geq 0\} \cap \partial W_1$ such that $\gamma_{C_1}\cap h_{\partial W_1}(\gamma_{C_1}) = \emptyset$ and $\{r\leq \rho_1\} \subset N^{s_1}(\gamma_{C_1})$. 
    \end{enumerate}

     \item \underline{Choose $q_1 > 2c_1 > 0$ large enough so that Assumption \ref{ass:chimney} holds.}
    
    In other words, we choose the $q$-length of $W_1$ large enough to ensure that $\gamma_{C_1}\cap h^j_{\partial W_1}(\gamma_{C_1}) = \emptyset$ for all $1\leq j \leq n$, where $n$ satisfies $n\delta_1 > t_0$. It is possible to give a precise quantitative estimate on $q_1$, but not necessary. We need $q_1 > 2c_1$ to ensure that $W_1$ is a domain (the critical point of the Liouville vector field of $W_1$ is at $(p=0, q=-2c_1)$), and we simply need $q_1$ to be large enough so that $h^n_{\partial W_1}(\gamma_{C_1})$ does not circle around $\partial W_1$ back to $\gamma_{C_1}$. It is clear that such a $q_1$ exists, provided we have an arbitrarily large stabilization direction.

    \item \underline{Pick $r_2 > 0$ such that $e^{s_2}\sqrt{p_1^2 + q_1^2} <  r_2$.}
    
    This will be the radius of the contact handlebody supporting $\rotreflPi_2$, which needs to be large enough so that the box-hole version of $U_{\mathrm{trap}, 1}$ in \ref{property:high_trap2} of Theorem \ref{theorem:new_smooth_box_fold} traps the entire stove of $H_1^{C_1}$.

    \item \underline{Pick $r_0 > r_2$.}

    This is just to provide enough room in the stabilization direction to support both folds in the blocking apparatus. 

\end{enumerate}

\begin{figure}[ht]
	\begin{overpic}[scale=0.43]{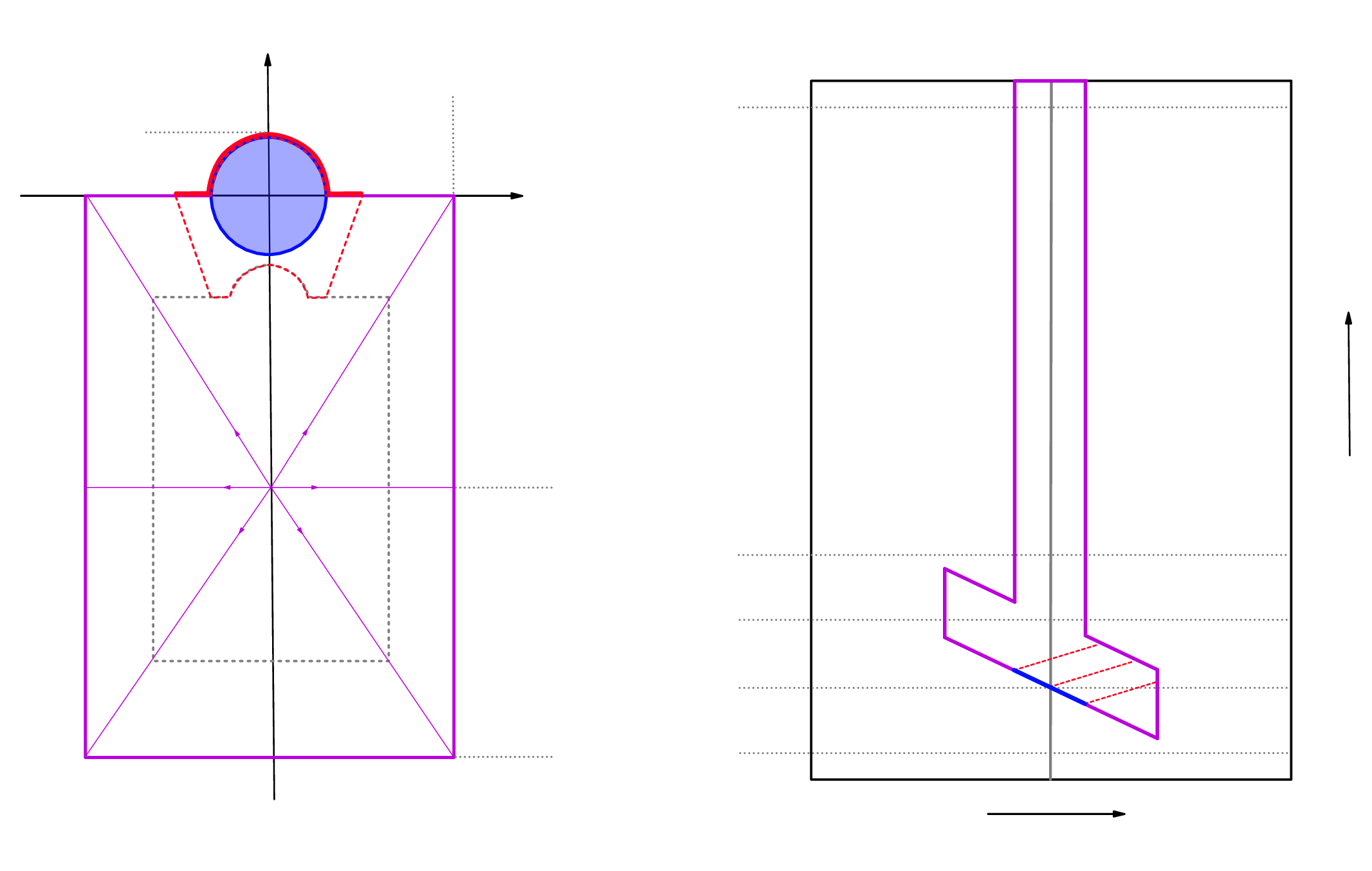}
	 \put(83,3.5){\small $p$}
     \put(98,41.5){\small $t$}
     \put(39,48.75){\small $p$}
     \put(19,61){\small $q$}

     \put(40.5,27.5){\small $-2c_1$}
     \put(40.5,7.75){\small $-q_1$}
     \put(8,53.5){\small $\rho_1$}

     \put(51.75,8.25){\small $\epsilon$}
     \put(51.25,12.75){\small $\delta_1$}
     \put(50.75,17.75){\small $2\delta_1$}
     \put(50.75,22.75){\small $3\delta_1$}

     \put(48.25,55.25){\small $t_0 - \epsilon$}

	\end{overpic}
	\caption{Two projections of $H_1^{C_1}$. On the left, $C_1$ is the blue disk, $\gamma_{C_1}$ is the thick red curve, and the region $N^{s_1}(\gamma_{C_1})$ is given by the dashed red region. On the right, $H_1^{C_1}$ is outlined in purple. The dashed red lines indicate the characteristic foliation of $\underline{\partial W_1}$.}
	\label{fig:BAadj}
\end{figure}

With these choices, we define the contact manifolds $H_1^{C_1}$ and $H_2$ supporting $C\Pi_1$ and $\rotreflPi_2$. First, let $H_1^{\mathrm{st}}$ be the stove of $H_1^{C_1}$ defined by flowing $W_1$ along $\partial_t$ for time $\delta_1$. Then let 
    \[
    H_1^{C_1} := H_1^{\mathrm{st}} \cup \left(\{\delta_1 - c_1p\leq t\leq t_0\}\cap \{r \leq \rho_1\}\right).
    \]
    See Figure \ref{fig:BAadj}. Next, let $W_2 := \{t=0\} \cap \{r\leq r_2\}$ and endow it with the Liouville form $\lambda_{\mathrm{stab}}$. Then let $H_2 := [0,t_2]\times W_2$. See Figure \ref{fig:BAadj2}. Finally: 

\begin{definition}
A \textbf{blocking apparatus (with smoothing parameter $\epsilon$)} consists of a chimney fold $C\Pi_1$ with smoothing parameter $\epsilon$ installed over $[s_0-s_1, s_0] \times H_1^{C_1}$ and a box hole with smoothing parameter $\epsilon$ installed over $[0,s_2]\times H_2$. 
\end{definition}

\begin{figure}[ht]
	\begin{overpic}[scale=0.43]{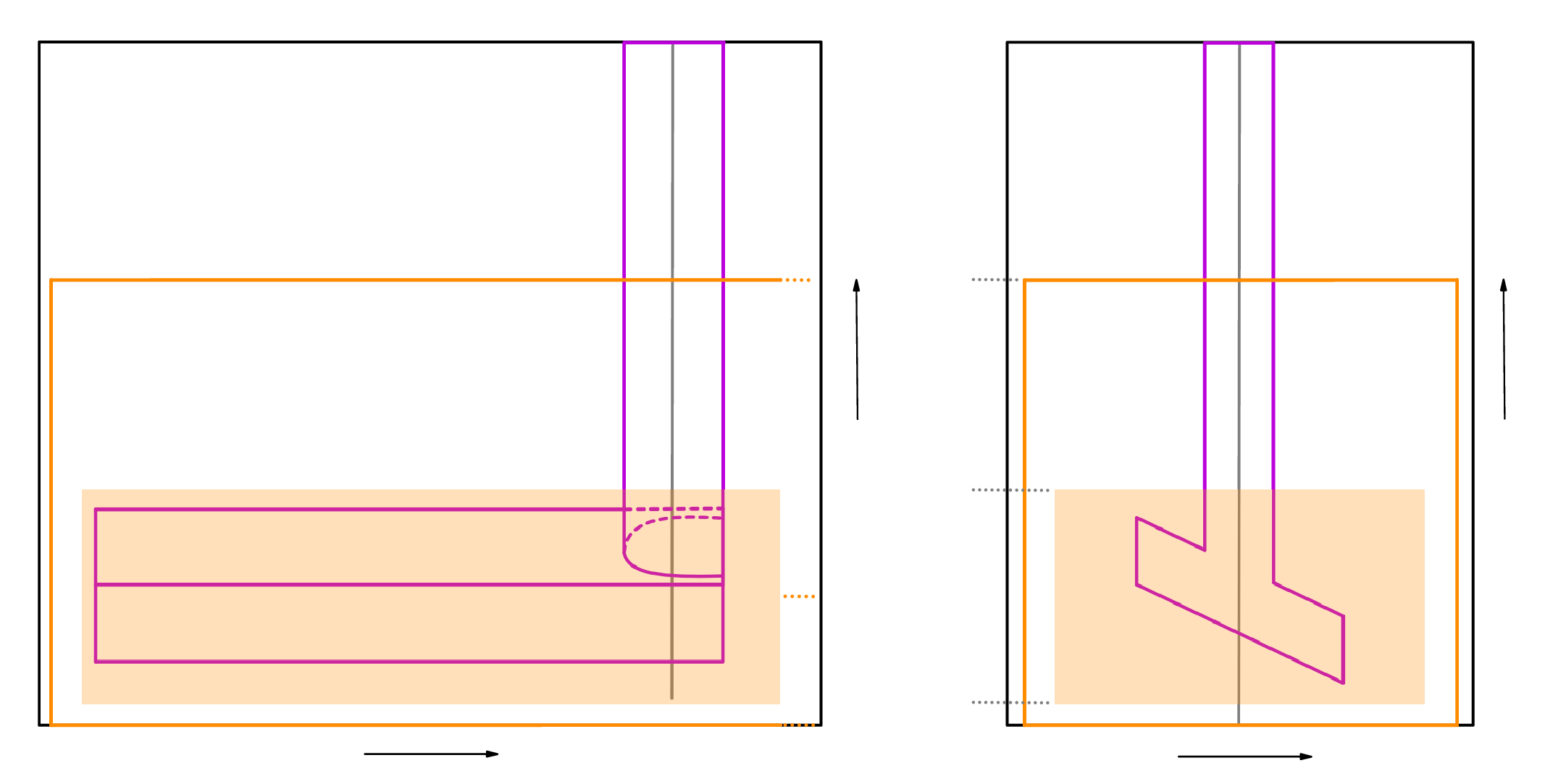}
	 \put(33,1.5){\small $q$}
  \put(86,1.5){\small $p$}
     \put(97,33.5){\small $t$}
     \put(55,33.5){\small $t$}

     \put(61.25,5){\small $\epsilon$}
     
     \put(61,18.5){\small $t^*$}

     \put(61,32.25){\small $t_2$}

	\end{overpic}
	\caption{The region $H_2$ in orange placed against $H_1^{C_1}$ in purple. The lower trapping region of $\rotreflPi_2$ is shaded in light orange. In particular, it contains the entire stove of $H_1^{C_1}$. The point of the design is that most of the complicated behavior induced by and emanating from the stove of the chimney fold is rendered obsolete by the box hole. The value $t^*$ is $t^* = (1-e^{-s_2})t_2 - \epsilon$.}
	\label{fig:BAadj2}
\end{figure}

\begin{remark}
Strictly speaking, the contact region $H_1^{C_1}$ is more complicated than the model region that supports a chimney fold as defined in Section \ref{section:chimney_folds}, for two reasons:
\begin{enumerate}
    \item The region $C_1$ is not a literal flowbox of the Liouville vector field of $W_1$. However, it is contained in a flowbox, and thus the naturally modified Proposition \ref{prop:new_chimney_smooth} still applies. 

    \item Reeb chords have variable length over $C_1$. This modification does not result in any significant change to the dynamics of the resulting chimney fold. The point is to ensure that the index $2$ critical point of the characteristic foliation of $\partial H_1^{C_1}$ at the top of the chimney coincides with the $t$-axis. It is possible to instead use a version of $\tilde{H}_1^{C_1}$ with a tilted top in accordance with the strict model of Section \ref{section:chimney_folds}. The blocking apparatus would require some extra fudging with more parameters, and possibly an extra ordinary box fold to trap some unwanted holonomy induced near $t=t_0$ as a result. We will not describe the details here, but we emphasize that the interplay of chimney folds and box folds is robust enough to allow for multiple design and parameter choices. 
\end{enumerate}

\end{remark}

\subsubsection{Proof of Proposition \ref{prop:reduced_mainprop}}\label{subsubsec:low-dim_BA_proof}

We begin with a lemma which is an explicit witness of the Weinstein norm estimate~\ref{thm-part:new_smooth_holonomy} of Theorem \ref{theorem:new_smooth_box_fold}.

\begin{lemma}\label{lemma:for_free}
Let $\Pi_0$ be a smooth box fold installed over the contact handlebody $(H_0 = [0,t_0] \times r_0 \D^2, \, dt + \frac{1}{2}r^2\, d\theta)$ with symplectization length $s_0$. Let $h:\{s_0\} \times H_0 \dashrightarrow \{0\} \times H_0$ be the partially-defined holonomy map given by backward passage through $\Pi_0$. Then 
\[
\|h(x)\|_{\mathrm{stab}} \leq e^{\frac{s_0}{2}}\|x\|_{\mathrm{stab}}.
\]
\end{lemma}

\begin{proof}
Observe that the Liouville vector field of the base $(r_0\D^2, \, \frac{1}{2}r^2\, d\theta)$ is $\frac{1}{2}r\, \partial_r$. An elementary calculation shows that the time-$s_0$ flow of $\frac{1}{2}r\, \partial_r$ induces an expansion by a factor of $e^{\frac{s_0}{2}}$. The desired inequality then follows from part~\ref{thm-part:new_smooth_holonomy} of Theorem~\ref{theorem:new_smooth_box_fold}.
\end{proof}

\begin{remark}[Installation for free]
The point of this lemma is that an ordinary box fold (or ordinary box hole) with flat base can be installed ``for free'' with respect to the radial holonomy. That is, when seeking to prove an estimate of the form $\|h(x)\|_{\mathrm{stab}}\leq K\,e^{\frac{s_0}{2}}\|x\|_{\mathrm{stab}}$ as in part~\ref{prop-part:reduced_mainprop_holonomy} of Proposition \ref{prop:reduced_mainprop}, one can install such a fold and remain compatible with the desired radial estimate. In contrast, the installation of the chimney fold $C\Pi_1$ by itself does not come for free, because, thanks to the tilted base, the holonomy through $C\Pi_1$ aggressively violates any estimate of the form $\|h(x)\|_{\mathrm{stab}} \leq K\,e^{\frac{s_0}{2}}\|x\|_{\mathrm{stab}}$. 
\end{remark}

\begin{proof}[Proof of Proposition \ref{prop:reduced_mainprop}.]

The Weinstein compatibility condition~\ref{reduced_mainprop:compat} together with the count of critical points follows from part~\ref{property:high_Weinstein_compat} of Theorem \ref{theorem:new_smooth_box_fold} and part~\ref{chimney:weinstein} of Proposition \ref{prop:new_chimney_smooth}. The box hole contributes $2$ critical points, since it is based over a Weinstein domain with a single critical point of index $0$. Neither of the $2$ critical points are of middle index. Likewise, the chimney fold contributes $2(1 + 2) = 6$ critical points, one of which is of middle index. 

\vspace{2mm}
    \textit{Trapping property~\ref{reduced_mainprop:trapping}}.
    \vspace{2mm}

Next we verify that there is a neighborhood $U_{\mathrm{trap}}$ of $[\epsilon,t_0-\epsilon]\times \{(0,0)\}$ such that every flowline entering $\{s=s_0\}\times U_{\mathrm{trap}}$ converges to a critical point in backward time. By part~\ref{chimney:trapping} of Proposition \ref{prop:new_chimney_smooth}, the trapping region of the chimney fold $C\Pi_1$ includes, for instance, $[2\delta_1, t_0 - \epsilon]\times \{r \leq \rho_1-\epsilon_1\}$. Here it is not necessary to give a precise lower $t$-bound. 

We now simply need to argue that all of the points in a neighborhood of $[\epsilon, 2\delta_1]\times\{(0,0)\}$ are trapped. Any such points outside of the support of $C\Pi_1$ are trapped by the box hole. By part~\ref{chimney:stove} of Proposition ~\ref{prop:new_chimney_smooth}, any such points with $W_1(x)\in C_1$ and $t(x)\leq 2\delta_1$ that are in the support of $C\Pi_1$ and not trapped by $C\Pi_1$ exit $C\Pi_1$ in the stove, and thus are trapped by $\rotreflPi_2$. 

These observations together show that, ultimately, the region $[\epsilon, t_0 - \epsilon] \times \{r\leq \rho_1 - \epsilon\}$ is trapped.

\vspace{2mm}
    \textit{Holonomy property~\ref{prop-part:reduced_mainprop_holonomy}.}
    \vspace{2mm}

    Next we verify the holonomy properties in~\ref{reduced_mainprop:holonomy}, beginning with~\ref{prop-part:reduced_mainprop_holonomy}. Let $h_1: \{s=s_0\} \dashrightarrow \{s=s_{2}\}$ and $h_2:\{s=s_2\} \dashrightarrow \{s=0\}$ be the individual holonomy maps through the regions of $U$ containing $C\Pi_1$ and $\rotreflPi_2$, respectively, so that $h:\partial_+ U \dashrightarrow \partial_- U$ is given by $h= h_2 \circ h_1$. By Lemma \ref{lemma:for_free}, $\|h_2(x)\|_{\mathrm{stab}} \leq e^{\frac{s_2}{2}}\|x\|_{\mathrm{stab}}$, and so $\|h(x)\|_{\mathrm{stab}} \leq e^{\frac{s_2}{2}}\|h_1(x)\|_{\mathrm{stab}}$ whenever $h_1(x)$ is in the domain of $h_2$. Thus it suffices to prove that if $h_1(x)$ is in the domain of $h_2$, then $\|h_1(x)\|_{\mathrm{stab}} \leq e^{\frac{s_1}{2}}\|x\|_{\mathrm{stab}}$. (Indeed, this gives $\|h(x)\|_{\mathrm{stab}} \leq e^{\frac{s_1 + s_2}{2}}\|x\|_{\mathrm{stab}}= Ke^{\frac{s_0}{2}}\|x\|_{\mathrm{stab}}$.) This is trivial for points outside the support of $C\Pi_1$, so it remains to consider points inside the support of $C\Pi_1$. Furthermore, by construction, any flowline exiting $C\Pi_1$ in the stove is trapped by $\rotreflPi_2$, so we only need to consider points exiting $C\Pi_1$ above the stove, necessarily in the chimney region. 

    Recall that $H_1^{C_1}$ was defined in a slightly modified way from that of Section \ref{section:chimney_folds}, so that the top of $H_1^{C_1}$ is contained in $\{t=t_0\}$. As a consequence, the holonomy through $C\Pi_1$ for points that never pass through the stove comes for free --- that is, it satisfies $\|h_1(x)\|_{\mathrm{stab}} \leq e^{\frac{s_1}{2}}\|x\|_{\mathrm{stab}}$ --- by Lemma \ref{lemma:for_free}. Therefore, we are further reduced to only considering points $x$ that pass through the stove and exit above the stove. Furthermore, since the projection of $C_1$ to $r_0\D^2$ is a disk centered around the origin, the only such points that could plausibly violate the desired radial estimate are those points $x$ entering with $W_1(x)\in C_1$.

    There are two cases to consider: 
    \begin{enumerate}
        \item The point $x$ enters the stove and exits above the stove. In this case~\ref{chimney:stove} of Proposition \ref{prop:new_chimney_smooth} implies that $W_1(x) \in C_1 \setminus C_1^{\epsilon_1}.$

        \item The point $x$ enters above the stove, passes through the stove, and exits above the stove. For this to happen it must be the case that $t(x) < t_0 - \epsilon$ by \ref{property:high_hol3b} of Theorem ~\ref{theorem:new_smooth_box_fold}; see also the reasoning below in the proof of holonomy property~\ref{prop-part:reduced_mainprop_t_holonomy}. But by part~\ref{chimney:trapping} of Proposition ~\ref{prop:new_chimney_smooth}, we must have $W_1(x) \in C_1 \setminus C_1^{\epsilon}$.
    \end{enumerate}

    Thus, in either case we have $W_1(x) \in C_1 \setminus C_1^{\epsilon_1}$ and $W_1(h(x)) \in C_1$. This gives 
        \[
        \norm{h_1(x)}_{\mathrm{stab}} \leq \frac{\rho_1}{\rho_1-\epsilon}\norm{x}_{\mathrm{stab}}.
        \]
        For $\epsilon$ sufficiently small, this will satisfy the desired estimate $\|h_1(x)\|_{\mathrm{stab}} \leq e^{\frac{s_1}{2}}\|x\|_{\mathrm{stab}}$. This completes the proof of~\ref{prop-part:reduced_mainprop_holonomy}.

    \vspace{2mm}
    \textit{Holonomy property~\ref{prop-part:reduced_mainprop_t_holonomy}.}
    \vspace{2mm}

    First we address the $t$-holonomy on the interval $[t_0 - \epsilon, \delta]$. Observe that, in a slight abuse of notation, $h[t_0 - \epsilon, t_0] = h_1[t_0 - \epsilon, t_0]$, since the supporting handlebody of $\rotreflPi_2$ is contained below $t=t_0 - 2\epsilon$. Moreover, by the design choice of $H_1^{C_1}$, namely the fact that the top of the chimney region is flat, the holonomy induced by the chimney fold $C\Pi_1$ for points entering in $[t_0 - \epsilon, t_0]$ is identical to the holonomy induced by a box fold based over $[0,t_0] \times \{r \leq \rho_1\}$ with smoothing parameter $\epsilon$. Thus, by part~\ref{property:high_hol3b} of Theorem ~\ref{theorem:new_smooth_box_fold}, we have 
    \[
    h_1[t_0 - \epsilon, t_0] \subset [t_0 - \epsilon, t_0].
    \]
    The argument is similar for $[0,\epsilon]$. Note that the support of the chimney fold is contained in $\{t> \epsilon\}$, so $h[0,\epsilon] = h_2[0,\epsilon]$. Since $\rotreflPi_2$ is an ordinary box hole with smoothing parameter $\epsilon$, again by the box hole version of part~\ref{property:high_hol3b} of Theorem \ref{theorem:new_smooth_box_fold} we have $h_2[0,\epsilon] \subset [0,\epsilon]$ as desired. This completes the proof of~\ref{prop-part:reduced_mainprop_t_holonomy} and the proof of the proposition. See also Figure \ref{fig:BAadj3}.
    
\end{proof}

\begin{figure}[ht]
	\begin{overpic}[scale=0.35]{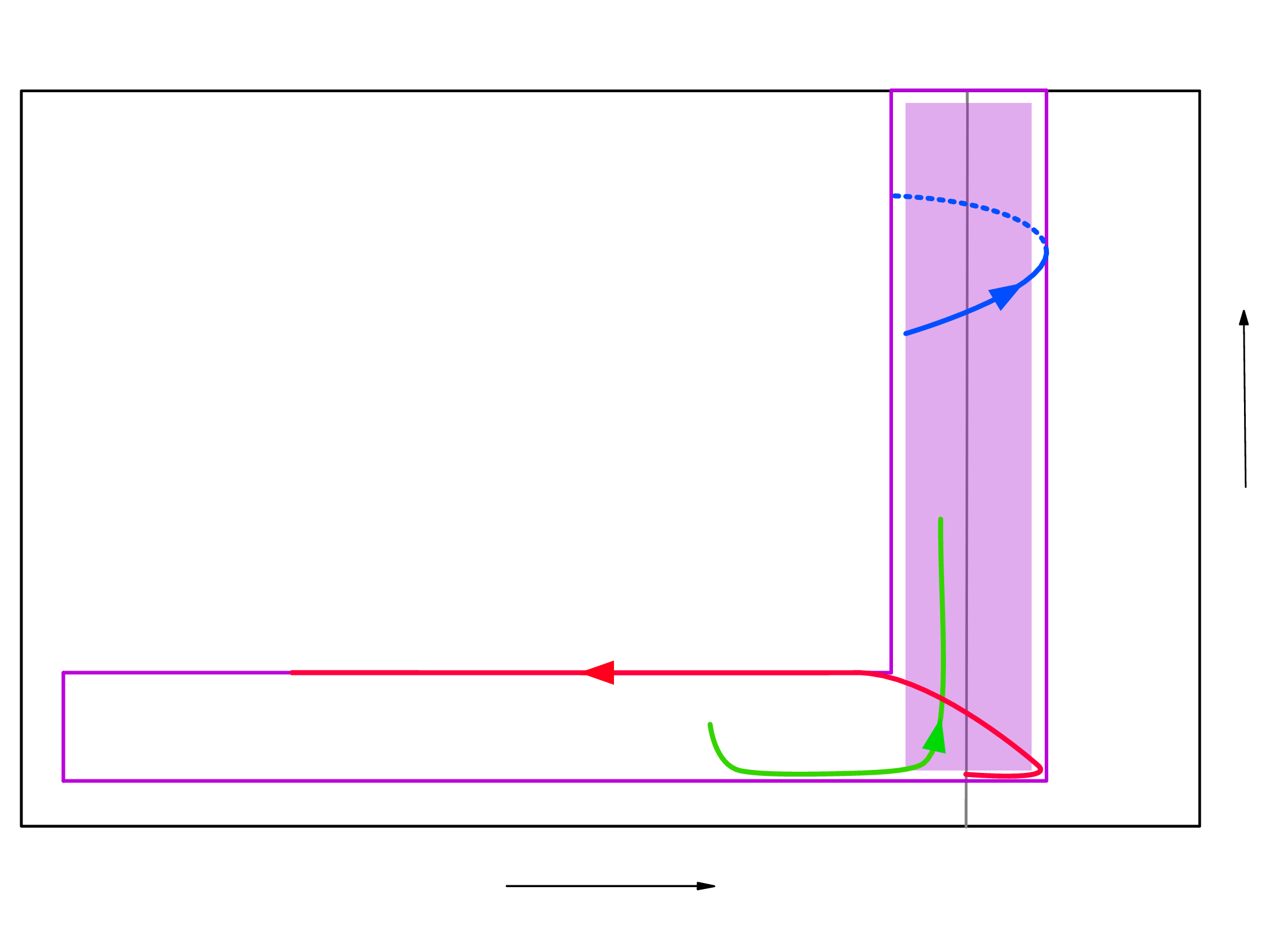}

     \put(97.5,51.5){\small $t$}
     \put(57,4.5){\small $q$}

	\end{overpic}
	\caption{Sample trajectories of points with holonomy in $C\Pi_1$. The shaded region indicates the trapping region of the chimney. The green point enters the fold in the stove and exits somewhere in the chimney. The red point enters along the $t$-axis near $t=\delta_1$, is heavily influenced by the characteristic foliation of $\partial H_1^{C_1}$, and exits far away in the stove. By itself, this violates the desired properties of Proposition \ref{prop:reduced_mainprop}, but $\rotreflPi_2$ will trap such a point. The blue point enters near $\partial C_1^{\epsilon_1}$, is influenced by the characteristic foliation of $\partial H_1^{C_1}$, and exits near $\partial C_1$. This is compatible with the statement of the proposition, provided $\epsilon$ is small enough.}
	\label{fig:BAadj3}
\end{figure}

\subsection{Proof of Theorem \ref{prop:mainprop}}\label{subsec:real_proof}

Finally we extend the definition of the blocking apparatus to higher dimensions and use it to prove Theorem~\ref{prop:mainprop}. Given Proposition~\ref{prop:reduced_mainprop}, this is completely straightforward as we are only concerned about trapping and holonomy behavior sufficiently far from the boundary of the additional $W_0$-direction. 

Consider a cobordism of the form 
\[
\left(U = [0,s_0] \times [0,t_0] \times W_0 \times r_0\D^2, \, e^s\, (dt + \lambda_0 +\lambda_{\mathrm{stab}})\right)
\]
where $(W_0, \lambda_0)$ is a Weinstein domain of arbitrary dimension. Let $H_1^{C_1}, H_2\subset [0,t_0]\times r_0 \D^2$ be constructed exactly as in Subsection~\ref{subsec:low_dim_BA_new}. Let $C\Pi_1$ be a high-dimensional chimney fold as defined in Section~\ref{section:chimney_folds} over $H_1^{C_1}\times W_0$ with symplectization length $s_1$ and smoothing parameter $\epsilon$, and let $\rotreflPi_2$ be a box hole installed over $H_2 \times W_0$ with symplectization length $s_2$ and smoothing parameter $\epsilon$. 

The definition of a blocking apparatus installed on the cobordism $U$ is the same as before, with the only difference being $C\Pi_1$ and $\rotreflPi_2$ are high-dimensional folds. 

\begin{definition}
A \textbf{blocking apparatus (with smoothing parameter $\epsilon$)} consists of a chimney fold $C\Pi_1$ with smoothing parameter $\epsilon$ installed over $[s_0-s_1, s_0] \times H_1^{C_1}$ and a box hole with smoothing parameter $\epsilon$ installed over $[0,s_2]\times H_2$. 
\end{definition}

\begin{proof}[Proof of Theorem \ref{prop:mainprop}.]
As with Proposition \ref{prop:reduced_mainprop}, part~\ref{mainprop:weinstein} of Theorem \ref{prop:mainprop} follows immediately from part~\ref{property:high_Weinstein_compat} of Theorem \ref{theorem:new_smooth_box_fold} and part~\ref{chimney:weinstein} of Proposition \ref{prop:new_chimney_smooth}.

Given the descriptions of the trapping regions for high-dimensional chimney folds (Proposition \ref{prop:new_chimney_smooth}) and box holes (the negative version of Theorem \ref{theorem:new_smooth_box_fold}), the argument for identifying the trapping region in Theorem \ref{prop:mainprop} is the same as the argument in Proposition \ref{prop:reduced_mainprop}.

Finally, the holonomy properties in~\ref{mainprop:holonomy} are immediate consequences of the holonomy properties in Proposition \ref{prop:reduced_mainprop} together with the $W_0$-norm estimates~\ref{thm-part:new_smooth_holonomy} of Theorem \ref{theorem:new_smooth_box_fold} and Proposition \ref{prop:new_chimney_smooth}. In particular, Proposition ~\ref{prop:new_chimney_smooth} bounds the expansion of the $W_0$-norm through the chimney fold by a factor of $e^{s_1}$ and Theorem ~\ref{theorem:new_smooth_box_fold} bounds the expansion of the $W_0$-norm through the box hole is bounded by a factor of $e^{s_2}$. These estimates then give~\ref{mainprop:holonomy:weinstein} of Theorem \ref{prop:mainprop} with $K= e^{s_1 + s_2 - s_0} < 1$. Moreover, this means that a flowline entering the apparatus outside of $N^{s_0}(\partial W_0)$ is prevented from ever interacting with the decay of the fold near $\partial W_0$ and hence this additional direction has no effect to the holonomy. The holonomy properties of Proposition ~\ref{prop:reduced_mainprop} then give the remaining desired properties in Theorem ~\ref{prop:mainprop}. 
\end{proof}

\section{Some technical properties of the blocking apparatus}\label{section:technical_properties}

In this section, we state a few more technical properties of the blocking apparatus. Their utility will become clear in Section \ref{section:mitsumatsu}, in particular for the purpose of ruling out the existence of broken loops. First, we introduce some terminology to identify certain critical points in a blocking apparatus. 

\begin{definition}
Let $C\Pi_1$ be the chimney fold of a blocking apparatus. A \textbf{stove critical point} of $C\Pi_1$ is a critical point in the stove region of $C\Pi_1$. 
\end{definition}

\begin{proposition}\label{prop:no-chimney-points-in-broken-loops}
A stove critical point in a blocking apparatus cannot be included in a broken loop.
\end{proposition}

\begin{proof}
We make the trivial observation that any critical point of index $0$ cannot be included in a broken loop. Thus, it suffices to prove the following. If $\gamma$ is a broken flowline involving a stove critical point, then every broken flowline containing $\gamma$ is further contained in a broken flowline that involves a critical point of index $0$. Indeed, if this is the case, then there are no broken loops containing stove critical points.

First, note that the dynamics of a chimney fold are Morse-Smale, and thus there are no broken loops contained entirely inside $C\Pi_1$. It therefore suffices to consider a broken flowline that contains a stove critical point and exits the stove region in backward time. 

In this case, consider a low-dimensional blocking apparatus installed over $[0,s_0] \times [0,t_0] \times r_0 \D^2$, with no $W_0$-component. In this setting, the main trapping region of the box hole $\rotreflPi_2$ is an $\{s=\frac{s_0}{2}\}$ cross section of the unstable manifold of a critical point of index $0$. By design of a blocking apparatus, the entire stove region of the chimney fold $C\Pi_1$ is a subset of this $\rotreflPi_2$ trapping region. In particular, any broken flowline emanating from a stove critical point in backward time and exiting the stove, if extended sufficiently far, will intersect this trapping region and ultimately converge to a critical point of index $0$.

In the high-dimensional setting --- when the blocking apparatus is installed with a $W_0$ component --- the proposition follows almost immediately from the low-dimensional case, because critical points of the blocking apparatus project to critical points of $X_{\lambda_0}$ in $W_0$. In particular, any sufficiently extended broken flowline emanating from the stove region via a stove critical point in backward time is ultimately engulfed by the main trapping region of $\rotreflPi_2$, hence by a critical point with $0$ index in the $[0,t_0] \times r_0 \D^2$ projection. In the $W_0$ projection, because $W_0$ is a Weinstein domain, $X_{\lambda_0}$ has Morse dynamics, and every broken flowline descends to a critical point with $0$ index in $W_0$. As such, every broken flowline then eventually reaches a critical point of total index $0$, as desired. 
\end{proof}

\begin{corollary}\label{cor:no-broken-loops-in-BA}
After applying Theorem~\ref{prop:mainprop} to a Weinstein cobordism 
\[
(U=[0,s_0]\times[0,t_0]\times W_0\times r_0\,\mathbb{D}^2,e^s\,(dt+\lambda_0+\lambda_{\mathrm{stab}}))
\]
there are no broken loops contained in $U$.
\end{corollary}

\begin{proof}
By the previous proposition, no stove critical point can be involved in any broken loop, let alone one contained in $U$. Thus, a broken loop contained in $U$ must be contained either entirely in $\rotreflPi_2$, or must involve one of the critical points corresponding to the top of the chimney. By definition of a box hole --- which requires Morse-Smale dynamics --- there are no broken loops entirely contained in $\rotreflPi_2$. Finally, none of the critical points near the top of the chimney interact with any critical point in $\rotreflPi_2$, so these critical points can not be involved in a broken loop contained in $U$. Thus, there are no broken loops contained in $U$. 
\end{proof}

The ultimate efficacy of Proposition \ref{prop:no-chimney-points-in-broken-loops} is due to the following observation about the main trapping mechanism of a blocking apparatus.

\begin{proposition}\label{prop:main-prop-goes-to-chimney-points}
After applying Theorem~\ref{prop:mainprop}, any flowline which passes through $\{s=s_0\}\times U_{\mathrm{trap}}\subseteq\partial_+U$ either converges in backward time to a stove critical point, or to the index $0$ critical point of the box hole $\rotreflPi_2$.
\end{proposition}

\begin{proof}
This is immediate from the design and definition of a chimney fold in Section \ref{section:chimney_folds}. In particular, in a low-dimensional piecewise-linear chimney fold, any flowline entering the chimney region is trapped in the stove region. The extension to higher dimensions does not impact this behavior, and $U_{\mathrm{trap}}$, by definition, is a set approximating the interior of the chimney region. 
\end{proof}

\begin{remark}
We remind the reader that the trapping neighborhood $\{s=s_0\}\times U_{\mathrm{trap}}\subseteq\partial_+U$ in the statements of Theorem~\ref{prop:mainprop} and Proposition \ref{prop:main-prop-goes-to-chimney-points} is not the \textit{entire} trapping region of the blocking apparatus. There are a number of other codimension $0$ regions that are trapped by other critical points of the apparatus. 
\end{remark}

Finally, we have the following proposition which concerns the behavior of flowlines entering a blocking apparatus in \textit{forward} time. Note that the perspective of this proposition is drastically different than the backward time analysis that permeates the majority of the previous sections.

\begin{proposition}\label{prop:points-trapped-in-forward-time}
After applying Theorem~\ref{prop:mainprop} to a Weinstein cobordism $(U=[0,s_0]\times[0,t_0]\times W_0\times r_0\,\mathbb{D}^2,e^s\,(dt+\lambda_0+\lambda_{\mathrm{stab}}))$ with parameter $\epsilon$, consider a point $(0,t^*,x,p,q)\in\partial_-U$.  If the flowline through this point converges to a critical point of the blocking apparatus in forward time, then $x\in\mathrm{Skel}(W_0)$, $(p,q)=(0,0)$, and $t^*\in[0,\epsilon) \cup (t_0 - \epsilon, t_0]$.
\end{proposition}

We first require a lemma concerning the forward time trapping behavior of a normal box fold.

\begin{lemma}\label{lemma:forward_time_box_fold_lemma}
Let $\Pi$ be a smooth box fold installed over $([0,s_0] \times [0,t_0] \times W_0, \, e^s\, (dt + \lambda_0))$. Let $X$ denote the perturbed Liouville vector field, and suppose that $(0,t^*,x)\in \{s=0\} \times [0,t_0] \times W_0$ is a point such that the flowline of $X$ through $(0,t^*,x)$ converges to a critical point in forward time. Then $x \in \mathrm{Skel}(W_0, \lambda_0)$ and $t^* >  t_0 - \epsilon$.
\end{lemma}

\begin{proof}
Let $F_{\epsilon}$ be the smooth box fold function. Recall that the Liouville vector field after installing a box fold is, away from $\partial W_0$,
\[
   X=  \partial_s + e^{-s}\, X_{F_{\epsilon}}^{ds\, dt} - e^{-s}\frac{\partial F_{\epsilon}}{\partial t}\, X_{\lambda_0}.
    \]
Rewriting this gives 
\[
X = \left(1 + e^{-s} \frac{\partial F_{\epsilon}}{\partial t}\right)\, \partial_s - e^{-s}\frac{\partial F_{\epsilon}}{\partial s} \, \partial_t - e^{-s}\frac{\partial F_{\epsilon}}{\partial t}\, X_{\lambda_0}.
\]
Similarly, identify the collar neighborhood of $\partial W_0$ as $([-\epsilon,0]_r \times \partial W_0,\, \lambda_0 = e^r\, \eta_0)$ where $\eta_0 := \lambda_0 \mid_{\partial W_0}$. Since $d_{\partial W_0}F_{\epsilon} = 0$, the Liouville vector field here is
    \[
    X = \left(1 + e^{-s} \frac{\partial F_{\epsilon}}{\partial t}\right)\, \partial_s - e^{-s}\frac{\partial F_{\epsilon}}{\partial s} \, \partial_t - e^{-s}\frac{\partial F_{\epsilon}}{\partial t} \, X_{\lambda_0} - e^{-s} \frac{\partial F_{\epsilon}}{\partial r}\,\left(- \partial_t + e^{-r}\, R_{\eta_0} \right).
    \]
We may also assume without loss of generality that $\mathrm{Skel}(W_0, \lambda_0) \cap ([-\epsilon,0]_r \times \partial W_0) = \emptyset$. This assumption implies that 
\[
\left((\pi_{W_0})_*X\right)\mid_{\mathrm{Skel}(W_0, \lambda_0)} = - e^{-s}\frac{\partial F_{\epsilon}}{\partial t} \, X_{\lambda_0}.
\]
In words, the $W_0$-projection of $X$ to (a neighborhood of) the skeleton of $(W_0, \lambda_0)$ is parallel to $X_{\lambda_0}$ throughout the entire box fold. This implies that if $x \notin \mathrm{Skel}(W_0, \lambda_0)$, then the flowline of $X$ through $(0,t^*,x)$ never traverses $\mathrm{Skel}(W_0, \lambda_0)$.

Next, observe that if $p$ is a critical point of $X$, then $\frac{\partial F_{\epsilon}}{\partial t}(p) < 0$. Thus, in a neighborhood of $p$, $(\pi_{W_0})_*X$ is a positive multiple of $X_{\lambda_0}$. In particular, $\pi_{W_0}(p)$ is a critical point of $X_{\lambda_0}$. Moreover, this implies that if the flowline of $X$ through $(0,t^*,x)$ converges to a critical point in forward time, then flowline must eventually reach $\mathrm{Skel}(W_0, \lambda_0)$. By the above remark, it follows that $x \in \mathrm{Skel}(W_0, \lambda_0)$. Finally, the fact that $t^* > t_0 - \epsilon$ is immediate from the observation that in a piecewise linear box fold, every flowline passing through $\{s=0\}\times (0,t_0) \times W_0$ travels through the fold and reaches $\{s=s_0\}$ in forward time. This behavior is approximated by a smooth box fold, and hence $t^* > t_0 - \epsilon$. This proves the lemma. 
\end{proof}

\begin{proof}[Proof of Proposition \ref{prop:points-trapped-in-forward-time}.]

We claim that such a flowline $(0,t^*, x,p,q)$ is not trapped by a stove critical point. That is, the critical point must either be in $\rotreflPi_2$, the box hole, or the critical point at the top of the chimney of $C\Pi_1$. Assuming this, Proposition \ref{prop:points-trapped-in-forward-time} follows immediately from the previous lemma: indeed, $\rotreflPi_2$ is an ordinary fold, and near the top of the chimney of $C\Pi_1$ the chimney fold is locally an ordinary fold. In particular, the skeleton of the stabilization direction component of each of these folds is the origin and the $t$-coordinates are arbitrarily close to $t=0$ and $t=t_0$, respectively.

Thus, it suffices to show that the flowline through $(0,t^*, x,p,q)$ cannot converge to a stove critical point in forward time. But this is immediate from the design of a blocking apparatus: in \textit{backward} time, every flowline emanating from a stove critical point enters the trapping region of $\rotreflPi_2$; see the proof of Proposition \ref{prop:no-chimney-points-in-broken-loops}. This implies that all flowlines through $\{s=0\}$ avoid stove critical points of $C\Pi_1$ in forward time. 

\end{proof}

\section{Torus bundle domains are stably Weinstein}\label{section:mitsumatsu}
In this section we use Theorem~\ref{prop:mainprop} to prove Theorem~\ref{thm:torus-bundle-domains}, which says that the torus bundle Liouville domains of \cite{huang2019dynamical} are stably Weinstein.

We start by restating Theorem~\ref{prop:mainprop} for \emph{standard stabilized regions}.  These are regions $U\subset (W,\lambda)$ in a Liouville domain that have the form
\[
(U,\lambda|_U)\cong([0,s_0]\times[0,t_0]\times W_0\times r_0\,\mathbb{D}^2,e^s\,(dt+\lambda_0)+\lambda_{\mathrm{stab}}),
\]
for some Weinstein domain $(W_0,\lambda_0)$.  That is, $(U,\lambda|_U)$ appears to be a stabilization of a symplectization of a contact handlebody.  These regions appear naturally in the stabilization of a Liouville domain, and thus a version of Theorem~\ref{prop:mainprop} stated for such regions will be useful to us in this section.

After restating Theorem~\ref{prop:mainprop} in this manner, we will recall the definition of torus bundle Liouville domains, identify the standard stabilized regions on which we want to perturb the Liouville form, and finally use Proposition~\ref{prop:weinstein-criterion} to verify that this perturbed domain is in fact Weinstein.

\subsection{The local operation for standard stabilized regions}
Theorem~\ref{prop:mainprop} perturbs the Liouville form on the symplectization of a stabilized contact handlebody.  When modifying the Liouville dynamics of a stabilized Liouville domain, we will often find it more convenient for the stabilization and symplectization to be ``decoupled."  Precisely, we are interested in regions which appear to be stabilizations of symplectizations (rather than symplectizations of stabilizations). 

\begin{definition}
Given a Liouville domain $(W,\lambda)$, a \emph{standard stabilized region} is a subset $U\subset(W,\lambda)$ such that there exists a diffeomorphism
\[
\varphi\colon [0,s_0]_s\times[0,t_0]_t\times W_0\times r_0\,\mathbb{D}^2 \to U,
\]
with $\varphi^*(\lambda|_U) = e^s\,(dt+\lambda_0)+\lambda_{\mathrm{stab}}$, for some Weinstein domain $(W_0,\lambda_0)$ and some choice of constants $s_0,t_0,r_0>0$.  If $U\subset(W,\lambda)$ is a standard stabilized region, we will typically write
\[
(U,\lambda|_U) \cong ([0,s_0]_s\times[0,t_0]_t\times W_0\times r_0\,\mathbb{D}^2,e^s\,(dt+\lambda_0)+\lambda_{\mathrm{stab}}),
\]
making no mention of the diffeomorphism $\varphi$.
\end{definition}

Notice that the Liouville vector field in a standard stabilized region is given by $\partial_s+\tfrac{1}{2}(p\,\partial_p+q\,\partial_q)$; with a careful choice of coordinates, we may identify a subregion of a standard stabilized region to which Theorem~\ref{prop:mainprop} applies.  The upshot is that some version of Theorem~\ref{prop:mainprop} holds for standard stabilized regions, with minor modifications to the holonomy statements.

\begin{corollary}\label{cor:main-local-operation}
Consider the Weinstein cobordism $(U=[0,s_0]\times[0,t_0]\times W_0\times r_0\,\mathbb{D}^2,e^s\,(dt+\lambda_0)+\lambda_{\mathrm{stab}})$.  Fix $0<\epsilon\ll \min(1,s_0,\frac{1}{2}(1-e^{-s_0})t_0)$ sufficiently small. If $e^{-s_0/2}r_0>0$ is sufficiently large, a blocking apparatus can be installed in $U$, i.e., a Liouville homotopy $\lambda_U \rightsquigarrow \lambda_{\mathrm{BA}}$ can be applied to $U$, such that the following properties hold. 
\begin{enumerate}
    \item (Weinstein compatibility)

    \noindent The Liouville vector field $X_{\lambda_{\mathrm{BA}}}$ is Morse with 
    $8N_0$ critical points, where $N_0$ is the number of critical points of $(W_0, \lambda_0)$. Of these $8N_0$ critical points, $N_0^{\mathrm{crit}}$ of them are critical points of middle index, where $N_0^{\mathrm{crit}}$ is the number of middle index critical points of $(W_0, \lambda_0)$. 

    \item (Trapping properties)

\noindent There is a neighborhood $U_{\mathrm{trap}}$ of 
\[
[\epsilon,t_0-\epsilon]\times I_\epsilon(W_0,\lambda_0)\times \{(0,0)\} \subset [0,t_0] \times W_0 \times r_0\,\mathbb{D}^2
\]
such that any flowline passing through $\{s=s_0\}\times U_{\mathrm{trap}}\subseteq \partial_+U$ converges to a critical point of $X_{\lambda_{\mathrm{BA}}}$ in backward time.

    \item (Holonomy properties)

\noindent Let $x\in \partial_+U$ be in the domain of the partially-defined holonomy map $h\colon\partial_+U\dashrightarrow\partial_-U$ induced by $X_{\lambda_{\mathrm{BA}}}$. Then the following properties hold. 
\begin{enumerate}
    \item For some constant $0<K<1$, we have $\|W_0(h(x))\|_{W_0}\leq Ke^{s_0}\,\|W_0(x)\|_{W_0}$.\label{cor:local-operation:constant-part}

    \item For the same constant $0<K<1$, whenever $\|x\|_{W_0}<e^{-s_0}$ we have
    \[
    \|\pi_{r_0\,\mathbb{D}^2}(h(x))\|_{\mathrm{stab}}\leq K\,\|\pi_{r_0\,\mathbb{D}^2}(x)\|_{\mathrm{stab}},
    \]
    where $\|\cdot\|_{\mathrm{stab}}$ is the usual Euclidean norm on $r_0\D^2$.\label{cor:local-operation:stabilization-part}
    
    \item If $t(x) \in [0,\epsilon] \cup [t_0 - \epsilon,t_0]$ and $W_0(x)\in I_{1 - e^{-s_0}}(W_0, \lambda_0)$, then $t(h(x)) \in [0,\epsilon] \cup [t_0 - \epsilon,t_0]$.\label{cor:local-operation:reeb-part}
    
\end{enumerate}
\end{enumerate}
\end{corollary}

\begin{proof}[Proof of Corollary~\ref{cor:main-local-operation}]
The corollary follows quickly from Theorem~\ref{prop:mainprop} once we identify a \emph{standard subregion} in $(U,\lambda|_U)$.  In particular, consider the subset
\[
\overline{U} := \{(s,t,w,p,q)\in U\,|\, e^{(s_0-s)/2}\,\|(p,q)\|\leq r_0\} \subset U.
\]
We can give $\overline{U}$ coordinates via a map
\[
\psi\colon [0,s_0]_{\overline{s}}\times [0,t_0]_{\overline{t}} \times W_0 \times e^{-s_0/2}\,r_0\,\mathbb{D}^2 \to \overline{U}\subset U
\]
defined by
\[
\psi(\overline{s},\overline{t},\overline{w},\overline{p},\overline{q}) := (\overline{s},\overline{t},\overline{w},e^{\overline{s}/2}\,\overline{p},e^{\overline{s}/2}\,\overline{q}).
\]
Notice that $\psi^*(\lambda|_{\overline{U}}) = e^{\overline{s}}\,(d\overline{t}+\lambda_0+\lambda_{\mathrm{stab}})$, and thus that $(\overline{U},\lambda_{\overline{U}})$ is a Weinstein cobordism of the type hypothesized in Theorem~\ref{prop:mainprop}.  If the stabilization radius $e^{-s_0/2}\,r_0$ is sufficiently large, then Theorem~\ref{prop:mainprop} gives us the conclusions of Corollary~\ref{cor:main-local-operation}.

In particular, part~\ref{prop:mainprop:stabilization-part} of Theorem~\ref{prop:mainprop} tells us that if $\|\overline{x}\|_{W_0}<e^{-s_0}$, then $\|h(\overline{x})\|_{\mathrm{stab}}\leq \overline{K}\,\|\overline{x}\|_{\mathrm{stab}}$, for a constant $0<\overline{K}<e^{s_0/2}$.  This inequality holds in the coordinates on $\overline{U}$ given by $\psi$.  Because $\overline{x}$ is an element of $\partial_+U$ and $h(\overline{x})$ is an element of $\partial_-U$, we have
\[
\|h(\overline{x})\|_{\mathrm{stab}} = \|\psi(h(\overline{x}))\|_{\mathrm{stab}}
\quad\text{and}\quad
\|\overline{x}\|_{\mathrm{stab}} = e^{-s_0/2}\,\|\psi(\overline{x})\|_{\mathrm{stab}},
\]
so
\[
\|\psi(h(\overline{x}))\|_{\mathrm{stab}} = \|h(\overline{x})\|_{\mathrm{stab}} \leq \overline{K}\,\|\overline{x}\|_{\mathrm{stab}} = \overline{K}\,e^{-s_0/2}\,\|\psi(\overline{x})\|_{\mathrm{stab}},
\]
giving us part~\ref{cor:local-operation:stabilization-part} of Corollary~\ref{cor:main-local-operation}.
\end{proof}

\subsection{The domains}\label{subsection:huangs-domains}
In this subsection we recall Huang's construction of the torus bundle Liouville domains, establishing detailed notation we will need later.

The construction begins with a matrix $A\in\mathrm{SL}(n,\mathbb{Z})$ which has the property that its eigenvalues $\lambda_1,\ldots,\lambda_n$ are all real, and satisfy
\[
0 < \lambda_n < |\lambda_i|,
\quad\text{for all~} 1\leq i\leq n-1.
\]
Notice that $A$ represents a linear Anosov map $T^n\to T^n$, where $T^n:=\mathbb{R}^n/\mathbb{Z}^n$.  Matrices of this sort were shown to exist in \cite[Appendix A]{huang2019dynamical}.

Given such a matrix $A$, we may choose linearly independent 1-forms $\beta_1,\ldots,\beta_n\in\Omega^1(T^n)$ satisfying
\[
A^*\beta_i = \lambda_i\beta_i,
\quad\text{for all~} 1\leq i\leq n.
\]
From these we define a contact form
\[
\alpha := \beta_n + \sum_{i=1}^{n-1}y_i\,\beta_i
\]
on $M=D^{n-1}\times T^n$, where the coordinates on $D^{n-1}$ are given by $y_1,\ldots,y_{n-1}$.  Notice that the map $\phi_A\colon M\to M$ defined by
\[
\phi_A(\vec{y},\vec{x}) := (\Lambda\vec{y},A\vec{x}),
\]
where
\[
\Lambda=\left(\begin{matrix}
    \lambda_n/\lambda_1 & 0 & \cdots & 0\\
    0 & \lambda_n/\lambda_2 & \cdots & 0\\
    \vdots & \vdots & \ddots & \vdots\\
    0 & 0 & \cdots & \lambda_n/\lambda_{n-1}
\end{matrix}\right),
\]
is a diffeomorphism onto its image, and that $\phi_A^*\alpha=\lambda_n\alpha$.  In the language of \cite{huang2019dynamical}, this makes $\phi_A$ a \emph{contraction} of $(M,\alpha)$ and allows us to define a Liouville domain $W_A$ which is a partial mapping torus of $\phi_A$.

We begin with the symplectization $(\mathbb{R}_s\times M,d(e^s\alpha))$ of $M$ and consider the map $\Phi_A\colon\mathbb{R}\times M\to\mathbb{R}\times M$ defined by
\[
\Phi_A(s,\vec{y},\vec{x}) := (s-\ln(\lambda_n),\Lambda\vec{y},A\vec{x}).
\]
Then
\[
\Phi_A^*(e^s\alpha) = e^{s-\ln(\lambda_n)}\phi_A^*\alpha = \frac{1}{\lambda_n}e^s(\lambda_n\alpha) = e^s\alpha,
\]
so the Liouville form $e^s\alpha$ on $\mathbb{R}\times M$ descends to the partial mapping torus
\[
W_A := ([0,|\ln(\lambda_n)|]_s\times M)/(0,\vec{y},\vec{x}) \sim \Phi_A(0,\vec{y},\vec{x}).
\]
Of course, the vector field dual to $\lambda:=e^s\,\alpha$ via $d(e^s\,\alpha)$ is $\partial_{s}$, which does not point out of the \emph{vertical boundary} $[0,|\ln(\lambda_n)|]\times\partial M$ of $W_A$.  This is easily fixed with a small perturbation of the vertical boundary, and the precise choice of perturbation does not affect the Liouville homotopy class of the resulting Liouville domain $(W_A,\lambda)$.  See Figure~\ref{fig:torus-bundle-domain-schematic}.

It is clear from this construction that the skeleton of $(W_A,\lambda)$ is given by
\[
\mathrm{Skel}(W_A,\lambda) = ([0,|\ln(\lambda_n)|]\times\{0\}\times T^n)/(0,0,\vec{x})\sim(-\ln(\lambda_n),0,A\vec{x}),
\]
which is smoothly the mapping torus of $A\colon T^n\to T^n$.  Our $2n$-dimensional Liouville domain $(W_A,\lambda)$ thus has the homotopy type of an $n+1$-manifold, and therefore fails to admit a Weinstein structure.  Our goal is to show that $(W_A,\lambda)$ is made Weinstein by a single stabilization --- i.e., that $(W_A\times r_0\,\mathbb{D}^2,\lambda+\tfrac{1}{2}(p\,dq-q\,dp))$ is Weinstein, up to homotopy, for some $r_0>0$.

\begin{figure}
    \centering
    \begin{overpic}[width=0.4\linewidth]{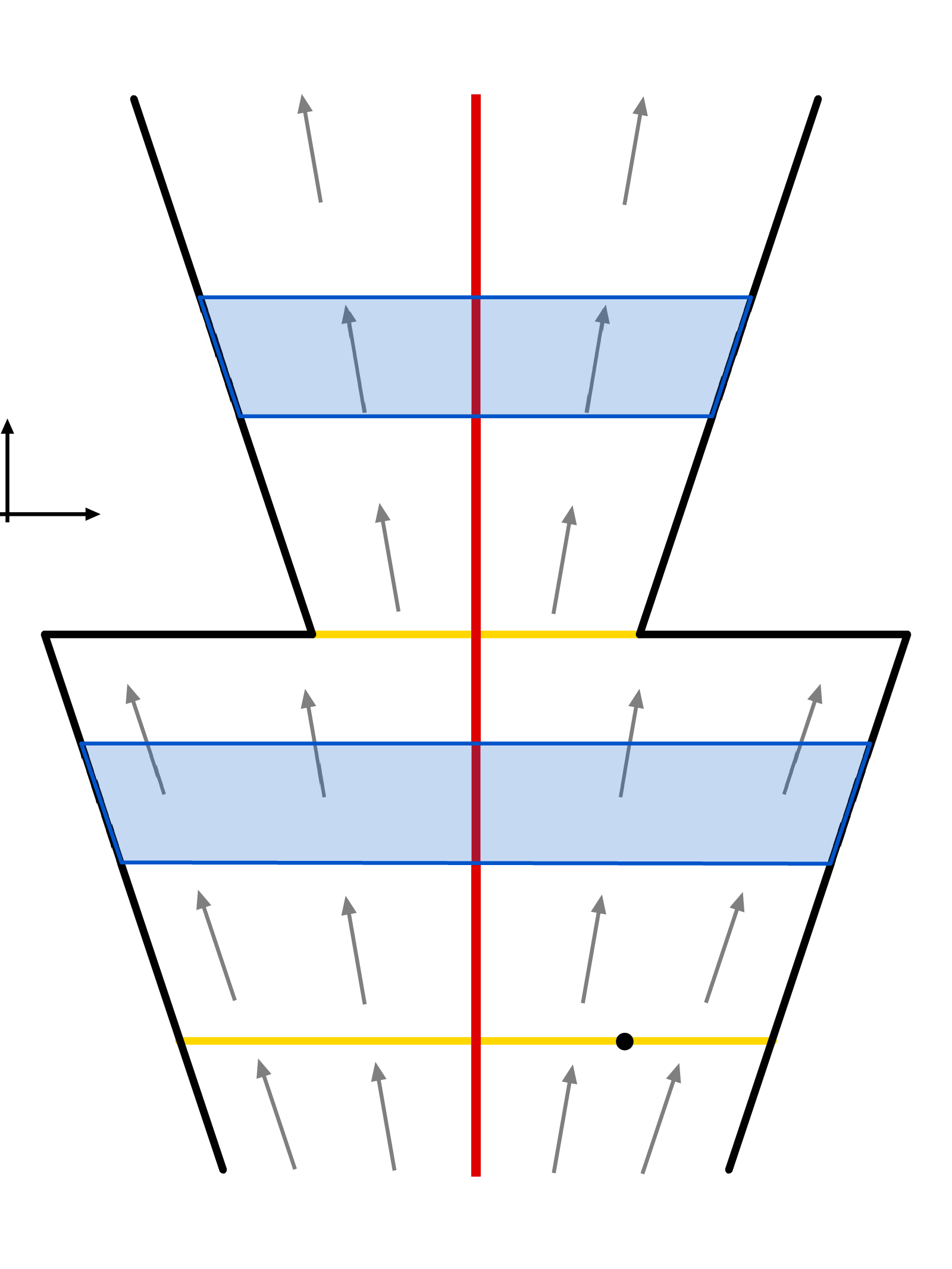}
        \put(28,2){\textcolor{Red}{$\mathrm{Skel}(W_A, \lambda)$}}
        \put(70,35){\textcolor{Blue}{$V_0$}}
        \put(60,70){\textcolor{Blue}{$V_1$}}
        \put(53,52){\textcolor{Goldenrod}{$M_0$}}
        \put(63,17){\textcolor{Goldenrod}{$M_s$}}
        \put(-0.25,67.5){\small $s$}
        \put(8.5,58.5){\small $D^{n-1}$}
        \put(46.5,20.5){\small $S_{s,\vec{y}}$}
    \end{overpic}
    \caption{A schematic for the torus bundle Liouville domain $(W_A,\lambda)$, along with its Liouville vector field.  Each point in this schematic represents a torus $T^{n}$, with the Reeb vector field of each contact slice $M_s$ parallel to this torus.  We caution that while, as depicted here, the regions $V_0$, $V_1$ of $W_A$ have thickness in the symplectization and $D^{n-1}$-directions, they include only strips $S_0,S_1\subsetneq T^{n}$ over each point.}
    \label{fig:torus-bundle-domain-schematic}
\end{figure}

\subsection{Identifying standard stabilized regions}\label{subsection:mitsumatsu-ssr}
As outlined in Section~\ref{section:introduction}, our strategy for simplifying Liouville dynamics focuses on using local perturbations to interrupt the Liouville flow along the skeleton of our domain.  When identifying the standard stabilized regions that will support our perturbations (i.e., where Corollary~\ref{cor:main-local-operation} will be applied), it is thus important to consider how these regions intersect our skeleton.

Because our skeleton has no interesting behavior in the stabilization direction, our search for standard stabilized regions will be guided only by the nature of the initial domain $(W_A,\lambda)$ and its skeleton.  This search is eased considerably by the fact that $(W_A,\lambda)$ agrees locally with the symplectization of $(M,\alpha)$.  This allows us to construct a standard stabilized region by identifying a contact handlebody $([0,t_0]\times W_0,dt+\lambda_0)$ in $(M,\alpha)$ and then considering the standard stabilized region
\[
(\sigma_0\times[0,t_0]\times W_0\times r_0\mathbb{D}^2,e^s(dt+\lambda_0)+\tfrac{1}{2}(p\,dq-q\,dp)) \subset W_A,
\]
for some closed interval $\sigma_0\subset(0,|\ln(\lambda_n)|)$ and some $r_0>0$.  The contact handlebody arises as a standard neighborhood of a closed Legendrian in $(M,\alpha)$, chosen so that this neighborhood meets $\mathrm{Skel}(W_A,\lambda)$ in a desirable manner.

\subsubsection{Identifying contact handlebodies}
Given any $s\in[0,|\ln(\lambda_n)|)$ and $\vec{y}\in D^{n-1}$, we may consider the submanifold
\[
S_{s,\vec{y}} := \{s\}\times\{\vec{y}\}\times T^n
\]
of the contact slice
\[
M_s:=\{s\}\times D^{n-1} \times T^n \subset (W_A,\lambda),
\quad
0\leq s < |\ln(\lambda_n)|.
\]
We denote by $\alpha_s:=e^s(\beta_n + \sum_i y_i\,\beta_i)$ the contact form on $M_s$, and by $\eta_{s,\vec{y}}$ the restriction of $\alpha_s$ to $S_{s,\vec{y}}$.  We thus have a distribution
\[
\ker(\eta_{s,\vec{y}}) = \mathbb{R}\langle y_1 V_n-V_1,\ldots, y_{n-1}V_n - V_{n-1} \rangle
\]
on $S_{s,\vec{y}}$, integrable by virtue of the fact that $d\eta_{s,\vec{y}}\equiv 0$.  Here $V_1,\ldots,V_n\in\mathfrak{X}(T^n)$ are eigenvectors of $A$ corresponding to the eigenvalues $\lambda_1,\ldots,\lambda_n$.

Observe that the intersection of $\mathrm{Skel}(W_A,\lambda)$ with $M_s$ is given by $S_{s,0}$.  While $\ker(\eta_{s,0})$ produces a foliation of $S_{s,0}$ by Legendrian submanifolds, the leaves of this foliation are not closed.  However, the collection of values $\vec{y}$ for which the foliation determined by $\ker(\eta_{s,\vec{y}})$ has closed leaves is dense in $D^{n-1}$, and thus we may choose $\vec{y}_0\neq 0$ arbitrarily close to 0 so that $S_{s,\vec{y}_0}$ is foliated by Legendrian tori.  With our choice of $\vec{y}_0\neq 0$ fixed\footnote{In practice, the particular value of $\vec{y}_0$ is unimportant.}, we will obtain our contact handlebodies as standard neighborhoods of Legendrian leaves.

Let us denote by $L_0$ the leaf which passes through $0$ in $T^n\subset S_{s,\vec{y}_0}$ and consider a strip $S_0\subset T^n$ centered on $L_0$.  That is, $S_0$ has coordinates $(t,\theta_1,\ldots,\theta_{n-1})$, with $\partial_t=V_n$ and $\partial_{\theta_i}=y_i\,V_n-V_i$, for $1\leq i\leq n-1$.  We write $S_0=[0,t_0]\times L_0$, with the original leaf $L_0$ given by $\{t_0/2\}\times L_0$.  The strip $S_0$ is chosen so that its volume is at least 2/3 in the flat metric on $T^n=\mathbb{R}^n/\mathbb{Z}^n$.  Now $N_0:=\{s\}\times D^{n-1}\times S_0\subset (M_s,\alpha_s)$ is a contact handlebody; in the natural coordinates $(\vec{y},t,\vec{\theta})$ on this neighborhood, we have
\[
\alpha|_{N_0} = e^s\left(dt+\sum_{i=1}^{n-1}(\vec{y}_0-\vec{y})_i\,d\theta_i\right).
\]
At last we may choose a closed interval $\sigma_0\subset(0,|\ln(\lambda_n)|)$ and define
\[
V_0:= \sigma_0 \times D^{n-1} \times [0,t_0]_t \times L_0 \subset W_A,
\]
so that $\lambda|_{V_0}=e^s\,(dt+\sum_i(\vec{y}_0-\vec{y})_i\,d\theta_i)$.  We emphasize that $(V_0,\lambda|_{V_0})$ is a region in the \emph{unstabilized} domain $(W_A,\lambda)$.  It remains to choose a radius $r_0>0$ and define a standard stabilized region $U_0:=V_0\times r_0\,\mathbb{D}^2$.  Before doing so, we point out that a second region, $(V_1,\lambda|_{V_1})$, may be constructed by letting $L_1\subset T^n$ be the unique leaf of $S_{s,\vec{y}_0}$ with the property that the two components of $T^n\setminus(L_0\cup L_1)$ are thickened tori of equal volume.  We then construct a strip $S_1=[0,t_0]\times L_1\subset T^n$ of the same Reeb thickness $t_0$ as $S_0$, with coordinates $(t,\vec{\theta})$.  At last we define
\[
V_1:= \sigma_1 \times D^{n-1} \times [0,t_0]_t \times L_1 \subset W_A,
\]
where $\sigma_1$ is an as-yet-undetermined closed subinterval of $(0,\min(\sigma_0))$.  In particular, the intervals $\sigma_0,\sigma_1$ are taken to be disjoint.

\subsubsection{Choosing parameters for Corollary~\ref{cor:main-local-operation}}
At this stage, we are treating the quantities $t_0>0$ and $\vec{y}_0\neq 0$ as fixed; it remains to choose the intervals $\sigma_0$ and $\sigma_1$, as well as the radius $r_0>0$ that will be used for our standard stabilized regions.  Moreover, applying Corollary~\ref{cor:main-local-operation} to $U_i:=V_i\times r_0\,\mathbb{D}^2$, $i=0,1$, will require the choice of a parameter $\epsilon>0$.  Our plan is to carefully choose $\epsilon$, then the intervals $\sigma_0$ and $\sigma_1$, and finally $r_0$.

For any $0<\epsilon\ll t_0$, let us denote by $S_i^\epsilon$ the strip
\[
S_i^\epsilon := [\epsilon,t_0-\epsilon]\times L_i \subset S_i \subset T^n,
\]
for $i=0,1$.  We will choose $\epsilon>0$ sufficiently small to ensure that $S_0^{2\epsilon}\cup S_1^{2\epsilon}=T^n$.  This accounts for the ``Reeb stretching" that is induced by Corollary~\ref{cor:main-local-operation}.  Next, we consider the codimension 2 domains underlying $V_0$ and $V_1$.  That is, we have
\[
(W_i,\lambda_i) = (D^{n-1}\times L_i,(\vec{y}_0-\vec{y})_i\,d\theta_i),
\]
the skeleton of which is given by $\{\vec{y}_0\}\times L_i$.  We require $\epsilon>0$ to be sufficiently small to ensure that the region $\{\vec{y}\in D^{n-1}\,|\,|\vec{y}|\leq|\vec{y}_0|\}\times L_i$ is contained in the interior of $I_\epsilon(W_i,\lambda_i)$.  This is because the intersection of $\mathrm{Skel}(W_A,\lambda)$ with $V_i$ is given by $\sigma_i\times\{0\}\times[0,t_0]\times L_i$; by choosing $\epsilon>0$ sufficiently small, we ensure that this intersection is trapped when Corollary~\ref{cor:main-local-operation} is applied.  See Figure~\ref{fig:mitsumatsu-band}.

\begin{figure}
    \centering
    \begin{tikzpicture}[scale=1.25]
\fill [green,even odd rule,opacity=0.25] (0,0) circle[radius=1.45cm] circle[radius=0.85cm];

\draw (0,0) circle (0.5 cm);
\draw[red] (0,0) circle (1 cm);
\draw[green] (0,0) circle (1.15 cm);
\draw (0,0) circle (1.5 cm);

\end{tikzpicture}
    \caption{The Weinstein domain $(W_i,\lambda_i)$ has skeleton $\{\vec{y}_0\}\times L_i$, represented by the green circle, while the intersection of $\mathrm{Skel}(W_A,\lambda)$ with $V_i$ projects to $W_i$ as $\{0\}\times L_i$, given by the red circle.  We choose $\epsilon>0$ sufficiently small to ensure that $I_\epsilon(W_i,\lambda_i)$ --- depicted here by the region shaded in green --- contains the region $\{|\vec{y}|\leq|\vec{y}_0|\}\times L_i$.}
    \label{fig:mitsumatsu-band}
\end{figure}

With $\epsilon>0$ chosen, we now set about defining the intervals $\sigma_0$ and $\sigma_1$.  We have already insisted that $\sigma_1$ come strictly ``before" $\sigma_0$, in the sense that $\max(\sigma_1)<\min(\sigma_0)$.  We now place some restrictions on the lengths of these intervals; beyond their lengths, order, and disjointness, the precise choice for these intervals is unimportant.  Our first restriction on the lengths is simple: we require that $|\sigma_i|<-\ln(1-\epsilon)$, for $i=0,1$.  This restriction allows us to conclude that $N^{|\sigma_i|}(\partial W_i)$ is a strict subset of $W_i\setminus I_\epsilon(W_i,\lambda_i)$.  So $I_\epsilon(W_i,\lambda_i)$ is contained in $W_i\setminus N^{|\sigma_i|}(\partial W_i)$, and thus part~\ref{cor:local-operation:reeb-part} of Corollary~\ref{cor:main-local-operation} (regarding the Reeb holonomy) applies to any point whose $W_i$-component is contained in $I_\epsilon(W_i,\lambda_i)$.

The second restriction on the lengths of $\sigma_0$ and $\sigma_1$ is quite a bit more complicated to state.  We begin with the observation that the Liouville flow $\psi^s\colon W_i\to W_i$ is given by
\[
\psi^s(\vec{y},\vec{\theta}) = (\vec{y}_0 + e^s\,(\vec{y}-\vec{y}_0),\vec{\theta}),
\]
for $i=0,1$.  We could use this to explicitly compute the norm $\|\cdot\|_{W_i}$, but the important point is that this norm depends only on $\vec{y}$.  Now Corollary~\ref{cor:main-local-operation} considers the partially-defined holonomy map $h\colon\partial_+U \dashrightarrow \partial_-U$ as consisting more-or-less of three distinct components: a component in the stabilization direction, another in the Weinstein direction, and a third in the Reeb direction.  Consider the following three maps $h_\bullet\colon W_i\to W_i$:
\[
h_i(\vec{y},\vec{\theta}) := (\vec{y}+e^{|\sigma_i|}\,(\vec{y}-\vec{y}_0),\vec{\theta}),~~i=0,1,
\quad
h_g(\vec{y},\vec{\theta}) := (\Lambda\,\vec{y},\vec{\theta}).
\]
According to Corollary~\ref{cor:main-local-operation}, $h_0$ and $h_1$ represent a sort of worst-case Weinstein $\vec{y}$-holonomy for points which are not trapped (ignoring any changes which might happen in the $\vec{\theta}$-component).  The final map, $h_g$, comes from the global holonomy of $(W_A,\lambda)$.  We observe that $h_0$ and $h_1$ have an unstable fixed point at $(\vec{y}_0,\vec{\theta})$, while $h_g$ brings the $\vec{y}$-component closer to 0, since all eigenvalues of $\Lambda$ are less than 1.  By our choice of $\epsilon>0$, the region in $W_i$ where we have $|\vec{y}|\leq|\vec{y}_0|$ is contained in the interior of $I_\epsilon(W_i,\lambda_i)$, and thus we can choose $|\sigma_0|$ and $|\sigma_1|$ sufficiently small to ensure that the image of $(h_g\circ h_1\circ h_0)^k$ is contained in $I_\epsilon(W_i,\lambda_i)$, for some $k\geq 1$.  That is, $\sigma_0$ and $\sigma_1$ are sufficiently short that the global holonomy will dominate in the $\vec{y}$-component.  This constitutes our second requirement on $|\sigma_0|$ and $|\sigma_1|$.

Our third and final requirement on $|\sigma_0|$ and $|\sigma_1|$ is that these quantities be chosen sufficiently small to ensure that $(h_0\circ h_g\circ h_1)(\vec{y},\vec{\theta})$ is contained in the interior of $I_\epsilon(W_i,\lambda_i)$, for any $\vec{y}$ satisfying $|\vec{y}|\leq|\vec{y}_0|$.  This condition will be used when verifying that our perturbed Liouville domain contains no broken loops.

Finally, we choose $r_0>0$ sufficiently large, given our choice of $\epsilon$, so that the conclusions of Corollary~\ref{cor:main-local-operation} hold for both $(U_0,(\lambda+\tfrac{1}{2}(p\,dq-q\,dp))|_{U_0})$ and $(U_1,(\lambda+\tfrac{1}{2}(p\,dq-q\,dp))|_{U_1})$.

\begin{remark}
It is worth pointing out one small wrinkle with our standard stabilized regions: the domains $(W_i,\lambda_i)$ are not technically Weinstein domains.  However, these particular domains can be made Weinstein via arbitrarily small perturbations supported in arbitrarily small neighborhoods of their skeleta.  The argument given below does not concern itself with the precise dynamics of $(W_i,\lambda_i)$ near the skeleton, and thus we may treat $(W_i,\lambda_i)$ has having been made Weinstein.
\end{remark}

\subsection{Verifying the Weinstein criteria}\label{subsection:mitsumatsu-verifying}
Let $(W_A\times r_0\,\mathbb{D}^2,\tilde{\lambda})$ denote the Liouville domain which results from applying Corollary~\ref{cor:main-local-operation} to the disjoint regions $U_0$ and $U_1$ identified above.  We now verify that every flowline of this domain limits to a critical point in backward time, and that this domain has no broken loops.  According to Proposition~\ref{prop:weinstein-criterion}, this will allow us to conclude that $(W_A\times r_0\,\mathbb{D}^2,\tilde{\lambda})$ is a Weinstein domain.

\subsubsection{The critical points criterion}
Notice that if a flowline of the perturbed Liouville domain fails to limit to a critical point in backward time, then this flowline must pass through the slice $M_{|\ln(\lambda_n)|}\times r_0\,\mathbb{D}^2$ of $(W_A\times r_0\,\mathbb{D}^2,\tilde{\lambda})$.  The same is true before perturbation (in which case our domain has no critical points), and thus we consider a point $(|\ln(\lambda_n)|,\vec{y},\vec{x},p,q)$ in $W_A\times r_0\,\mathbb{D}^2$.  We let $\ell$ denote the flowline through this point before perturbation, while $\tilde{\ell}$ is the flowline through the point after perturbation.  Our goal is to show that $\tilde{\ell}$ must converge to a critical point of our modified Liouville domain in backward time.

We first observe that, since the strips $S_0$ and $S_1$ in $T^n$ on which $U_0$ and $U_1$ are modeled cover $T^n$, the flowline $\ell$ must intersect at least one of $\partial_+U_0$ and $\partial_+U_1$ in backward time.  Ideally, the perturbations installed on $U_0$ and $U_1$ will cause the new flowline $\tilde{\ell}$ to limit to a critical point in one of these regions without returning to the slice $M_{|\ln(\lambda_n)|}\times r_0\,\mathbb{D}^2$, but there are several ways in which this might fail to occur.

\subsubsection*{The $\vec{y}$ holonomy}
First, if $\vec{y}$ is near the boundary $\partial D^{n-1}$, then our perturbations cannot be expected to trap $\tilde{\ell}$.  However, the global holonomy of $W_A$ ensures that $\vec{y}$ cannot stay near $\partial D^{n-1}$.  In particular, we have chosen the lengths of the intervals $\sigma_0$ and $\sigma_1$ so that the global holonomy, which tends to shrink the modulus of $\vec{y}$, dominates the holonomy due to our perturbations on $U_0$ and $U_1$.  Eventually, if $\tilde{\ell}$ does not first limit to a critical point, we will have $\vec{y}$ sufficiently near $\vec{y}_0$ so that $\tilde{\ell}$ intersects either $\partial_+U_0$ or $\partial_+U_1$ in $\{\max(\sigma_i)\}\times[0,t_i]\times I_\epsilon(W_i,\lambda_i)$.  For this reason we may assume that $\vec{y}$ is near $\vec{y}_0$ --- in particular, $\|\tilde{\ell}\cap\{s=|\ln(\lambda_n)|\}\|_{W_0} \leq e^{-|\sigma_0|}(1-\epsilon)<1-\epsilon$.

\subsubsection*{The stabilization holonomy}
According to the first condition placed on $\sigma_0$ and $\sigma_1$, we have $1-\epsilon<e^{-|\sigma_i|}$ for $i=0,1$, and thus
\[
\|\tilde{\ell}\cap\{s=|\ln(\lambda_n)|\}\|_{W_0} < e^{-|\sigma_i|}
\]
According to part~\ref{cor:local-operation:stabilization-part} of Corollary~\ref{cor:main-local-operation}, we have constants $0<K_0,K_1<1$ such that if $\tilde{\ell}$ is not trapped as it passes through $U_i$, then the stabilization norms of its intersections with the ends of $U_i$ are related by
\[
\|\tilde{\ell}\cap\partial_-U_i\|_{\mathrm{stab}} \leq K_i\|\tilde{\ell}\cap\partial_+U_i\|_{\mathrm{stab}}.
\]
That is, the holonomy associated to $U_i$ brings untrapped flowlines closer to $W_A\times\{(0,0)\}$.  We may therefore restrict our attention to the case where $(p,q)=(0,0)$; if all flowlines through such points limit to critical points in backward time, then in fact all flowlines of $(W\times r_0\,\mathbb{D}^2,\tilde{\lambda})$ limit to critical points in backward time.

\subsubsection*{The Reeb holonomy}
We are now considering a flowline $\tilde{\ell}$ which passes through a point $(|\ln(\lambda_n)|,\vec{y},\vec{x},0,0)$, with $\vec{y}$ sufficiently close to $\vec{y}_0$ to ensure that the unperturbed flowline $\ell$ intersects at least one of $\partial_+U_0$, $\partial_+U_1$ in its trapping region.  If $\vec{x}\in T^n$ lies in $S_0^\epsilon$, then $\ell$ meets the trapping region of $\partial_+U_0$ and thus $\tilde{\ell}$ converges to a critical point of $U_0$ in backward time.  If $\vec{x}$ lies outside of $S_0$, then $\tilde{\ell}$ will not intersect $U_0$ and $\vec{x}\in S_1^\epsilon$, so $\tilde{\ell}$ will converge to a critical point of $U_1$ in backward time.  Finally, if $\vec{x}\in S_0-S_0^\epsilon$, then $\tilde{\ell}$ will either converge to a critical point of $U_0$ in backward time or, according to part~\ref{cor:local-operation:reeb-part} of Corollary~\ref{cor:main-local-operation}, intersect $\partial_-U_0$ with $T^n$-coordinates lying in $S_0-S_0^{2\epsilon}\subset S_1^\epsilon$.  In the latter case, $\tilde{\ell}$ will intersect $\partial_+U_1$ in its trapping region, and in either case $\tilde{\ell}$ converges to a critical point in backward time.\\

We conclude that every flowline of $(W\times r_0\,\mathbb{D}^2,\tilde{\lambda})$ limits to a critical point in backward time.

\subsubsection{The broken loops criterion}
We now show that if $(W\times r_0\,\mathbb{D}^2,\tilde{\lambda})$ contains a broken loop $c$, then $c$ includes a stove critical point, in violation of Proposition~\ref{prop:no-chimney-points-in-broken-loops}.  Throughout this subsection, we suppose that $c$ is a broken loop of $(W\times r_0\,\mathbb{D}^2,\tilde{\lambda})$.  We begin with a simple observation.

\begin{lemma}\label{lemma:exit-ui}
The broken loop $c$ includes a flowline $\gamma$ which converges to a critical point of $U_i$ in forward time and exits $U_i$ in backward time, where $i=0$ or $i=1$.
\end{lemma}
\begin{proof}
Because every flowline of $(W\times r_0\,\mathbb{D}^2,\tilde{\lambda})$ converges to a critical point in backward time, $c$ must include a critical point --- that is, $c$ is genuinely broken.  The critical points of $(W\times r_0\,\mathbb{D}^2,\tilde{\lambda})$ are contained in $U_0$ and $U_1$, so $c$ must include a critical point of $U_i$, for $i=0$ or $i=1$.  But Corollary~\ref{cor:no-broken-loops-in-BA} says that neither $U_0$ nor $U_1$ contains a broken loop, so $c$ must exit $U_i$ in backward time.  In particular, $c$ includes a flowline $\gamma$ which converges in forward time to a critical point of $U_i$ (and thus this critical point does not have index 0) and exits $U_i$ in backward time.
\end{proof}

If $\gamma$ exits $U_0$ in backward time, we quickly obtain our stove critical point.

\begin{lemma}\label{lemma:exit-u0}
If $\gamma$ is a flowline which converges to a critical point of $U_0$ in forward time and exits $U_0$ in backward time, then $\gamma$ converges in backward time to a stove critical point of $U_1$.
\end{lemma}
\begin{proof}
According to Proposition~\ref{prop:points-trapped-in-forward-time}, $\gamma$ must intersect $\partial_-U_0$ with coordinates $(\min(\sigma_0),\vec{y}_0,\vec{x},0,0)$, for some point $\vec{x}\in T^n$ which lies in $S_0-S_0^{2\epsilon}$.  But $S_0-S_0^{2\epsilon}$ is a subset of $S_1^\epsilon\subset T^n$, so $\gamma$ meets $\partial_+U_1$ in its trapping region.  According to Proposition~\ref{prop:main-prop-goes-to-chimney-points}, $\gamma$ converges in backward time to a critical point of $U_1$ which either has index 0 or is a stove critical point.  But the backward-time limit of $\gamma$ is the forward-time limit of some other flowline of $c$, and thus does not have index 0.  So $\gamma$ converges in backward time to a stove critical point of $U_1$.
\end{proof}

An immediate consequence of Lemma~\ref{lemma:exit-u0} is that if $c$ includes any critical points of $U_0$, then $c$ includes a stove critical point --- simply move backwards along $c$ until reaching a flowline which exits $U_0$ in backward time, and this flowline will converge to a stove critical point of $U_1$.  For this reason, we may assume that all critical points of $c$ --- of which there is at least one --- are contained in $U_1$.\\

We now let $\gamma$ be a flowline of $c$ which converges in both forward and backward time to critical points of $U_1$.  According to Proposition~\ref{prop:points-trapped-in-forward-time}, $\gamma$ will intersect $\partial_-U_1$ in a point of the form $(\min(\sigma_1),\vec{y}_0,\vec{x},0,0)$, and thus intersect $M_0\times r_0\mathbb{D}^2$ at $(0,\vec{y}_0,\vec{x},0,0)$ under backward flow.  Applying the global holonomy, $\gamma_i$ reaches the point $(|\ln(\lambda_n)|,\Lambda\vec{y}_0,A\vec{x},0,0)$ in $M_{\ln(\lambda_s)}\times r_0\mathbb{D}^2$.  Now $A\vec{x}\in T^n$ must not lie in $S_0^\epsilon$, lest $\gamma$ converge in backward time to a critical point of $U_0$.  It is possible that $A\vec{x}$ lies in $S_0-S_0^\epsilon$, and thus that the $D^{n-1}$- and $T^n$-coordinates of $\gamma$ are affected by the holonomy of $U_0$; however, Corollary~\ref{cor:main-local-operation} ensures that $\gamma$ will intersect $\partial_-U_0$ with $D^{n-1}$-coordinate satisfying $|\vec{y}|\leq|\vec{y}_0|$, and with $T^n$-coordinate in $S_1^\epsilon$, and thus $\gamma$ intersects $\partial_+U_1$ in its trapping region.  In particular, Proposition~\ref{prop:main-prop-goes-to-chimney-points} tells us that $\gamma$ limits in backward time to a stove critical point.\\

In any case, we now see that if $(W\times r_0\,\mathbb{D}^2,\tilde{\lambda})$ contains a broken loop, then this broken loop includes a stove critical point.  Because this violates Proposition~\ref{prop:no-chimney-points-in-broken-loops}, we conclude that $(W\times r_0\,\mathbb{D}^2,\tilde{\lambda})$ contains no broken loops, and thus is a Weinstein domain.

\bibliographystyle{alpha}
\bibliography{references}

\end{document}